\documentclass[10pt,reqno]{amsart}
\usepackage[a4paper,centering]{geometry}
\usepackage[utf8]{inputenc}
\inputencoding{utf8} 
\usepackage{newpxtext,newpxmath}
\usepackage[british]{babel}
\usepackage{hyperref}
\usepackage{amssymb}
\usepackage{mathcomp}
\usepackage{relsize}
\usepackage{amsmath}
\usepackage{xfrac}
\usepackage{mathrsfs}
\usepackage{color}
\usepackage{mathtools,amssymb}
\usepackage{MnSymbol,accents}
\usepackage[spacing]{microtype}
\usepackage{enumitem}
\setlist{leftmargin=*}
\usepackage{etoolbox}
\usepackage{nicefrac}
\newtoggle{arxiv}
\toggletrue{arxiv}
\DeclarePairedDelimiter{\group}{(}{)}
\DeclarePairedDelimiter{\sqgroup}{[}{]}
\DeclarePairedDelimiter{\set}{\{}{\}}

\newcommand{\reals}{\mathbb{R}}
\newcommand{\extreals}{\overline{\mathbb{R}}}
\newcommand{\posreals}{\reals_{>}}
\newcommand{\nnegreals}{\reals_{\geq}}
\newcommand{\nats}{\mathbb{N}}
\newcommand{\natz}{\mathbb{N}_{0}}

\newcommand{\indica}[1]{\mathbb{I}_{#1}}

\newcommand{\upprev}{\overline{\mathrm{E}}}
\newcommand{\upprevvovkk}{\overline{\mathrm{E}}_\mathrm{V}}
\newcommand{\lowprev}{\underline{\mathrm{E}}}
\newcommand{\lowprevvovkk}{\underline{\mathrm{E}}_\mathrm{V}}

\newcommand{\lupprev}[1]{\overline{\mathrm{Q}}_{#1}}

\newcommand{\sit}{x_{1:n}}
\newcommand{\situa}[1]{x_{#1}}
\newcommand{\situation}[2]{x_{#1:#2}}

\newcommand{\martingale}{\mathscr{M}}

\newcommand{\setofextsupmartb}{\overline{\mathbb{M}}_\mathrm{b}}

\newcommand{\process}{\mathscr{P}}

\newcommand{\situations}{\mathscr{X}^\ast}
\newcommand{\statespace}{\mathscr{X}}
\newcommand{\statespaceseq}[2]{\mathscr{X}_{#1:#2}}

\newcommand{\samplespace}{\Omega}

\newcommand{\setofgengambles}{\mathscr{L}}
\newcommand{\setofgenextvariables}{\overline{\mathscr{L}}}
\newcommand{\setofgenextvariablesb}{\overline{\mathscr{L}}_b}

\newcommand{\setofgambles}{\mathbb{V}}
\newcommand{\setofextvariables}{\overline{\mathbb{V}}}
\newcommand{\setofextvariablesb}{\overline{\mathbb{V}}_b}

\newcommand{\setoflimitsoffinmeas}{\mathbb{V}_{\mathrm{lim}}}
\newcommand{\setoflimitsoffinmeasext}{\overline{\mathbb{V}}_{\mathrm{lim}}}
\newcommand{\setoflimitsoffinmeasb}{\overline{\mathbb{V}}_{b,\mathrm{lim}}}

\newcommand{\prodindica}[2]{#1\,\indica{#2}}

\newcommand{\andstate}{\,\cdot}

{\itshape}{\rmfamily}

\newtheorem{definition}{Definition}{\bfseries}{\rmfamily}
\newtheorem{theorem}{Theorem}{\bfseries}
\newtheorem{proposition}[theorem]{Proposition}{\bfseries}
\newtheorem{corollary}[theorem]{Corollary}{\bfseries}
\newtheorem{lemma}[theorem]{Lemma}{\bfseries}

\begin{document}

\title{Continuity Properties of Game-theoretic Upper Expectations}
% \subtitle{} % optional
\author{Natan T'Joens \and Jasper De Bock \and Gert de Cooman}
\address{Electronics and Information Systems, Ghent University, Belgium}
\email{\{natan.tjoens,jasper.debock,gert.decooman\}@ugent.be}
\maketitle
\openup5pt

\section{Introduction}
Discrete-time uncertain processes, including discrete-time Markov processes, can be described mathematically in various ways.
For many, measure theory is the preferred framework for describing the uncertain dynamics of such processes.
We consider the alternative game-theoretic framework that was developed by Shafer and Vovk \cite{Shafer:2005wx}.
In particular, we study the mathematical properties of the upper expectations that appear in this theory.
Some of these, including Doob's Convergence Theorem and L\'evy's Zero-one Law, have already been proved by Shafer and Vovk.
For these results, our contribution consists in adapting them to our setting.
However, we also present several results that, to the best of our knowledge, appear here for the first time.
Important examples are continuity with respect to non-decreasing sequences and continuity with respect to specific sequences of so-called `finitary' functions, which depend only on a finite number of states.

We start by introducing upper (and lower) expectations on extended real-valued functions that are bounded below and explore their relation to the well-known concept of coherence.
These upper (and lower) expectations will then be used to model the dynamic behaviour of uncertain processes on a local level, allowing us to introduce the concept of a supermartingale; a particular kind of capital process that depends on the realisations of the process of interest.
Subsequently, to model the dynamic behaviour of uncertain processes on a global level, we introduce a particular game-theoretic upper expectation operator that fundamentally relies on this concept of a supermartingale. 
The remainder of the paper provides an overview, where we prove various mathematical properties for this game-theoretic upper expectation.
This article serves as a technical reference for an upcoming paper \cite{Tjoens2019NaturalExtensionISIPTA} that describes an alternative characterisation of this game-theoretic upper expectation as the most conservative uncertainty model satisfying a set of intuitive axioms.
For this reason, we omit from the present text any elaborate discussion of the interpretation and consequences of our results.

\section{Upper Expectations}\label{Sect: Upper Expectations}
We start this section with the introduction of a number of preliminary notions.
We denote the set of all natural numbers, without $0$, by $\nats$, and let $\natz \coloneqq\nats \cup \{0\}$. 
The set of extended real numbers is denoted by $\extreals \coloneqq\reals \cup \{+\infty, -\infty\}$.  
The set of positive real numbers is denoted by $\posreals$ and the set of non-negative real numbers by $\nnegreals$.
We extend the partial order relation $\leq$ on $\reals{}$ to $\extreals$ by positing that $-\infty<c<+\infty$ for all $c\in\reals{}$.

For any non-empty set $\mathscr{Y}$, a \emph{variable} $f$ on $\mathscr{Y}$ is a map on $\mathscr{Y}$.
We then define a \emph{gamble} on $\mathscr{Y}$ as a bounded real(-valued) variable on $\mathscr{Y}$ and an \emph{extended real variable} on $\mathscr{Y}$ as a variable on $\mathscr{Y}$ taking values in $\extreals$.
We say that an extended real variable $f$ is \emph{bounded below} if there is an $M\in\reals{}$ such that $f(y) \geq M$ for all $y\in\mathscr{Y}$ and bounded above if $-f$ is bounded below.
The set of all extended real variables on $\mathscr{Y}$ is denoted by $\setofgenextvariables(\mathscr{Y})$, the set of all bounded below extended real variables on $\mathscr{Y}$ by $\setofgenextvariablesb(\mathscr{Y})$ and the linear space of all gambles on $\mathscr{Y}$ by $\setofgengambles(\mathscr{Y})$.
For any $f\in\setofgenextvariables(\mathscr{Y})$ we use $\sup f$ and $\inf f$ to denote the supremum $\sup\{f(y) \colon y\in\mathscr{Y}\}$ and the infimum $\inf\{f(y) \colon y\in\mathscr{Y}\}$ of the variable $f$, respectively.
We say that a sequence $\{f_n\}_{n\in\natz}$ in $\setofgenextvariables(\mathscr{Y})$ is \emph{uniformly bounded below} if there is an $M\in\reals{}$ such that $f_n \geq M$ for all $n\in\natz$.
For any sequence $\{f_n\}_{n\in\natz}$ in $\setofgenextvariables(\mathscr{Y})$, we write $\lim_{n\to+\infty} f_n$ to mean the point-wise limit of the functions $f_n$, unless mentioned otherwise.
For a subset $A$ of $\mathscr{Y}$, we define the \emph{indicator} $\indica{A}$ of $A$ as the gamble on $\mathscr{Y}$ that assumes the value $1$ on $A$ and $0$ elsewhere.

To model uncertainty, we will use upper and lower expectations.
The following definition is similar to what Shafer and Vovk call an `outer probability content' \cite{shafer2011levy}.

\begin{definition}
Consider any non-empty set $\mathscr{Y}$.
Then we define an \emph{upper expectation} $\upprev$ on $\setofgenextvariablesb(\mathscr{Y})$ as an extended real-valued map on $\setofgenextvariablesb(\mathscr{Y})$ that satisfies the following four axioms:
\begin{enumerate}[leftmargin=*,ref={\upshape E\arabic*},label={\upshape E\arabic*}.,series=sepcoherence]
\item \label{coherence: const is const} 
$\upprev(c)=c$ for all $c\in\reals{}$; 
\item \label{coherence: sublinearity} 
$\upprev(f+g)\leq\upprev(f) + \upprev(g)$ for all $f,g\in\setofgenextvariablesb(\mathscr{Y})$;
\item \label{coherence: homog for ext lambda}
$\upprev(\lambda f)=\lambda \upprev(f)$ for all $\lambda\in\reals{}_{>} \cup\{+\infty\}$ and all non-negative $f\in\setofgenextvariablesb(\mathscr{Y})$.
\item \label{coherence: monotonicity}
if\/ $f \leq g$ then $\upprev(f)\leq\upprev(g)$ for all $f, g\in\setofgenextvariablesb(\mathscr{Y})$;
\end{enumerate}
\end{definition}
\noindent
Alternatively, we can also consider the so-called \emph{conjugate lower expectation}, defined by $\lowprev(f)\coloneqq-\upprev(-f)$ for all extended real variables $f$ on $\mathscr{Y}$ that are \emph{bounded above}. 
It clearly suffices to focus on only one of the two functionals and we will work mainly with upper expectations.

Note that in the definition above, as well as further on, we adopt the following conventions: $c +\infty=+\infty$, $+\infty+(+\infty)=+\infty$, $\lambda \cdot (+\infty)= +\infty$ and $0 \cdot (+\infty)=0$ for all real $c$ and all $\lambda\in\reals{}_{>} \cup\{+\infty\}$. 

As a consequence of their defining axioms, it can be shown that upper expectations satisfy various additional properties.

\begin{proposition}\label{prop: coherence properties}
For any non-empty set $\mathscr{Y}$ and any upper expectation $\upprev$ on $\setofgenextvariablesb(\mathscr{Y})$, we have that
\begin{enumerate}[leftmargin=*,ref={\upshape E\arabic*},label={\upshape E\arabic*}., resume=sepcoherence ]

\item \label{coherence: bounds}
$-\infty <\inf f\leq\upprev(f)\leq\sup f$ for all $f\in\setofgenextvariablesb(\mathscr{Y})$;
\item \label{coherence: const add}
$\upprev(f + \mu)=\upprev(f) + \mu$ for all $\mu\in\reals{} \cup \{+\infty\}$ and all $f\in\setofgenextvariablesb(\mathscr{Y})$;
\item \label{coherence: homog for real lambda}
$\upprev(\lambda f)=\lambda \upprev(f)$ for all $\lambda\in\nnegreals{}$ and all $f\in\setofgenextvariablesb(\mathscr{Y})$.
\item \label{coherence: low smaller than up} 
$\lowprev(f + g)\leq\upprev(f) + \lowprev(g)\leq\upprev(f+g)$ for all $f,g\in\setofgengambles(\mathscr{Y})$;
\item \label{coherence: uniform convergence}
if\/ $\lim_{n\to+\infty} \sup{\vert f-f_n \vert}=0$ then $\lim_{n\to+\infty} \vert \upprev(f)-\upprev(f_n) \vert=0$ and  $\lim_{n\to+\infty} \vert \lowprev(f)-\lowprev(f_n) \vert=0$ for any sequence $\{f_n\}_{n\in\natz}$ in $\setofgengambles(\mathscr{Y})$.
\end{enumerate}
\end{proposition}
\begin{proof}
\ref{coherence: bounds}:
Consider any $f\in\setofgenextvariablesb(\mathscr{Y})$.
If $\sup f=+\infty$, we trivially have that $\upprev(f)\leq\sup f$.
If $\sup f$ is real, it follows immediately from \ref{coherence: monotonicity} that $\upprev(f)\leq\upprev(\sup f)$ and therefore that $\upprev(f)\leq\sup f$ because of \ref{coherence: const is const}.
That $\sup f=-\infty$, is impossible because $f$ is bounded below.
To prove that $-\infty <\inf f\leq\upprev(f)$, note that $\inf f$ is real or equal to $+\infty$, because $f$ is bounded below.
Then for any real $\alpha <\inf f$ we clearly have that $\alpha < f$, implying by \ref{coherence: monotonicity} and \ref{coherence: const is const} that $\alpha < \upprev(f)$.
Since this holds for any $\alpha <\inf f$ we indeed have that $\inf f\leq\upprev(f)$.

\ref{coherence: const add}:
That $\upprev(f + \mu)\leq\upprev(f) + \mu$ for all real $\mu$ and all $f\in\setofgenextvariablesb(\mathscr{Y})$, follows directly from \ref{coherence: sublinearity} and \ref{coherence: const is const}.
The other inequality follows from the fact that  
\begin{align*}
\upprev(f)=\upprev(f +\mu-\mu) \overset{\text{\ref{coherence: sublinearity}}}{\leq} \upprev(f + \mu) + \upprev(- \mu) \overset{\text{\ref{coherence: const is const}}}{=} \upprev(f + \mu)-\mu,  
\end{align*}
 for all $f\in\setofgenextvariables(\mathscr{Y})$ and all real $\mu$.
 If $\mu=+\infty$, then since $\upprev(f) > -\infty$ [because of \ref{coherence: bounds}], we are left to show that $\upprev(+\infty)=+\infty$, which follows trivially from \ref{coherence: bounds}.

 \ref{coherence: homog for real lambda}:
 If $\lambda$ is positive, this follows trivially from \ref{coherence: homog for ext lambda} and \ref{coherence: const add}. 
 If $\lambda=0$, we have to prove that $\upprev(0)=0$, which follows immediately from \ref{coherence: const is const}.

 \ref{coherence: low smaller than up}:
 For all $f,g\in\setofgengambles(\mathscr{Y})$ we have that
 \begin{align*}
 \upprev(f)=\upprev(f + g-g) \overset{\text{\ref{coherence: sublinearity}}}{\leq} \upprev(f + g) + \upprev(- g)=\upprev(f + g)-\lowprev(g),
 \end{align*}
 where the last step follows from the definition of $\lowprev$.
 Hence, because $\lowprev(g)$ is real by \ref{coherence: bounds}, we have that $\upprev(f) + \lowprev(g)\leq\upprev(f + g)$ for all $f,g\in\setofgengambles(\mathscr{Y})$.
 The proof of the remaining inequality is completely analogous; it suffices to replace $\upprev$ by $\lowprev$ and $f$ by $g$ in the reasoning above.

\ref{coherence: uniform convergence}:
It is easy to see that, if $\lim_{n\to+\infty} \sup{\vert f-f_n \vert}=0$ for some sequence $\{f_n\}_{n\in\natz}$ in $\setofgengambles(\mathscr{Y})$, then $f$ is also a gamble and so is each $f-f_n$.
Hence, it follows from \ref{coherence: low smaller than up} that
\begin{align}\label{Eq: proof coherence uniform convergence}
\lowprev(f-f_n)\leq\upprev(f) + \lowprev(-f_n)=\upprev(f)-\upprev(f_n)\leq\upprev(f-f_n) \text{ for all } n\in\natz.
\end{align}
If we now apply \ref{coherence: bounds} to $\upprev(f-f_n)$, and \ref{coherence: bounds} and conjugacy to $\lowprev(f-f_n)$, it follows from \eqref{Eq: proof coherence uniform convergence} that $\inf (f-f_n)\leq\upprev(f)-\upprev(f_n)\leq\sup (f-f_n)$ for all  $n\in\natz$.
That $\lim_{n\to+\infty} \vert \upprev(f)-\upprev(f_n) \vert=0$ then follows from $\lim_{n\to+\infty} \sup{\vert f-f_n \vert}=0$.
The equality for the lower expectations $\lowprev$ moreover follows immediately from conjugacy together with the fact that $\sup{\vert (-f)- (-f_n) \vert}=\sup{\vert f-f_n \vert}$.
\end{proof}

\noindent
If $\mathscr{Y}$ is a \emph{finite} set, any upper expectation on $\setofgenextvariablesb(\mathscr{Y})$ additionally satisfies continuity with respect to non-decreasing sequences.

\begin{proposition}\label{Prop: local continuity wrt non-decreasing seq}
Consider a finite non-empty set $\mathscr{Y}$ and an upper expectation $\upprev$ on $\setofgenextvariablesb(\mathscr{Y})$.
Then $\upprev$ satisfies
\begin{enumerate}[leftmargin=26pt,ref={\upshape E\arabic*},label={\upshape E\arabic*}., resume=sepcoherence ]
\item \label{ext coherence continuity}
$\lim_{n\to+\infty} \upprev(f_n)=\upprev\left( \lim_{n\to+\infty} f_n \right)$ for any non-decreasing sequence $\{f_n\}_{n\in\natz}$ in $\setofgenextvariablesb(\mathscr{Y})$.
\end{enumerate} 
\end{proposition}
\begin{proof}
Consider any non-decreasing sequence $\{f_n\}_{n\in\natz}$ in $\setofgenextvariablesb(\mathscr{Y})$ and let $f \coloneqq\lim_{n\to+\infty} f_n$.
Note that then also $f\in\setofgenextvariablesb(\mathscr{Y})$.
Moreover, because $\{f_n\}_{n\in\natz}$ is uniformly bounded below [by $\inf f_0$], we can assume without loss of generality that $f$ and all $f_n$ are non-negative .
Indeed, this can be achieved by adding the same sufficiently large constant to all of them and then applying \ref{coherence: const add}.

Because $\{f_n\}_{n\in\natz}$ is non-decreasing, we have that $f_n \leq f_{n+1} \leq f$ for all $n\in\natz$.
Then it follows from \ref{coherence: monotonicity} that $\upprev(f_n)\leq\upprev(f_{n+1})\leq\upprev(f)$ for all $n\in\natz$.
Hence, $\lim_{n\to+\infty} \upprev(f_n)$ exists and $\lim_{n\to+\infty} \upprev(f_n)\leq\upprev(f)$.
To prove the converse inequality, let $A \coloneqq\{ y\in\mathscr{Y} \colon f(y)=+\infty \}$ and consider the following two cases.
If $\upprev(\indica{A})=0$, we have that 
\begin{equation}\label{Eq: proof cont wrt nondecreasing seq 1}
\upprev(f) 
= \upprev\big((+\infty)\indica{A} + \prodindica{f}{A^c} \big) 
\overset{\text{\ref{coherence: sublinearity}}}{\leq} \upprev\big((+\infty)\indica{A}\big) + \upprev(\prodindica{f}{A^c}) 
\overset{\text{\ref{coherence: homog for ext lambda}}}{=} (+\infty) \upprev(\indica{A}) + \upprev(\prodindica{f}{A^c})
= \upprev(\prodindica{f}{A^c}).
\end{equation}
Because $\prodindica{f}{A^c}$ is real-valued [it cannot be $-\infty$ because it is non-negative] and $\mathscr{Y}$ is finite, $\prodindica{f}{A^c}$ is a gamble and $\{\prodindica{f_n}{A^c} \}_{n\in\natz}$ converges uniformly to $\prodindica{f}{A^c}$.
$\{\prodindica{f_n}{A^c} \}_{n\in\natz}$ is moreover also a sequence of gambles because it converges non-decreasingly to the gamble $\prodindica{f}{A^c}$.
Hence, it follows from \ref{coherence: uniform convergence} that
\begin{align*}
\upprev(\prodindica{f}{A^c}) 
= \lim_{n\to+\infty} \upprev(\prodindica{f_n}{A^c})
\overset{\text{\ref{coherence: monotonicity}}}{\leq} \lim_{n\to+\infty} \upprev(f_n),
\end{align*}
which, together with Equation \eqref{Eq: proof cont wrt nondecreasing seq 1}, leads to the desired inequality.

If $\upprev(\indica{A}) \not= 0$, we have that $\upprev(\indica{A}) > 0$ because of \ref{coherence: bounds}.
Furthermore, all $f_n$ are non-negative, and therefore
\begin{align}\label{Eq: proof cont wrt nondecreasing seq 2}
\upprev(f_n) 
\overset{\text{\ref{coherence: monotonicity}}}{\geq} \upprev(\prodindica{f_n}{A}) 
\overset{\text{\ref{coherence: monotonicity}}}{\geq} \upprev\group[\bigg]{\sqgroup[\Big]{\inf_{y\in A} f_n(y)}\,\indica{A}} 
\overset{\text{\ref{coherence: homog for ext lambda}, \ref{coherence: homog for real lambda}}}{=} \sqgroup[\Big]{\inf_{y\in A} f_n(y)}  \upprev(\indica{A}) \text{ for all } n\in\natz.
\end{align}
Since $\{f_n\}_{n\in\natz}$ converges to $+\infty$ on $A$ and $A$ is moreover finite, we have that $\lim_{n\to+\infty}\inf_{y\in A} f_n(y)=+\infty$.
This implies, together with $\upprev(\indica{A}) > 0$ and \eqref{Eq: proof cont wrt nondecreasing seq 2}, that $\lim_{n\to+\infty} \upprev(f_n)=+\infty$.
Hence, the desired inequality follows.
\end{proof}

\subsection{An Alternative Characterisation using Coherence}
The theory of upper expectations as described by Walley \cite{Walley:1991vk} --- who calls them upper previsions --- only considers gambles.
This allows for a clear behavioural interpretation in terms of attitudes towards gambling \cite{Walley:1991vk,Quaeghebeur:2012vw} , which in turn leads to a notion of rationality that he calls coherence. 
In particular, in that context, a real map $\upprev$ on the linear space $\setofgengambles(\mathscr{Y})$ of all gambles on some non-empty set $\mathscr{Y}$, is called \emph{coherent} if it satisfies the following three \emph{coherence axioms} \cite[Definition 2.3.3]{Walley:1991vk}:
\begin{enumerate}[leftmargin=*,ref={\upshape C\arabic*},label={\upshape C\arabic*}., series=sepcoherence ]
\item \label{sep coherence 1} 
$\upprev(f)\leq\sup f$ for all $f\in\setofgengambles(\mathscr{Y})$; 
\item \label{sep coherence 2} 
$\upprev(f+g)\leq\upprev(f) + \upprev(g)$ for all $f,g\in\setofgengambles(\mathscr{Y})$;
\item \label{sep coherence 3}
$\upprev(\lambda f)=\lambda \upprev(f)$ for all $\lambda\in\posreals{}$ and $f\in\setofgengambles(\mathscr{Y})$.
\end{enumerate}
One can easily show \cite[2.6.1]{Walley:1991vk} that these coherence axioms imply the following additional properties, with $\lowprev(f) \coloneqq -\upprev(-f)$ for all $f\in\setofgengambles(\mathscr{Y})$.
\begin{enumerate}[leftmargin=*,ref={\upshape C\arabic*},label=\normalfont{\upshape C\arabic*}.,resume=sepcoherence]
\item \label{sep coherence 4}
if\/ $f \leq g$ then $\upprev(f)\leq\upprev(g)$ for all $f, g\in\setofgengambles(\mathscr{Y})$;
\item \label{sep coherence 5}
$\inf f\leq\lowprev(f)\leq\upprev(f)\leq\sup f$ for all $f\in\setofgengambles(\mathscr{Y})$;
\item \label{sep coherence 6}
$\upprev(f + \mu)=\upprev(f) + \mu$ for all real $\mu$ and all $f\in\setofgengambles(\mathscr{Y})$;
\item \label{sep coherence 7}
$\lim_{n\to+\infty} \sup{\vert f-f_n \vert}=0 \Rightarrow \lim_{n\to+\infty} \vert \upprev(f)-\upprev(f_n) \vert=0$ for any sequence $\{f_n\}_{n\in\natz}$ in $\setofgengambles(\mathscr{Y})$.
\end{enumerate}

\noindent
It should be clear that any upper expectation satisfies~\ref{sep coherence 1}--\ref{sep coherence 3} and therefore, that its restriction to $\setofgengambles(\mathscr{Y})$ is coherent.
We now set out to prove that, if $\mathscr{Y}$ is finite, upper expectations can be characterised alternatively using \ref{sep coherence 1}--\ref{sep coherence 3} and the following weakened version of \ref{ext coherence continuity}:

\begin{enumerate}[leftmargin=29pt,ref={\upshape E\arabic**},label={\upshape E\arabic**}., start=10]
\item \label{ext coherence continuity gambles}
$\lim_{n\to+\infty} \upprev(f_n)=\upprev\left(\lim_{n\to+\infty} f_n \right)$ for any non-decreasing sequence $\{f_n\}_{n\in\natz}$ in $\setofgengambles(\mathscr{Y})$.
\end{enumerate}
We start with the following lemma.

\begin{lemma}\label{lemma: C1-C3 + cont.=upper exp.}
Consider any non-empty set $\mathscr{Y}$ and any extended real-valued map $\upprev$ on $\setofgenextvariablesb(\mathscr{Y})$.
If \/ $\upprev$ satisfies~\ref{sep coherence 1}--\ref{sep coherence 3} and \ref{ext coherence continuity gambles}, then $\upprev$ is an upper expectation on $\setofgenextvariablesb(\mathscr{Y})$.
\end{lemma}
\begin{proof}
Assume that $\upprev$ satisfies~\ref{sep coherence 1}--\ref{sep coherence 3} and \ref{ext coherence continuity gambles}.
We show that $\upprev$ then also satisfies~\ref{coherence: const is const}--\ref{coherence: monotonicity}.

That \ref{coherence: const is const} holds, follows immediately from \ref{sep coherence 5}.
To prove \ref{coherence: sublinearity}, first note that, for any $f\in\setofgenextvariablesb(\mathscr{Y})$, there is always a non-decreasing sequence $\{f_n\}_{n\in\natz}$ in $\setofgengambles(\mathscr{Y})$ such that $\lim_{n\to+\infty} f_n=f$.
To see this, it suffices to consider the sequence $\{f^{\wedge n}\}_{n\in\natz}$ defined by $f^{\wedge n}(y) \coloneqq\min \{f(y),n\}$ for all $y\in\mathscr{Y}$ and all $n\in\natz$.
Hence, for any two $f,g\in\setofgenextvariablesb(\mathscr{Y})$, we can consider two non-decreasing sequences $\{f_n\}_{n\in\natz}$ and $\{g_n\}_{n\in\natz}$ in $\setofgengambles(\mathscr{Y})$ that converge to $f$ and $g$ respectively.
It then follows from \ref{ext coherence continuity gambles} that $\lim_{n\to+\infty} \upprev(f_n)=\upprev(f)$ and $\lim_{n\to+\infty} \upprev(g_n)=\upprev(g)$.
Moreover, $\{f_n + g_n\}_{n\in\natz}$ is also a non-decreasing sequence in $\setofgengambles(\mathscr{Y})$ and clearly $\lim_{n\to+\infty} (f_n + g_n)=f + g$, which again implies by \ref{ext coherence continuity gambles} that $\lim_{n\to+\infty} \upprev(f_n + g_n)=\upprev(f + g)$.
All together, we have that
\begin{align*}
\upprev(f+g)=\lim_{n\to+\infty} \upprev(f_n + g_n) \overset{\text{\ref{sep coherence 2}}}{\leq} \lim_{n\to+\infty} \left[ \upprev(f_n) + \upprev(g_n)\right]
= \upprev(f) + \upprev(g),
\end{align*}
which concludes the proof of \ref{coherence: sublinearity}.

We prove \ref{coherence: monotonicity} in a similar way.
Consider any $f,g\in\setofgenextvariablesb(\mathscr{Y})$ such that $f \leq g$, and the non-decreasing sequences $\{f^{\wedge n}\}_{n\in\natz}$ and $\{g^{\wedge n}\}_{n\in\natz}$ in $\setofgengambles(\mathscr{Y})$ defined by $f^{\wedge n}(y) \coloneqq\min \{f(y),n\}$ and $g^{\wedge n}(y) \coloneqq\min \{g(y),n\}$ for all $y\in\mathscr{Y}$ and all $n\in\natz$.
Clearly, $f^{\wedge n} \leq g^{\wedge n}$ for all $n\in\natz$ and therefore $\upprev(f^{\wedge n})\leq\upprev(g^{\wedge n})$ by \ref{sep coherence 4}. 
Hence, $\lim_{n\to+\infty} \upprev(f^{\wedge n})\leq\lim_{n\to+\infty} \upprev(g^{\wedge n})$ and therefore, because of \ref{ext coherence continuity gambles}, also $\upprev(f)\leq\upprev(g)$.

That \ref{coherence: homog for ext lambda} holds for real $\lambda$, can be proven as before by using a non-decreasing sequence of gambles and applying \ref{ext coherence continuity gambles} and \ref{sep coherence 3}.
For $\lambda=+\infty$, we require the following more involved argument.
Consider any non-negative $f\in\setofgenextvariablesb(\mathscr{Y})$, fix an arbitrary $c\in\posreals{}$, and let $f^{\wedge c}$ be the gamble defined by  $f^{\wedge c}(y) \coloneqq\min \{ f(y) , c\}$ for all $y\in\mathscr{Y}$.
Then $-\infty < 0\leq\upprev(f^{\wedge c})\leq\upprev(f)$ by \ref{coherence: const is const}, \ref{coherence: monotonicity} and the non-negativity of $f^{\wedge c}$, and $\lambda f=(+\infty) f=(+\infty) f^{\wedge c}=\lim_{n\to+\infty} n f^{\wedge c}$.
Hence,
\begin{align*}
\upprev(\lambda f) 
= \upprev\left(\lim_{n\to+\infty} n f^{\wedge c}\right) 
\overset{\text{\ref{ext coherence continuity gambles}}}{=} \lim_{n\to+\infty} \upprev(n f^{\wedge c})
\overset{\text{\ref{sep coherence 3}}}{=} \lim_{n\to+\infty} n \upprev(f^{\wedge c})
= \lambda \upprev(f^{\wedge c}) 
\leq \lambda \upprev(f),
\end{align*}
where the last equality follows because $\upprev(f^{\wedge c}) \geq 0$ and the last inequality follows from the fact that $\lambda=+\infty$ and $0\leq\upprev(f^{\wedge c})\leq\upprev(f)$.
To see that the converse inequality holds, note that $n f\leq\lambda f$ for all $n\in\nats$ and therefore that $\upprev(n f)\leq\upprev(\lambda f)$ because of \ref{coherence: monotonicity}. 
Applying \ref{coherence: homog for ext lambda} (we already proved that it holds for a real factor) then implies that $n \upprev(f)\leq\upprev(\lambda f)$ for all $n\in\nats$.
Hence, taking into account that $-\infty < \upprev(f)$, we have that $\lambda \upprev(f)=\lim_{n\to+\infty} n \upprev(f)\leq\upprev(\lambda f)$.
\end{proof}

\begin{corollary}\label{corollary: E10' implies E10}
Consider any non-empty set $\mathscr{Y}$ and any extended real-valued map $\upprev$ on $\setofgenextvariablesb(\mathscr{Y})$ that satisfies~\ref{sep coherence 1}--\ref{sep coherence 3} and \ref{ext coherence continuity gambles}.
Then $\upprev$ also satisfies~\ref{ext coherence continuity}.
\end{corollary}
\begin{proof}

Consider any non-decreasing sequence $\{f_n\}_{n\in\natz}$ in $\setofgenextvariablesb{}(\mathscr{Y})$ and let $f \coloneqq \lim_{n \to +\infty} f_n$.
Let $\{f^{\wedge n}_n\}_{n\in\natz}$ be the sequence defined by $f^{\wedge n}_n(y) \coloneqq \min \{f_n(y), n \}$ for all $y \in \mathscr{Y}$ and all $n\in\natz$.
Clearly, $\{f^{\wedge n}_n\}_{n\in\natz}$ is then also a non-decreasing sequence in $\setofgenextvariablesb{}(\mathscr{Y})$ that converges point-wise to $f$.
It is moreover a sequence of gambles because every $f_n$ is also bounded above by $n$.
Hence, by \ref{ext coherence continuity gambles}, 
\begin{align}\label{Eq: proof corollary E10' implies E10}
\upprev{}(f) = \upprev{}\left(\lim_{n \to +\infty} f^{\wedge n}_n\right) = \lim_{n \to +\infty} \upprev{}( f^{\wedge n}_n).
\end{align}

Lemma \ref{lemma: C1-C3 + cont.=upper exp.} guarantees that $\upprev{}$ is an upper expectation, implying that it satisfies \ref{coherence: monotonicity}.
Hence, because $f^{\wedge n}_n \leq f_n$, we have that $\upprev{}(f^{\wedge n}_n) \leq \upprev{}(f_n)$ for all $n\in\natz$, and therefore also $\lim_{n \to +\infty} \upprev{}(f^{\wedge n}_n) \leq \lim_{n \to +\infty} \upprev{}(f_n)$.
Then it follows from Equation \eqref{Eq: proof corollary E10' implies E10} that $\upprev{}(f) \leq \lim_{n \to +\infty} \upprev{}(f_n)$.
To show that also $\upprev{}(f) \geq \lim_{n \to +\infty} \upprev{}(f_n)$, it suffices to note that $f \geq f_n$ and again apply \ref{coherence: monotonicity}.
\end{proof}

\begin{proposition}\label{Prop: alt. characterisation upper exp.}
Consider any finite non-empty set $\mathscr{Y}$ and any extended real-valued map $\upprev$ on $\setofgenextvariablesb(\mathscr{Y})$.
Then $\upprev$ is an upper expectation if and only if it satisfies~\ref{sep coherence 1}--\ref{sep coherence 3} and \ref{ext coherence continuity gambles}.
\end{proposition}
\begin{proof}
If $\upprev$ is an upper expectation, it trivially satisfies~\ref{sep coherence 1}--\ref{sep coherence 3}.
Moreover, \ref{ext coherence continuity gambles} is also satisfied because of Proposition~\ref{Prop: local continuity wrt non-decreasing seq}.
The converse implication follows from Lemma~\ref{lemma: C1-C3 + cont.=upper exp.}.
\end{proof}

We end by proving a convenient countable super-additivity property.

\begin{lemma}\label{lemma: countable subadditivity}
Consider any non-empty set $\mathscr{Y}$ and any extended real-valued map $\upprev$ on $\setofgenextvariablesb(\mathscr{Y})$ that satisfies~\ref{sep coherence 1}--\ref{sep coherence 3} and \ref{ext coherence continuity gambles}.
Then 
$\upprev\left(\sum_{n\in\nats} f_n\right)\leq\sum_{n\in\nats} \upprev(f_n)$ 
for all non-negative sequences $\{f_n\}_{n\in\nats}$ in $\setofgenextvariablesb(\mathscr{Y})$.
\end{lemma}
\begin{proof}
Consider the sequence $\{g_n\}_{n\in\nats}$ in $\setofgenextvariablesb(\mathscr{Y})$ defined by $g_n \coloneqq\sum_{i=1}^{n} f_i$ for all $n\in\nats$.
Then, $\{g_n\}_{n\in\nats}$ is non-decreasing because $\{f_n\}_{n\in\nats}$ is non-negative.
Moreover, it is clear that $\{g_n\}_{n\in\nats}$ converges point-wise to $\sum_{n\in\nats} f_n$.
Hence, by Corollary~\ref{corollary: E10' implies E10}, we can apply \ref{ext coherence continuity} to find that 
\begin{align*}
\upprev\left(\sum_{n\in\nats} f_n\right) 
\overset{\text{\ref{ext coherence continuity}}}{=} \lim_{n\to+\infty} \upprev(g_n) 
= \lim_{n\to+\infty} \upprev\left(\sum_{i=1}^{n} f_i\right)
\overset{\text{\ref{coherence: sublinearity}}}{\leq}  \lim_{n\to+\infty} \sum_{i=1}^{n} \upprev(f_i)
= \sum_{n\in\nats} \upprev(f_n),
\end{align*}
where we can apply \ref{coherence: sublinearity} because $\upprev$ is an upper expectation by Lemma~\ref{lemma: C1-C3 + cont.=upper exp.}.
\end{proof}

\begin{corollary}\label{corollary: countable subadditivity for upper exp}
Consider any finite non-empty set $\mathscr{Y}$ and any upper expectation $\upprev$ on $\setofgenextvariablesb(\mathscr{Y})$.
Then 
$\upprev\left(\sum_{n\in\nats} f_n\right)\leq\sum_{n\in\nats} \upprev(f_n)$ 
for all non-negative sequences $\{f_n\}_{n\in\nats}$ in $\setofgenextvariablesb(\mathscr{Y})$.
\end{corollary}
\begin{proof}
This follows immediately from Lemma~\ref{lemma: countable subadditivity} and the fact that $\upprev$ satisfies~\ref{sep coherence 1}--\ref{sep coherence 3} and \ref{ext coherence continuity gambles} by Proposition~\ref{Prop: alt. characterisation upper exp.}.
\end{proof}

\section{Uncertain Processes}
We consider a sequence of uncertain states $X_1, X_2, ..., X_n , ...$ where the state $X_k$ at each discrete time $k\in\nats$ takes values in some fixed non-empty \emph{finite} set $\statespace$, called the \emph{state space}.
Such a sequence will be called an \emph{uncertain (finite state) process}.
We call any $\sit \coloneqq (x_1,...,x_n)\in\statespace_{1:n} \coloneqq\statespace^n$, for $n\in\natz$, a \emph{situation} and we denote the set of all situations by $\statespace^\ast \coloneqq\cup_{n\in\natz} \mathscr{X}_{1:n}$.
So any finite string of possible values for a sequence of consecutive states is called a situation.
In particular, the unique empty string $x_{1:0}$, denoted by $\Box$, is called the \emph{initial situation}, and $\statespace_{1:0} \coloneqq\{\Box\}$.

An infinite sequence of state values $\omega$ is called a \emph{path}, and the set of all paths is called the \emph{sample space} $\samplespace \coloneqq\statespace^\nats$.
For any path $\omega\in\samplespace$, the initial sequence that consists of its first $n$ state values is a situation in $\statespaceseq{1}{n}$ that is denoted by $\omega^n$. 
The $n$-th state value is denoted by $\omega_n\in\statespace$.
% In particular, we have $\omega^0=\omega_0=\Box$ for all $\omega\in\samplespace$.

We will distinguish between local variables and global variables.
A \emph{local} variable is a variable on the set $\statespace$ of all state values, whereas a \emph{global} variable is a variable on the set $\samplespace$ of all paths.
We will show below how we can associate a global variable with any local one.
We denote the set of all \emph{global} extended real variables by $\setofextvariables \coloneqq\setofgenextvariables(\samplespace)$, and similarly for $\setofextvariablesb \coloneqq\setofgenextvariablesb(\samplespace)$ and $\setofgambles \coloneqq\setofgengambles(\samplespace)$.

For any natural $k\leq\ell$, we use $X_{k:\ell}$ to denote the global variable that assumes the value $X_{k:\ell}(\omega)\coloneqq(\omega_k,...,\omega_\ell)$ on the path $\omega\in\samplespace$.
As such, the state $X_k=X_{k:k}$ at any discrete time $k$ can also be regarded as a global variable.
For any $m, n\in\nats$ and any map $f \colon \statespace^n \to \extreals$, we will write $f(X_{m+1:m+n})$ to denote the extended real global variable defined by $f(X_{m+1:m+n}) \coloneqq f \circ X_{m+1:m+n}$.
In particular, we can associate a global variable $f(X_n)$ with any local variable $f \colon \statespace \to \extreals$ and any index $n\in\nats$.

A collection of paths $A \subseteq \samplespace$ is called an \emph{event}.
% The \emph{indicator} $\indica{A}$ of an event $A$ is defined as the global variable that assumes the value $1$ on $A$ and $0$ elsewhere.
With any situation $\sit$, we associate the \emph{cylinder event} $\Gamma(\sit)\coloneqq\{\omega\in\samplespace\colon \omega^n=\sit\}$: the set of all paths $\omega\in\samplespace$ that `go through' the situation $\sit$.
Sometimes, when it is clear from the context, we will also use the notation `$\sit$' to denote the \emph{set} $\Gamma(\sit)$.
So for example, $\indica{\sit}$ is equal to $\indica{\Gamma(\sit)}$.
Moreover, for any two variables $g,h\in\setofextvariables$ and any situation~$s\in\situations$, we use $g \leq_s f$ to denote that $g(\omega) \leq f(\omega)$ for all $\omega\in\Gamma(s)$, and similarly for $\geq_s$, $>_s$ and $<_s$.

For a given $n\in\natz$, we call a global variable $f$ \emph{$n$-measurable} if it is constant on the cylinder events $\Gamma(\sit)$ for all $\sit\in\statespaceseq{1}{n}$, that is, if $f=\tilde{f}(X_{1:n})$ for some map $\tilde{f}$ on $\statespace^n$. 
We will then also use the notation $f(\sit)$ for its constant value $f(\omega)$ on all paths $\omega\in\Gamma(\sit)$.
Similarly, for a global variable $f$ that only depends on the $n$-th state $X_n$, we will use $f(x_n)$ to denote its constant value on the event $\{\omega\in\Omega \colon \omega_n=x_n\}$.
We call a global variable $f\in\setofextvariables$ \emph{finitary} if it is $n$-measurable for some $n\in\natz$.

% In accordance with the notation above, all extended real variables, bounded below extended real variables and gambles on $\samplespace$ that are $n$-measurable are gathered in the respective sets $\setofnextvariables{}$, $\setofnextvariablesb{}$ and $\setofngambles{}$.
We will also use the notation $\setoflimitsoffinmeasext{}$ for the set of all extended real global variables $f$ that are the point-wise limit of some sequence of finitary variables.
Furthermore, $\setoflimitsoffinmeasb{}$ is the set of all bounded below (extended real) variables in $\setoflimitsoffinmeasext{}$, and $\setoflimitsoffinmeas{}$ the set of all gambles in $\setoflimitsoffinmeasext{}$.

% \begin{lemma}\label{lemma: limits of n-meas gambles equal to limits of n-meas variables}
% For any $f\in\setoflimitsoffinmeas{}$ there is a sequence of $n$-measurable gambles $\{f_n\}_{n\in\natz}$ that is uniformly bounded above and below, such that it converges point-wise to $f$.
% \end{lemma}
% \begin{proof}
% Since $f\in\setoflimitsoffinmeas{}$, there is a sequence of $n$-measurable, possibly extended real, variables $\{f_n\}_{n\in\natz}$ that converges point-wise to $f$.
% Moreover, $f$ is a gamble, so $\inf f\in\reals{}$ and $\sup f\in\reals{}$. 
% Then it is easy to see that the sequence $\{f'_n\}_{n\in\natz}$ constructed by bounding $\{f_n\}_{n\in\natz}$ above by $\sup f$ and below by $\inf f$ --- so, it is defined by $f'_n(\omega) \coloneqq\min\{ \max\{f_n(\omega),\inf f \}, \sup f\}$ for all $\omega\in\samplespace$ and all $n\in\natz$ --- is a sequence of $n$-measurable gambles that is uniformly bounded above and below.
% Furthermore, $\{f'_n\}_{n\in\natz}$ converges point-wise to $f$ because $\{f_n\}_{n\in\natz}$ converges point-wise to $f$ and $\vert f(\omega)-f'_n(\omega) \vert\leq\vert f(\omega)-f_n(\omega) \vert$ for any $n\in\natz$ and all $\omega\in\samplespace$.
% \end{proof}

\begin{lemma}\label{lemma: limits of n-measurables are equal to limits of finitary}
For any $f\in\setofextvariables{}$, we have that $f\in\setoflimitsoffinmeasb{}$ if and only if $f$ is the point-wise limit of some sequence $\{f_n\}_{n\in\natz}$ of $n$-measurable gambles that is uniformly bounded below such that $f_n \leq \sup f$ for all $n \in \natz{}$.
If \/ $\inf f \in \reals$, then we can moreover guarantee that $\inf f \leq f_n$ for all $n \in \natz{}$.
\end{lemma}
\begin{proof}
It is clear that any $f\in\setofextvariables{}$ that is the point-wise limit of a sequence $\{f_n\}_{n\in\natz}$ of $n$-measurable gambles that is uniformly bounded below such that $f_n \leq \sup f$ for all $n \in \natz{}$, is an element of $\setoflimitsoffinmeasb{}$.
So suppose that $f \in \setoflimitsoffinmeasb{}$, meaning that $f = \lim_{n \to +\infty} g_n$ is the point-wise limit of a sequence $\{g_n\}_{n\in\natz}$ of, possibly extended real, finitary variables.
We first show that $f$ is then also the limit of a sequence $\{h_n\}_{n\in\natz}$ of $n$-measurable variables.

Let $h_0 \coloneqq c$ for some $c \in \reals{}$ and $\gamma(1) \coloneqq 0$.
Let $\{h_n\}_{n\in\natz}$ be defined by the following recursive expressions:
\begin{flalign*}
h_n &\coloneqq
\begin{cases}
g_{\gamma(n)} &\text{ if $g_{\gamma(n)}$ is $n$-measurable} \\
h_{n-1} &\text{ otherwise}
\end{cases} \\
\text{ and } \qquad \qquad \qquad \qquad \qquad & & \\ 
\gamma(n+1) &\coloneqq
\begin{cases}
\gamma(n)+1 &\text{ if $g_{\gamma(n)}$ is $n$-measurable}\\
\gamma(n) &\text{ otherwise.}
\end{cases}
\end{flalign*}
for all $n\in\nats$.
The original sequence $\{g_n\}_{n\in\natz}$ is a subsequence of $\{h_n\}_{n\in\natz}$. 
The additional elements in $\{h_n\}_{n\in\natz}$ clearly do not change the limit behaviour and therefore both limits are equal.
We show by induction that $\{h_n\}_{n\in\natz}$ is a sequence of $n$-measurable variables, and hence $f$ is a limit of $n$-measurable variables.
$h_0=c$ is clearly $0$-measurable.
To prove the induction step, suppose that $h_{n-1}$ is $(n-1)$-measurable for some $n\in\nats$.
Then either we have that $g_{\gamma(n)}$ is $n$-measurable, which directly implies that $h_n=g_{\gamma(n)}$ is $n$-measurable.
Otherwise, $h_n$ is equal to $h_{n-1}$ implying that $h_n$ is $(n-1)$-measurable and therefore automatically $n$-measurable.
This concludes the induction step and hence, $f$ is a limit of $n$-measurable variables $\{h_n\}_{n \in \natz{}}$.

To show that $f$ is also a limit of a sequence $\{f_n\}_{n\in\natz}$ of $n$-measurable gambles that is uniformly bounded below such that $f_n \leq \sup f$ for all $n \in \natz{}$, we first assume that $\inf f$ is real.
Let $\{f_n\}_{n\in\natz}$ be the sequence defined by bounding each $h_n$ above by $\min\{n, \sup f\}$ and below by $\inf f$; so $f_n(\omega) \coloneqq \max\{ \min\{h_n(\omega), n, \sup f \}, \inf f \}$ for all $\omega \in \samplespace{}$ and all $n\in\natz$.
Then it is clear that $\{f_n\}_{n\in\natz}$ is a sequence of $n$-measurable gambles because $\{h_n\}_{n \in \natz{}}$ is a sequence of $n$-measurable (possibly extended real) variables.
It also converges point-wise to $f$ because 
\begin{align*}
\lim_{n \to +\infty} f_n(\omega) 
&= \lim_{n \to +\infty} \max\Big\{ \min\{h_n(\omega), n, \sup f \} , \inf f \Big\} \\
&=  \max\Big\{ \min\big\{ \lim_{n \to +\infty} h_n(\omega), +\infty, \sup f \big\}, \inf f \Big\} \\
&=  \max\Big\{ \min\{ f(\omega), \sup f \}, \inf f \Big\} = f(\omega),
\end{align*}
for any $\omega \in \samplespace{}$.
Since moreover $\inf f \leq f_n \leq \sup f$ for all $n \in \natz{}$ and $\{f_n\}_{n \in \natz{}}$ is uniformly bounded below [because $\inf f$ is real], we have immediately established both claims in the lemma if $\inf f$ is real.
Finally, if $\inf f = +\infty$, implying that $f = \sup f = +\infty$, it suffices to consider the increasing sequence of gambles $\{n\}_{n \in \natz{}}$ to see that the first claim in the lemma holds.
\end{proof}

\section{Supermartingales}
Any map $\process$ on $\situations$ is called a \emph{process}. 
A real process $\process \colon \situations \to \reals{}$ is a real-valued map on $\situations$ and an extended real process $\process \colon \situations \to \extreals$ is a extended real-valued map on $\situations$. 
An extended real process is called positive (non-negative) if it is positive (non-negative) in every situation.
An extended real process $\process$ is called bounded below if there is some $M\in\reals{}$ such that $\process(s) \geq M$ for all $s\in\situations$.
For any extended real process $\process$ we can consider the extended real variable $\process(X_{1:n}) \coloneqq\process \circ X_{1:n}$ that only depends of the first $n$ states, and is therefore finitary.
% With any extended real process $\process$ we associate a sequence of $n$-measurable variables $\{\process_n\}_{n\in\natz}$: for all $n\in\natz$, we let $\process_n(\omega) \coloneqq\process(\omega^n)$ for all $\omega\in\samplespace$ or, equivalently, $\process_n \coloneqq\process \circ X_{1:n}=\process(X_{1:n})$.
Moreover, with any situation $s\in\situations$, we can associate the local variable $\process(s\andstate)\in\setofgenextvariables(\statespace)$ defined by
\begin{align*}
\process(s\andstate)(x) &\coloneqq\process(sx) \text{ for all } x\in\statespace.
\end{align*} 
% An \emph{$s$-process}, for some situation $s\in\situations$, is a process with domain $\{t\in\situations : s \sqsubseteq t\}$, instead of $\situations$.
% A \emph{gamble process} $\mathscr{D}$ is a process that associates with any situation $x_{1:n}\in\situations$ a local gamble $\mathscr{D}(x_{1:n})\in\setofgambles(\statespace)$. 
% With any real process $\process$, we can associate a gamble process $\Delta\process$, called its \emph{process difference}. 
% For any situation $x_{1:n}$ the corresponding gamble $\Delta\process(x_{1:n})\in\setofgambles(\statespace)$ is defined by
% \begin{equation*}
% \Delta\process(x_{1:n})(x_{n+1})\coloneqq\process(x_{1:n+1}) -\process(x_{1:n}) \text{ for all } x_{n+1}\in\statespace.
% \end{equation*}

\noindent
We will also use the extended real variables $\liminf\process\in\setofextvariables$ and $\limsup\process\in\setofextvariables$, defined by:
\begin{align*}
\liminf\process(\omega)\coloneqq\liminf_{n\to+\infty}\process(\omega^n)
\text{~~and~~}
\limsup\process(\omega)\coloneqq\limsup_{n\to+\infty}\process(\omega^n) 
\end{align*}
for all $\omega\in\Omega$.
If $\liminf\process=\limsup\process$, we denote their common value by $\lim\process$.

In a so-called \emph{imprecise probability tree} we attach to each situation $s\in\situations$ a \emph{local} uncertainty model: an upper expectation $\lupprev{s}$ on $\setofgenextvariablesb(\statespace)$.
Since $\statespace$ is finite, Proposition~\ref{Prop: local continuity wrt non-decreasing seq} implies that $\lupprev{s}$ also satisfies~\ref{ext coherence continuity}.
For a \emph{given} imprecise probability tree, a \emph{supermartingale} $\martingale$ is an extended real process that is \emph{bounded below}\footnote{Traditionally, the condition of being bounded below is not included as part of the definition of a supermartingale. However, introducing them without this additional requirement here would  necessitate an additional extension of the domain of the local models and would therefore be rather cumbersome. Since we only consider supermartingales that are bounded below, we include this condition directly in their definition.} and such that $\lupprev{s}(\martingale(s\andstate))\leq\martingale(s)$ for all $s\in\situations$. 
In other words, a supermartingale is an extended real process that is bounded below and such that, according to the local models, is expected to decrease.
The condition that $\lupprev{s}(\martingale(s\andstate))\leq\martingale(s)$ for all $s\in\situations$ is well-defined because $\martingale$, and therefore also the local variable $\martingale(s\andstate)$, is bounded below.
We denote the set of all supermartingales \emph{for a given imprecise probability tree} by $\setofextsupmartb{}$. 

\begin{lemma}\label{LemmaBoundedSupermartingale}
Consider any supermartingale $\martingale\in\setofextsupmartb$ and, for any $B\in\reals$, the real process $\martingale_B$ defined by
\begin{equation*}
\martingale_B(s) 
\coloneqq\min\{\martingale(s), B\} \text{ for all } s\in\situations.
\end{equation*} 
Then $\martingale_B$ is a real supermartingale.
\end{lemma}
\begin{proof}
It is clear that, since $\martingale$ is a bounded below \emph{extended real} process, $\martingale_B$ is a bounded below \emph{real} process. 
Moreover, $\martingale_B(s)\leq\martingale(s)$ for all $s\in\situations$, so it follows that $\martingale_B(s\andstate)\leq\martingale(s\andstate)$ for all $s\in\situations$.
Fix any $s\in\situations$.
If $\martingale(s) \leq B$, then
\begin{align*}
\lupprev{s}(\martingale_B(s\andstate))\leq\lupprev{s}(\martingale(s\andstate))\leq\martingale(s)=\martingale_B(s),
\end{align*}
where the first equality follows from the monotonicity \ref{coherence: monotonicity} of $\lupprev{s}$.
If $\martingale(s) > B$, it follows from $\martingale_B(s\andstate) \leq B$ and the monotonicity \ref{coherence: monotonicity} of $\lupprev{s}$ that
\begin{align*}
\lupprev{s}(\martingale_B(s\andstate))\leq\lupprev{s}( B) \overset{\text{\ref{coherence: const is const}}}{=} B=\martingale_B(s).
\end{align*}
So, we conclude that $\lupprev{s}(\martingale_B(s\andstate))\leq\martingale_B(s)$ for all situations $s\in\situations$.
Hence, $\martingale_B$ is a real supermartingale.
\end{proof}

\begin{lemma}\label{Lemma: the minimum-limit inferior switch}
Consider any extended real process $\process$ that is bounded below and any path $\omega\in\samplespace$. Then 
\begin{equation*}
\min \left\{B,\liminf_{n\to+\infty}\process(\omega^n)\right\}=\liminf_{n\to+\infty} \min\{B,\process(\omega^n)\} \text{ for all $B\in\reals$.}
\end{equation*} 
\end{lemma}
\begin{proof}
Consider any $B\in\reals$.
It is easy to check that 
\begin{equation*}
\min \left\{B,\liminf_{n\to+\infty}\process(\omega^n)\right\} \geq \liminf_{n\to+\infty} \min\{B,\process(\omega^n)\}.
\end{equation*} 
We prove the converse inequality by contradiction. 
Suppose that 
\begin{align}
\min \left\{B,\liminf_{n\to+\infty}\process(\omega^n) \right\} &> 
\liminf_{n\to+\infty} \min\{B,\process(\omega^n)\} \nonumber,
\end{align}
or, equivalently, that
\begin{equation*}
\min \left\{B,\liminf_{n\to+\infty}\process(\omega^n) \right\} > \sup_{m}\inf_{n \geq m} \min\{B,\process(\omega^n)\}.
\end{equation*}
Then there is some $\epsilon > 0$ such that 
\begin{equation*}
\min \set*{B,\liminf_{n\to+\infty}\process(\omega^n)}-\epsilon >\inf_{n \geq m} \min\{B,\process(\omega^n)\}=\min \set*{B,\inf_{n \geq m}\process(\omega^n)}
\end{equation*}
for all $ m\in\natz$.
Since $\min \{B,\liminf_{n\to+\infty}\process(\omega^n)\}-\epsilon < B$, this implies that
\begin{align*}
\inf_{n \geq m}\process(\omega^n) < \min \set*{B,\liminf_{n\to+\infty}\process(\omega^n)}-\epsilon\leq\liminf_{n\to+\infty}\process(\omega^n)-\epsilon \text{ for all $m\in\natz$,}
\end{align*}
from which we infer that
\begin{align*}
\inf_{n \geq m}\process(\omega^n) < \sup_{k}\inf_{n \geq k}\process(\omega^n)-\epsilon \text{ for all $m\in\natz$,}
\end{align*} 
contradicting the definition of the supremum operator.
\end{proof}

\begin{lemma}\label{lemma: infima of supermartingales}
	Consider any supermartingale $\martingale\in\setofextsupmartb{}$ and any situation~$s\in\situations$. Then
	\begin{equation*}
		\martingale(s) \geq\inf_{\omega\in\Gamma(s)} \limsup\martingale(\omega) \geq\inf_{\omega\in\Gamma(s)} \liminf\martingale(\omega).
	\end{equation*}
\end{lemma}
\begin{proof}
The proof is similar to that of \cite[Lemma 1]{DECOOMAN201618}.
Since $\martingale$ is a supermartingale, we have that $\lupprev{s}(\martingale(s\andstate))\leq\martingale(s)$, which implies by coherence [\ref{coherence: bounds}] of $\lupprev{s}$ that $\inf_{x\in\statespace} \martingale(sx)\leq\martingale(s)$.
Hence, since $\statespace$ is finite, there is at least one $x\in\statespace$ such that $\martingale(sx)\leq\martingale(s)$.
Repeating this argument over and over again, leads us to the conclusion that there is some $\omega\in\Gamma(s)$ such that $\limsup_{n\to+\infty} \martingale(\omega^n)\leq\martingale(s)$ and therefore also $\inf_{\omega\in\Gamma(s)} \limsup\martingale(\omega)\leq\martingale(s)$.
The rest of the proof is now trivial.
\end{proof}

% \begin{lemma}\label{lemma:positive:countable:linear:combination}
% Consider any two supermartingales $\martingale$ and $\martingale'$ and any two real numbers $\epsilon, \delta \geq 0$.
% Then the process $\epsilon \martingale + \delta \martingale'$ is also a supermartingale.
% \end{lemma}
% \begin{proof}
% It is easy to see that the processes $\epsilon \martingale$ and $\delta \martingale'$ are both supermartingales.
% Note that they are obviously bounded below and that both $\lupprev{s}(\epsilon \martingale(s\andstate))=\epsilon \martingale(s)$ and $\lupprev{s}(\delta \martingale'(s\andstate))=\delta \martingale'(s)$ hold because the local models satisfy \ref{coherence: homog for real lambda}.
% So, to prove that $\epsilon \martingale + \delta \martingale'$ is a supermartingale, it suffices to apply \ref{coherence: sublinearity}. 
% \end{proof}

\begin{lemma}\label{lemma:positive:countable:linear:combination}
Consider any countable collection $\{\martingale_n\}_{n\in\nats}$ of supermartingales that have some common lower bound, and any countable collection of non-negative real numbers $\{\lambda_n\}_{n\in\nats}$ such that $\sum_{n\in\nats}\lambda_n$ is a real number $\lambda$.
Then $\martingale\coloneqq\sum_{n\in\nats}\lambda_n\martingale_n$ is again a supermartingale.
If, moreover, all $\martingale_n$ are non-negative, then so is $\martingale$.
\end{lemma}

\begin{proof}
We only prove the first statement, as the second is then trivially true.
Since all $\martingale_n$ have a common lower bound, say $B \in \reals{}$, the processes $\martingale{}_n - B$ will be non-negative and they will moreover be supermartingales because of \ref{coherence: const add}.
Then because all $\lambda_n$ and all $\martingale{}_n - B$ are non-negative, the sum $\sum_{n\in\nats}\lambda_n[\martingale_n(s)-B]$ exists and is also non-negative for all $s \in \situations{}$.
To see that the non-negative process $\sum_{n\in\nats}\lambda_n[\martingale_n-B]$ is a supermartingale, fix any $t \in \situations{}$ and note that
\begin{align*}
\lupprev{t}\group[\bigg]{\sum_{n\in\nats} \lambda_n [\martingale_n(t\andstate) - B] }
&\leq \sum_{n\in\nats} \lupprev{t}\Bigl( \lambda_n [\martingale_n(t\andstate) - B] \Bigr) \\
&\overset{\text{\ref{coherence: homog for real lambda}}}{=} \sum_{n\in\nats} \lambda_n \lupprev{t}\bigl([\martingale_n(t\andstate) - B]\bigr) \\
&\leq \sum_{n\in\nats} \lambda_n (\martingale_n(t) - B), 
\end{align*}
where the first inequality follows from Corollary~\ref{corollary: countable subadditivity for upper exp} [which we can apply because all $\lambda_n [\martingale_n(t\andstate) - B]$ are non-negative] and the sum on the right hand side of this inequality exists since all $\lupprev{t}( \lambda_n [\martingale_n(t\andstate) - B])$ are non-negative because of \ref{coherence: bounds} and the fact that all $\lambda_n [\martingale_n(t\andstate) - B]$ are non-negative.
The last inequality follows from the fact that all $\martingale{}_n - B$ are supermartingales.

Since $\sum_{n\in\nats}\lambda_n[\martingale_n-B]$ is a supermartingale, the process $\sum_{n\in\nats}\lambda_n[\martingale_n-B] + \lambda B$ is also a supermartingale because of \ref{coherence: const add} [which we can apply because $\lambda B \in \reals{}$].
Moreover, for any $t \in \situations{}$, we have that
\begin{align*}
\sum_{n\in\nats}\lambda_n[\martingale_n(t)-B] + \lambda B 
&= \sum_{n\in\nats}\lambda_n[\martingale_n(t)-B] + \sum_{n\in\nats}\lambda_n B \\
&= \sum_{n\in\nats}\Bigl(\lambda_n[\martingale_n(t)-B] + \lambda_n B \Bigr)
= \sum_{n\in\nats}\lambda_n \martingale_n(t), 
\end{align*}
where the second equality follows from the fact that $\sum_{n\in\nats}\lambda_n B$ is real and the third from the fact that $B$ and all $\lambda_n$ are real.
Hence, the process $\sum_{n\in\nats}\lambda_n \martingale_n$ is equal to $\sum_{n\in\nats}\lambda_n[\martingale_n-B] + \lambda B$, which is a supermartingale, therefore proving the stated.
\end{proof}

\section{Game-theoretic Upper Expectations}
Given an imprecise probability tree consisting of local upper expectations $\lupprev{s}$ for all $s\in\situations$, we use its compatible set of supermartingales $\setofextsupmartb{}$ to construct a global uncertainty model $\upprevvovkk$ as follows.

\begin{definition}\label{def:upperexpectation2}
The map $\upprevvovkk(\cdot \vert \cdot) \colon \setofextvariables \times \situations \to \extreals$ is defined by
\begin{align}\label{upprev3}
\upprevvovkk (f \vert s)\coloneqq\inf \big\{ \martingale(s) \colon \martingale\in\setofextsupmartb \text{ and } \liminf\martingale \geq_s f \big\},
\end{align}
for all extended real variables $f\in\setofextvariables$ and all $s\in\situations$. 
\end{definition}

\noindent
Moreover, the conjugate map $\lowprevvovkk(\cdot \vert \cdot) \colon \setofextvariables \times \situations \to \extreals$ is defined by $\lowprevvovkk(f \vert s) \coloneqq -\upprevvovkk(-f \vert s)$ for all $f\in\setofextvariables$ and all $s\in\situations$.
We will show later (see Corollary~\ref{corollary: Vovk is an upper expectation}) that, for any $s\in\situations$, the map $\upprevvovkk(\cdot \, \vert s) \colon \setofextvariables \to \extreals$ satisfies~\ref{coherence: const is const}--\ref{coherence: monotonicity} on $\setofextvariablesb$.
We will therefore call $\upprevvovkk$ the \emph{global upper expectation} corresponding to the considered probability tree. 

In the remainder of this paper we will study this global upper expectation, proving several properties ranging from basic compatibility with the local models to more involved continuity properties.
Before we do so, we want to stress that Definition \ref{def:upperexpectation2} is mainly due to the work of Shafer and Vovk. 
% Moreover, due to~\cite[Prop. 8.8]{Shafer:2005wx}, for every situation $s$, the restriction of $\upprevvovk(\cdot\vert s)$ to $\setofgambles(\samplespace)$ satisfies the coherence axioms \ref{coherence: const is const}--\ref{coherence 3}. 
However, they have been using many different versions of global upper expectations throughout their work.
The link with our setting can therefore be rather unclear for readers that are not familiar with the theory.
Hence, it seems appropriate to give a brief overview of how our work here relates to theirs.

Most of the definitions they consider only differ in how the supermartingales are allowed to behave.
In \cite{Shafer:2005wx}, they mainly consider supermartingales to be real-valued processes instead of extended real-valued ones. 
In \cite[Chapter 6]{Augustin:2014di} they define global upper expectations on gambles using a version where supermartingales are not necesarily bounded below.
However, this definition leads to undesirable behaviour when applying it to extended real variables, as shown in \cite[Example 1]{DECOOMAN201618}.
In \cite{shafer2011levy}, a version similar to Definition \ref{def:upperexpectation2} with extended real-valued supermartingales that are bounded below is used.
Because of the parallel between both definitions, we here chose to axiomatise the local models in a way that mimics theirs.
However, as we have shown in Section \ref{Sect: Upper Expectations}, we can use an alternative characterisation based on coherence when considering local models on a finite state space $\statespace$.
As we intend to show elsewhere \cite{Tjoens2019NaturalExtensionISIPTA}, this characterisation allows for an intuitive and practically sensible way to motivate the framework.

From a technical point of view, we would like to point out that our axioms \ref{coherence: const is const}--\ref{coherence: monotonicity} differ from Shafer and Vovk's axioms for an outer probability content, in the sense that \ref{coherence: homog for ext lambda} is stronger than their version of this axiom since it also allows $\lambda$ to be $+\infty$.
As we have seen, this implies that, if the state space is finite, the local models are continuous with respect to non-decreasing sequences.
This is a property that will be essential in order to prove some of our results below.
Shafer, Vovk and Takemura initially impose this continuity property as an extra axiom, but afterwards state that this continuity property is redundant to prove their results in \cite{shafer2011levy}.
However, their axioms of an outer probability content [1.--4.] are too weak to prove all of our results presented here.

Another difference is that our local models are only defined on $\setofgenextvariablesb(\statespace)$ and not on $\setofgenextvariables(\statespace)$ as in their case.
This does not make a difference regarding the definition of supermartingales --- and therefore the definition of $\upprevvovkk$ --- since these supermartingales are required to be bounded below.
However, the advantage of our approach is that it implies that the global upper expectation is compatible with the local models on their entire domain.
If the local models would be defined on the entire set $\setofgenextvariables(\statespace)$ of all extended real variables on $\statespace$, this is not necessarily the case, unless one imposes additional properties on the local models.
We choose not to do so.
Another difference is that Shafer and Vovk do not require the state space $\statespace$ to be finite.
In that sense, their definition is more general.

Finally, we want to mention that some of our results, especially the ones in Section \ref{Sect: Doob and Levy}, were already proven by Shafer et. al., although typically for a slightly different setting.
Sometimes, their results apply to a different domain, sometimes they prove these results within a different setting or using a different argument.
We then often borrow their ideas and adapt them to our setting.
When we do, we will mention this explicitly.
Moreover, during the writing of this paper, it has come to our attention that some of the results that we present here are similar to the results in the soon to be released new book of Shafer and Vovk \cite{Vovk2019finance}.
The extent to which these results coincide with ours, remains to be seen.
An important exception are the results in Section \ref{Sect: Continuity wrt n-measurables}, which, to the best of our knowledge, have never been considered by Shafer and Vovk, nor will be in their new book. 

\section{Basic Properties of Game-Theoretic Upper Expectations}
% Let us first prove the following proposition which says that, for any $s\in\situations$, $\upprevvovkk(\cdot \vert s)$ satisfies an extended version of the coherence axioms.
We use the convention that $+\infty -\infty=-\infty+\infty=+\infty$.
This is a typical choice when working with upper expectations, see \cite{shafer2011levy} and \cite{DECOOMAN201618} where they use the dual convention for lower expectations.
If we would assume that $+\infty -\infty=-\infty+\infty=-\infty$, then the subadditivity property \ref{vovk coherence 2} below would for example not hold in general.
We will henceforth use this convention without mentioning it explicitly.
So, for example $a \geq b$ implies that $a-b \geq 0$, but not necessarily $0 \geq b-a$ for any two $a$ and $b$ in $\extreals$.
Moreover, we also assume that $c -\infty=-\infty$, $-\infty-\infty=-\infty$, $\lambda \cdot (-\infty)= -\infty$, $(- \lambda) \cdot (+\infty)= -\infty$ and $0 \cdot (-\infty)=0$ for all real $c$ and all $\lambda\in\posreals{}\cup \{+\infty\}$.

%   sets the upper bound of $-\infty$ for a whole range of variables.
% This implies that the upper expectation of these variables is automatically determined and restricted to $-\infty$.
% On the other hand, when using our convention, the upper bound for the upper expectation of the same variables is equal to $+\infty$ and therefore implies no constraints on these variables.
% Hence, our convention results in a more general framework.
% Moreover, note that we would want the converse convention if we would work with lower expectations.

\begin{proposition}
For all extended real variables $f,g\in\setofextvariables$, all $\lambda\in\nnegreals{}$, all $\mu\in\reals{}$ and all situations $s\in\situations$, $\upprevvovkk$ satisfies
\begin{enumerate}[leftmargin=*,ref={\upshape V}\arabic*,label={\upshape V\arabic*.}, series=sepcoherence ]
\item \label{vovk coherence 1} 
$\upprevvovkk(f \vert s)\leq\sup_{\omega \in \Gamma(s)} f(\omega)$; 
\item \label{vovk coherence 2} 
$\upprevvovkk(f+g \vert s)\leq\upprevvovkk(f \vert s) + \upprevvovkk(g \vert s)$;
\item \label{vovk coherence 3}
$\upprevvovkk(\lambda f \vert s)=\lambda \upprevvovkk(f \vert s)$.
\item \label{vovk coherence 4}
$f \leq_s g \Rightarrow \upprevvovkk(f \vert s)\leq\upprevvovkk(g \vert s)$;
\item \label{vovk coherence 5}
$\inf_{\omega \in \Gamma(s)} f(\omega)\leq\lowprevvovkk(f \vert s)\leq\upprevvovkk(f \vert s)\leq\sup_{\omega \in \Gamma(s)} f(\omega)$;
\item \label{vovk coherence 6}
$\upprevvovkk(f + \mu \vert s)=\upprevvovkk(f \vert s) + \mu$.
\end{enumerate}
\end{proposition}
\begin{proof}
Our proof is very similar to that of \cite[Prop. 14]{DECOOMAN201618}: We adapt it here to the fact that our supermartingales take values in $\reals{} \cup \{+\infty\}$ rather than $\reals{}$.

\ref{vovk coherence 1}.
If $\sup_{\omega\in\Gamma(s)} f(\omega)=+\infty$, the inequality is trivially satisfied. 
If this is not the case, consider any real $M \geq \sup_{\omega\in\Gamma(s)} f(\omega)$ and the real process $\martingale$ that assumes the constant value $M$.
Then clearly $\martingale$ is a supermartingale and moreover $\liminf\martingale(\omega)=M \geq f(\omega)$ for all $\omega\in\Gamma(s)$.
Hence, Definition \ref{def:upperexpectation2} guarantees that $\upprevvovkk(f \vert s)\leq\martingale(s)=M$.
Since this is true for every $M \geq \sup_{\omega\in\Gamma(s)} f(\omega)$, \ref{vovk coherence 1} follows.

\ref{vovk coherence 2}.
If either $\upprevvovkk(f \vert s)$ or $\upprevvovkk(g \vert s)$ equals $+\infty$, then the inequality is trivially true.
So suppose that $\upprevvovkk(f \vert s) < +\infty$ and $\upprevvovkk(g \vert s) < +\infty$ and consider any real $c_1 > \upprevvovkk(f \vert s)$ and any real $c_2 > \upprevvovkk(g \vert s)$.
Then there are two supermartingales $\martingale_1$ and $\martingale_2$ such that $\martingale_1(s) \leq c_1$ and $\martingale_2(s) \leq c_2$ and moreover $\liminf\martingale_1 \geq_s f$ and $\liminf\martingale_2 \geq_s g$.
Now consider the extended real process $\martingale \coloneqq\martingale_1 + \martingale_2$.
Then $\martingale$ is a supermartingale because of Lemma~\ref{lemma:positive:countable:linear:combination} [which we can apply because $\martingale_1$ and $\martingale_2$ are both bounded below and hence have a common lower bound].
Moreover, we show that $\liminf (\martingale_1 + \martingale_2) \geq \liminf\martingale_1 + \liminf\martingale_2$ and therefore that $\liminf\martingale \geq_s f+g$, which, by Definition \ref{def:upperexpectation2}, implies that $\upprevvovkk(f+g \vert s)\leq\martingale(s) \leq c_1 + c_2$. 
Since this holds for any real $c_1 > \upprevvovkk(f \vert s)$ and any real $c_2 > \upprevvovkk(g \vert s)$, it follows that $\upprevvovkk(f+g \vert s)\leq\upprevvovkk(f \vert s) + \upprevvovkk(g \vert s)$.

So consider any $\omega\in\samplespace$ and any real $\alpha_1$ and $\alpha_2$ such that $\liminf\martingale_1(\omega) > \alpha_1$ and $\liminf\martingale_2(\omega) > \alpha_2$.
This is always possible because $\martingale_1$ and $\martingale_2$ are bounded below.
Then there are two natural numbers $N_1$ and $N_2$ such that $\martingale_1(\omega^{n_1}) \geq \alpha_1$ and $\martingale_2(\omega^{n_2}) \geq \alpha_2$ for all $n_1 \geq N_1$ and all $n_2 \geq N_2$.
Hence, we have that $\martingale_1(\omega^{n}) + \martingale_2(\omega^{n})  \geq \alpha_1 + \alpha_2$ for all $n \geq \max\{N_1,N_2\}$, implying that $\liminf (\martingale_1 + \martingale_2)(\omega) \geq \alpha_1 + \alpha_2$.
Since this holds for any real $\alpha_1$ and $\alpha_2$ such that $\liminf\martingale_1(\omega) > \alpha_1$ and $\liminf\martingale_2(\omega) > \alpha_2$, we indeed find that $\liminf (\martingale_1 + \martingale_2)(\omega) \geq \liminf\martingale_1(\omega) + \liminf\martingale_2(\omega)$.

\ref{vovk coherence 3}.
For $\lambda > 0$, it suffices to note that $\martingale$ is a supermartingale such that $\liminf\martingale \geq_s f$ if and only if $\lambda \martingale$ is a supermartingale such that $\liminf \lambda \martingale \geq_s \lambda f$.
If $\lambda=0$, then $\lambda \upprevvovkk(f \vert s)=0$ because $(+\infty) \cdot 0=(-\infty) \cdot 0=0$. 
To see that $\upprevvovkk(\lambda f \vert s)=0$, start by noting that $\lambda f=0$ and hence, because of \ref{vovk coherence 1}, $\upprevvovkk(\lambda f \vert s) \leq 0$.
That $\upprevvovkk(\lambda f \vert s) < 0$ is impossible, follows from Lemma~\ref{lemma: infima of supermartingales} and Definition \ref{def:upperexpectation2}.
Hence, we indeed have that $\upprevvovkk(\lambda f \vert s)=0$.

\ref{vovk coherence 4}.
For any $\martingale\in\setofextsupmartb{}$ such that $\liminf\martingale \geq_s g$, we also have that $\liminf\martingale \geq_s f$, and hence, by Definition \ref{def:upperexpectation2}, $\upprevvovkk(f \vert s)\leq\upprevvovkk(g \vert s)$.

\ref{vovk coherence 5}.
The first and third inequality follow trivially from \ref{vovk coherence 1} and the definition of the conjugate lower expectation $\lowprevvovkk$.
To prove the second inequality, assume \emph{ex absurdo} that $\lowprevvovkk(f \vert s) > \upprevvovkk(f \vert s)$. 
Then $0 > \upprevvovkk(f \vert s)-\lowprevvovkk(f \vert s)$ which, by \ref{vovk coherence 2} and the definition of the conjugate lower expectation $\lowprevvovkk$, implies that $0 > \upprevvovkk(f + (- f) \vert s)$. 
Since, according to our convention, the extended real variable $f + (-f)$ only assumes values in $0$ and $+\infty$, we have that $f + (-f) \geq 0$ and therefore, by \ref{vovk coherence 4} and \ref{vovk coherence 3}, that $\upprevvovkk(f + (- f) \vert s) \geq \upprevvovkk(0 \vert s)=0$. 
This is a contradiction.

\ref{vovk coherence 6}.
For any $\martingale\in\setofextsupmartb{}$ such that $\liminf\martingale \geq_s f + \mu$, we have that $\martingale-\mu\in\setofextsupmartb{}$ because of \ref{coherence: const add} and moreover $\liminf (\martingale-\mu) \geq_s f$.
Hence, $\upprevvovkk(f \vert s) + \mu\leq\martingale(s)-\mu + \mu=\martingale(s)$.
Since this holds for any $\martingale\in\setofextsupmartb{}$ such that $\liminf\martingale \geq_s f + \mu$, we have that $\upprevvovkk(f \vert s) + \mu\leq\upprevvovkk(f + \mu \vert s)$.
By applying this inequality to $f'=f+ \mu$ and $\mu'=-\mu$, we also find that $\upprevvovkk(f + \mu \vert s)-\mu\leq\upprevvovkk(f \vert s)$.
\end{proof}

\begin{proposition}
Consider two imprecise probability trees consisting of their local upper expectations $\lupprev{s}$ and $\lupprev{s}'$ and consider their corresponding global upper expectations $\upprevvovkk{}$ and $\overline{\mathrm{E}}^\prime_\mathrm{V}$.
If $\lupprev{s}(f) \leq \lupprev{s}'(f)$ for all $f \in \setofgenextvariablesb{}(\statespace{})$ and all $s \in \situations{}$, then also $\upprevvovkk{}(g \vert s) \leq \overline{\mathrm{E}}^\prime_\mathrm{V}(g \vert s)$ for all $g \in \setofextvariables{}$ and all $s \in \situations{}$.
\end{proposition}
\begin{proof}
Let $\setofextsupmartb{}$ and $\overline{\mathbb{M}}^\prime_\mathrm{b}$ be the sets of supermartingales associated with respectively $\lupprev{s}$ and $\lupprev{s}'$.
Then it is clear that, from the definition of a supermartingale, we have that $\overline{\mathbb{M}}^\prime_\mathrm{b} \subseteq \setofextsupmartb{}$.
It then follows immediately from Definition \ref{def:upperexpectation2} that $\upprevvovkk{}(g \vert s) \leq \overline{\mathrm{E}}^\prime_\mathrm{V}(g \vert s)$ for all $g \in \setofextvariables{}$ and all $s \in \situations{}$.
\end{proof}

\noindent
With any situation $\sit\in\situations$ and any $(n+1)$-measurable extended real variable $f$ that is bounded below, we now associate a local variable $f(\sit \cdot)$ defined by
$f(\sit \cdot)(\situa{n+1}) \coloneqq f(\situation{1}{n+1}) \text{ for all } \situa{n+1}\in\statespace$, and we then use $\lupprev{\sit}(f)$ to denote the local upper expectation $\lupprev{\sit}(f(\sit \cdot))$.
This allows us to formulate the following result, which shows that the game-theoretic upper expectation $\upprevvovkk$ is compatible with the local models $\lupprev{s}$.

\begin{proposition}\label{prop: global compatible with local}
Consider any situation $\sit\in\situations$ and any $(n+1)$-measurable extended real variable $f$ that is bounded below.
Then,
\begin{equation*}
		\upprevvovkk (f \vert \sit)=\lupprev{\sit}(f(\sit \cdot)).
\end{equation*}
\end{proposition}
\begin{proof}
The proof is similar to that of \cite[Corollary~3]{DECOOMAN201618}.
Consider any $\martingale\in\setofextsupmartb{}$ such that $\liminf\martingale \geq_{\sit} f$.
Then it follows from Lemma~\ref{lemma: infima of supermartingales} that, for all $x_{n+1}\in\statespace$,
\begin{align*}
\martingale(x_{1:n+1}) \geq\inf_{\omega\in\Gamma(x_{1:n+1})} \liminf\martingale(\omega) \geq\inf_{\omega\in\Gamma(x_{1:n+1})} f(\omega)=f(x_{1:n+1}).
\end{align*}
Hence, we have that $\martingale(\sit \cdot) \geq f(\sit \cdot)$, which implies by \ref{coherence: monotonicity} and the supermartingale character of $\martingale$ that 
\begin{equation*}
\martingale(\sit) \geq \lupprev{\sit}(\martingale(\sit \cdot)) \geq \lupprev{\sit}(f(\sit \cdot)).
\end{equation*}
Since this holds for any $\martingale\in\setofextsupmartb{}$ such that $\liminf\martingale \geq_{x_{1:n}} f$, it follows from Definition \ref{def:upperexpectation2} that $\upprevvovkk (f \vert \sit) \geq \lupprev{\sit}(f(\sit \cdot))$.
To see that the inequality is an equality, consider the extended real process $\martingale$ defined by $\martingale(s) \coloneqq\lupprev{\sit}(f(\sit \cdot))$ for all $s \not\sqsupset \sit$, and by $\martingale(s)=f(x_{1:n+1})$ for any $s\in\situations$ such that $s \sqsupseteq x_{1:n+1}$ for some $x_{n+1}\in\statespace$.
Then $\martingale$ is a supermartingale because on the one hand, it is bounded below because $f$ is bounded below and $\lupprev{\sit}$ satisfies~\ref{coherence: bounds}, and on the other hand, $\lupprev{\sit}(\martingale(\sit \cdot))=\lupprev{\sit}(f(\sit \cdot))=\martingale(\sit)$ and $\lupprev{s}(\martingale(s\andstate))=\martingale(s)$ for all $s \not= \sit$ because of \ref{coherence: bounds}.
It is moreover easy to see that $\liminf\martingale \geq_{\sit} f$ is guaranteed because $f$ is $(n+1)$-measurable. 
\end{proof}

\noindent
The next theorem shows that the game-theoretic upper expectation $\upprevvovkk$ satisfies a \emph{Law of Iterated Upper Expectations}.
For its proof, we require the following additional notation and terminology.
We write that $s \sqsubseteq t$, and say that $s$ \emph{precedes} $t$ or that $t$ \emph{follows} $s$, when every path that goes through $t$ also goes through $s$. 
When $s \sqsubseteq t$ and $ s \not= t$, we write $s \sqsubset t$. 
When neither $s \sqsubseteq t$ nor $t \sqsubseteq s$, we say that $s$ and $t$ are \emph{incomparable}.

\begin{theorem}\label{theorem: vovk iterated}
For any $f\in\setofextvariables$ and any $x_{1:n}\in\situations$, we have that
\begin{equation*}
\upprevvovkk(f \vert x_{1:n})=\upprevvovkk\Big(\upprevvovkk\left(f \vert x_{1:n} X_{n+1}\right) \Big\vert x_{1:n}\Big).
\end{equation*}
\end{theorem}
\begin{proof}
This was already proven in \cite[Theorem 16]{DECOOMAN201618} for a version of global upper expectations with real-valued supermartingales.
We here adapt it to our setting.
Fix any $f\in\setofextvariables$ and any $x_{1:n}\in\situations$.
We first show that $\upprevvovkk(\upprevvovkk(f \vert x_{1:n} X_{n+1}) \vert x_{1:n})\leq\upprevvovkk(f \vert x_{1:n})$.
If $\upprevvovkk(f \vert x_{1:n})=+\infty$, this is trivially satisfied.
If not, then for any fixed real $\alpha > \upprevvovkk(f \vert x_{1:n})$ there is a supermartingale $\martingale$ such that $\martingale(x_{1:n})\leq\alpha$ and $\liminf\martingale \geq_{x_{1:n}} f$.
Then it is clear that, for all $x_{n+1}\in\statespace$, $\liminf\martingale \geq_{x_{1:n+1}} f$, and hence $\upprevvovkk(f \vert x_{1:n+1})\leq\martingale(x_{1:n+1})$ by Definition \ref{def:upperexpectation2}.
Let $\martingale'$ be the process that is equal to $\martingale$ for all situations that precede $x_{1:n}$ or are incomparable with $x_{1:n}$, and that is equal to the constant $\martingale(x_{1:n+1})$ for all situations that follow $x_{1:n+1}$ for some $x_{n+1}\in\statespace$.
Clearly, $\martingale'$ is again a supermartingale and moreover $\upprevvovkk(f \vert x_{1:n} X_{n+1})\leq\martingale(x_{1:n} X_{n+1}) \leq_{x_{1:n}}  \liminf\martingale'$.
Hence, it follows from Definition \ref{def:upperexpectation2} that 
% \begin{align*}
$\upprevvovkk(\upprevvovkk(f \vert x_{1:n} X_{n+1}) \vert x_{1:n})\leq\martingale'(x_{1:n})=\martingale(x_{1:n})\leq\alpha.$
% \end{align*}
This holds for any real $\alpha > \upprevvovkk(f \vert x_{1:n})$ and therefore we indeed have that $\upprevvovkk(\upprevvovkk(f \vert x_{1:n} X_{n+1}) \vert x_{1:n})\leq\upprevvovkk(f \vert x_{1:n})$.

We now prove the other inequality.
Again, if $\upprevvovkk(\upprevvovkk(f \vert x_{1:n} X_{n+1}) \vert x_{1:n})=+\infty$ it trivially holds, so we can assume it to be real or equal to $-\infty$.
Fix any real $\alpha > \upprevvovkk(\upprevvovkk(f \vert x_{1:n} X_{n+1}) \vert x_{1:n})$ and any $\epsilon > 0$.
Then there must be a supermartingale $\martingale$ such that $\martingale(x_{1:n})\leq\alpha$ and $\liminf\martingale \geq_{x_{1:n}} \upprevvovkk(f \vert x_{1:n} X_{n+1})$.
Consider any such supermartingale.
Then for any $x_{n+1}\in\statespace$, we have that $\liminf\martingale \geq_{x_{1:n+1}} \upprevvovkk(f \vert x_{1:n+1})$, which by Lemma~\ref{lemma: infima of supermartingales} implies that $\martingale(x_{1:n+1}) \geq \upprevvovkk(f \vert x_{1:n+1})$.
Fix any $x_{n+1}\in\statespace$.
Then $\martingale(x_{1:n+1})$ is either real or equal to $+\infty$ because $\martingale$ is bounded below.
If $\martingale(x_{1:n+1})$ is real, then since $\martingale(x_{1:n+1}) \geq \upprevvovkk(f \vert x_{1:n+1})$, it follows from Definition \ref{def:upperexpectation2} that there is a supermartingale $\martingale_{x_{1:n+1}}$ such that $\martingale_{x_{1:n+1}}(x_{1:n+1})\leq\martingale(x_{1:n+1}) + \epsilon$ and $\liminf\martingale_{x_{1:n+1}} \geq_{x_{1:n+1}} f$.
If $\martingale(x_{1:n+1})$ is $+\infty$, let $\martingale_{x_{1:n+1}}$ be the constant supermartingale that is equal to $+\infty$ everywhere. 
So, for all $x_{n+1}\in\statespace$, we have found a supermartingale $\martingale_{x_{1:n+1}}$ such that $\martingale_{x_{1:n+1}}(x_{1:n+1})\leq\martingale(x_{1:n+1}) + \epsilon$ and $\liminf\martingale_{x_{1:n+1}} \geq_{x_{1:n+1}} f$.
Let $\martingale^\ast$ be the process that is equal to $\martingale + \epsilon$ for all situations that precede or are incomparable with $x_{1:n}$, and that is equal to $\martingale_{x_{1:n+1}}$ for all situations that follow $x_{1:n+1}$ for some $x_{n+1}\in\statespace$.
Then $\martingale^\ast$ is also a supermartingale.
Indeed, it is clearly bounded below because $\martingale{}$ and all $\martingale{}_{x_{1:n+1}}$ are bounded below and $\statespace{}$ is finite.
Furthermore, for any $x_{n+1}\in\statespace$, we have that $\martingale^\ast(x_{1:n+1})=\martingale_{x_{1:n+1}}(x_{1:n+1})\leq\martingale(x_{1:n+1}) + \epsilon$, implying that $\martingale^\ast(x_{1:n} \cdot)\leq\martingale(x_{1:n} \cdot) + \epsilon$ and therefore, by \ref{coherence: monotonicity} and \ref{coherence: const add}, that 
\begin{equation*}
\lupprev{x_{1:n}}(\martingale^\ast(x_{1:n} \cdot))\leq\lupprev{x_{1:n}}(\martingale(x_{1:n} \cdot) + \epsilon) = \lupprev{x_{1:n}}(\martingale(x_{1:n} \cdot))+ \epsilon \leq  \martingale(x_{1:n}) + \epsilon=\martingale^\ast(x_{1:n}).
\end{equation*}
Moreover, for all situations $s \not\sqsupseteq x_{1:n}$, we have by \ref{coherence: const add} that $\lupprev{s}(\martingale^\ast(s\andstate))=\lupprev{s}(\martingale(s\andstate) + \epsilon)=\lupprev{s}(\martingale(s\andstate))+ \epsilon\leq\martingale(s) + \epsilon=\martingale^\ast(s)$, and for all $s\in\situations$ such that  $s \sqsupseteq x_{1:n+1}$ for some $x_{n+1}\in\statespace$, we have that $\lupprev{s}(\martingale^\ast(s\andstate))=\lupprev{s}(\martingale_{x_{1:n+1}}(s\andstate))\leq\martingale_{x_{1:n+1}}(s)=\martingale^\ast(s)$.
All together, we have that $\lupprev{s}(\martingale^\ast(s\andstate))\leq\martingale^\ast(s)$ for all $s \in \situations{}$, implying, together with its bounded belowness, that $\martingale^\ast$ is indeed a supermartingale.
Also note that $\liminf\martingale^\ast \geq_{x_{1:n}} f$ and that $\martingale^\ast(x_{1:n})=\martingale(x_{1:n}) + \epsilon\leq\alpha + \epsilon$.
Hence, by Definition~\ref{def:upperexpectation2}, $\upprevvovkk(f \vert x_{1:n})\leq\alpha + \epsilon$.
Since this holds for any $\epsilon > 0$ and any real $\alpha > \upprevvovkk(\upprevvovkk(f \vert x_{1:n} X_{n+1}) \vert x_{1:n})$, we indeed have that $\upprevvovkk(f \vert x_{1:n})\leq\upprevvovkk(\upprevvovkk(f \vert x_{1:n} X_{n+1}) \vert x_{1:n})$.
\end{proof}

\noindent

\begin{corollary}\label{corollary: upprevvovk is supermartingale}
For any $f\in\setofextvariablesb$, the process $\process$, defined by $\process(s) \coloneqq\upprevvovkk(f \vert s)$ for all $s\in\situations$, is a supermartingale.
\end{corollary}
\begin{proof}
Consider any $f\in\setofextvariablesb$. 
Then $\process$ is bounded below because $f$ is bounded below and $\upprevvovkk$ satisfies~\ref{vovk coherence 5}.
Moreover, using the notation $\upprevvovkk(f \vert s\andstate)$ to denote the local variable that assumes the value $\upprevvovkk(f \vert sx)$ for any $x\in\statespace$, it follows from Proposition~\ref{prop: global compatible with local} and Theorem \ref{theorem: vovk iterated} that
\vspace{0.2cm}
\begin{align*}
\lupprev{x_{1:n}}(\process(x_{1:n} \cdot)) 
= \lupprev{x_{1:n}}(\upprevvovkk(f \vert x_{1:n} \cdot)) 
&= \upprevvovkk(\upprevvovkk(f \vert x_{1:n} X_{n+1}) \vert x_{1:n}) \\
&= \upprevvovkk(f \vert x_{1:n})=\process(x_{1:n}) \text{ for all } x_{1:n}\in\situations.
\end{align*}
Hence, $\process$ is indeed a supermartingale.
\end{proof}

\noindent
Introducing further terminology, for any $s\in\situations$, we say that a supermartingale $\martingale\in\setofextsupmartb{}$ is an \emph{$s$-test supermartingale} if it is non-negative and $\martingale(s)=1$.
If $s=\Box$, we simply say it is a test supermartingale.
For any $s\in\situations$, we say that an event $A \subseteq \Gamma(s)$ is \emph{strictly almost sure (s.a.s.)} within $\Gamma(s)$ if there is an $s$-test supermartingale that converges to $+\infty$ on $\Gamma(s) \setminus A$.
Again, if $s=\Box$, we drop the `within' and simply speak of `strictly almost sure'.
For any two $f,g\in\setofextvariables$, we will then use the notation $f \geq_s g$ s.a.s. --- and similarly for $\leq_s$, $>_s$ and $<_s$ --- to mean that the event $\{\omega\in\Gamma(s) \colon f(\omega) \geq g(\omega)\}$ is strictly almost sure within $\Gamma(s)$.
% The following lemma was mentioned without proof by Shafer et. al. in \cite{shafer2011levy}.
% We explicitly prove it here. 

\begin{lemma}\label{lemmaMonotoneConvergenceSAS}
Consider any $f\in\setofextvariables$ and any $s\in\situations$.
Then
\begin{align}\label{Eq: lemmaMonotoneConvergenceSAS 1}
	\upprevvovkk(f \vert s) &=\inf \Big\{ \martingale(s) :  \martingale\in\setofextsupmartb \text{ and } \liminf\martingale \geq_s f \text{ s.a.s.} \Big\}. 
\end{align}
\end{lemma}
\begin{proof}
Since every supermartingale $\martingale$ that satisfies $\liminf\martingale \geq_s f$ also satisfies $\liminf\martingale \geq_s f$ s.a.s., we clearly have that
\begin{equation*}
\upprevvovkk(f \vert s) \geq\inf \Big\{ \martingale(s) \colon  \martingale\in\setofextsupmartb \text{ and } \liminf\martingale \geq_s f \text{ s.a.s.} \Big\},
\end{equation*}
so it remains to prove the other inequality.
If the right hand side of Equation \eqref{Eq: lemmaMonotoneConvergenceSAS 1} is equal to $+\infty$ this inequality is trivially satisfied, so consider the case where it is not.
Fix any $\alpha\in\reals{}$ such that $\alpha >\inf \big\{\martingale(s) \colon \martingale\in\setofextsupmartb \text{ and } \liminf\martingale \geq_s f \text{ s.a.s.} \big\}$ and any $\epsilon > 0$.
Then there is some supermartingale $\martingale_\alpha$ such that $\liminf\martingale_\alpha \geq_s f \text{ s.a.s.}$ and 
\begin{equation}\label{Eq: lemmaMonotoneConvergenceSAS 2}
\martingale_\alpha(s)\leq\alpha.
\end{equation}
Since $\liminf\martingale_\alpha \geq_s f$ s.a.s., there is some $s$-test supermartingale $\martingale_\alpha^\ast$ that converges to $+\infty$ on $A \coloneqq\{ \omega\in\Gamma(s) \colon \liminf\martingale_\alpha(\omega) < f(\omega)\}$.
Consider the extended real process $\martingale_\alpha + \epsilon \martingale_\alpha^\ast$. 
This process is again a supermartingale because of Lemma~\ref{lemma:positive:countable:linear:combination} [which we can apply because $\martingale_\alpha$ and $\martingale_\alpha^\ast$ are both bounded below and hence have a common lower bound]. 
Since $\martingale_\alpha^\ast$ converges to $+\infty$ on $A$ and because $\martingale_\alpha$ is bounded below, we have $\liminf (\martingale_\alpha + \epsilon \martingale_\alpha^\ast)(\omega)=+\infty \geq f(\omega)$ for all $\omega\in A$. 
Moreover, for all $\omega\in\Gamma(s) \setminus A$, we also have that $\liminf (\martingale_\alpha + \epsilon \martingale_\alpha^\ast)(\omega) \geq f(\omega)$, because $\liminf\martingale_\alpha(\omega) \geq f(\omega)$ and because $\epsilon \martingale_\alpha^\ast$ is non-negative.
Hence $\liminf (\martingale_\alpha + \epsilon \martingale_\alpha^\ast) \geq_s f$, and consequently $\upprevvovkk(f \vert s) \leq (\martingale_\alpha + \epsilon \martingale_\alpha^\ast)(s)$.
It therefore follows from Equation \eqref{Eq: lemmaMonotoneConvergenceSAS 2} that
\begin{align*}
\upprevvovkk(f \vert s)
&\leq (\martingale_\alpha + \epsilon \martingale_\alpha^\ast)(s) 
= \martingale_\alpha(s) + \epsilon\leq\alpha + \epsilon.
\end{align*}
As this holds for any $\epsilon\in\posreals{}$, we have that $\upprevvovkk(f \vert s)\leq\alpha$, and since this is true for every $\alpha\in\reals{}$ such that $\alpha >\inf \big\{\martingale(s) \colon \martingale\in\setofextsupmartb \text{ and } \liminf\martingale \geq_s f \text{ s.a.s.} \big\}$, it follows that
\begin{align*}
\upprevvovkk(f \vert s)
\leq\inf \Big\{ \martingale(s) :  \martingale\in\setofextsupmartb \text{ and } \liminf\martingale \geq_s f \text{ s.a.s.} \Big\}.
\end{align*}
\end{proof}

\section{Doob's Convergence Theorem and L\'evy's Zero-one Law}\label{Sect: Doob and Levy}
% In the current section, we will prove two fundamental theorems for our framework: Doob's Convergence Theorem and L\'evy's Zero-one Law.
% Both were already proven by Shafer, Vovk and Takemura in \cite{Shafer:2005wx} and \cite{shafer2011levy} for a similar case.
% Moreover, a version of Doob's Convergence Theorem with supermartingales as real processes is proven in \cite{DECOOMAN201618}.
% We here prove them for our setting; the methods used are similar to the ones used by Shafer, Vovk and Takemura.
% In our proofs, we will also need the concept of a cut.
A cut $U$ is collection of pair-wise incomparable situations. 
For any two cuts $U$ and $V$, we can define the following sets of situations:
\begin{align*}
[U,V] \coloneqq\{s\in\situations : (\exists u\in U)(\exists v\in V)u \sqsubseteq s \sqsubseteq v \}, \\
[U,V) \coloneqq\{s\in\situations : (\exists u\in U)(\exists v\in V)u \sqsubseteq s \sqsubset v \}, \\
(U,V] \coloneqq\{s\in\situations : (\exists u\in U)(\exists v\in V)u \sqsubset s \sqsubseteq v \}, \\
(U,V) \coloneqq\{s\in\situations : (\exists u\in U)(\exists v\in V)u \sqsubset s \sqsubset v \}.
\end{align*}
We call a cut $U$ \emph{complete} if for all $\omega\in\Omega$ there is some $u\in U$ such that $\omega\in\Gamma(u)$. 
Otherwise, we call $U$ \emph{partial}. 
We will also use the simpler notation $s$ to denote the cut $\{s\}$ that consists of the single situation $s\in\situations$. 
In this way we can define $[U,s]$, $[U,s)$, $[s,V]$, ... in a similar way as above.
We also write $U \sqsubset V$ if $(\forall v\in V)(\exists u\in U)u \sqsubset v$.
Analogously as before, we say that a path $\omega\in\samplespace$ goes through a cut $U$ when there is some $n\in\natz$ such that $\omega^n\in U$.

% \begin{theorem}\label{Theorem: Doob}
% Consider any supermartingale $\martingale\in\setofextsupmartb{}$. 
% If $\martingale(s)$ is real for some $s\in\situations$, then $\martingale$ converges to a real number strictly almost surely within $\Gamma(s)$.
% \end{theorem}
\begin{proposition}\label{Prop: Doob}
Consider any supermartingale $\martingale\in\setofextsupmartb{}$.
If $\martingale(t)$ is real for some $t\in\situations$, then there is a $t$-test supermartingale $\martingale^\ast$ that converges to $+\infty$ on all paths $\omega\in\Gamma(t)$ where $\martingale$ does not converge to an extended real number.
Moreover, if $\martingale$ converges to a real number on some path $\omega\in\Gamma(t)$, then $\martingale^\ast$ converges to an extended real number on $\omega$.
\end{proposition}
\begin{proof}
% Note that the situation $s\in\situations$, where $\martingale(s)$ is real, has no absolute meaning.
% If the theorem holds for some situation, it holds for any situation.
% For the sake of simplicity, we consider the case where $\martingale(\Box)$ is real.

% Since $\martingale$ is bounded below by definition, we can assume without loss of generality that $\martingale$ is non-negative.
% Indeed, it suffices to add a sufficiently large constant to the initially considered supermartingale.

% We have to prove that there is a non-negative supermartingale $\martingale^\ast\in\setofextsupmartb{}$ that converges on all paths $\omega\in\samplespace$ for which either $\lim \martingale(\omega)=+\infty$ or either $\liminf\martingale(\omega) < \limsup\martingale(\omega)$.
% The supermartingale $\martingale$ can not converge to $-\infty$ because it is positive.
We can assume that $\martingale$ is non-negative and that $\martingale(t)=1$ without loss of generality.
Indeed, because the original supermartingale is bounded below and real in $t$, we can obtain such a process by translating and scaling --- by adding a positive constant and then multiplying the supermartingale by a positive real --- the originally considered supermartingale in an appropriate way.
This process will then again be a supermartingale because of Lemma~\ref{lemma:positive:countable:linear:combination}.
Moreover, the new supermartingale will have the same convergence character as the original one.

To start, we associate with any couple of rational numbers $0<a<b$ the following recursively constructed sequences of cuts $\{U_k^{a,b}\}_{k\in\nats}$ and $\{V_k^{a,b}\}_{k\in\nats}$. 
	Let $V_1^{a,b} \coloneqq\{ s \sqsupseteq t \colon   \ \martingale(s) < a \text{ and } (\forall s' \sqsubset s) \ \martingale(s') \geq a \}$
	and, for $k\in\nats$,
	\begin{enumerate}
		\item let
		\begin{align*}
			U_k^{a,b} \coloneqq\{s\in\situations \colon V_{k}^{a,b} \sqsubset s : \martingale(s) > b \text{ and } (\forall s'\in (V_{k}^{a,b} , s)) \ \martingale(s') \leq b \};
		\end{align*}
		\item 
		let
		\begin{align*}
			V_{k+1}^{a,b} \coloneqq\{ s\in\situations \colon  U_{k}^{a,b} \sqsubset s , \ \martingale(s) < a \text{ and } (\forall s'\in (U_{k}^{a,b} , s)) \ \martingale(s') \geq a \};
		\end{align*} 
	\end{enumerate}
	The cuts $U_k^{a,b}$ and $V_k^{a,b}$ can be partial or complete.
	
	Next, consider the extended real process $\martingale^{a,b}$ defined by $\martingale^{a,b}(s) \coloneqq\martingale(t)$ for all $s \not\sqsupset t$ and 
	\begin{align}\label{Eq: proof Doob: def M^a,b}
	 	\martingale^{a,b} (s\andstate) \coloneqq\begin{cases}
	 		\martingale^{a,b}(s) + \left[\martingale(s\andstate) -\martingale(s) \right]   &\text{ if } s\in  [V_{k}^{a,b} , U_{k}^{a,b}) \text{ for some } k\in\nats;\\
	 		\martingale^{a,b}(s) &\text{ otherwise,}
	 		\end{cases}
	 \end{align}
	 for all $s \sqsupset t$.
	 We prove that this process is a non-negative supermartingale that converges to $+\infty$ on all paths $\omega\in\Gamma(t)$ such that
	 \begin{equation}\label{Eq: doob conv to infty}
	 \liminf\martingale(\omega) < a < b < \limsup\martingale(\omega).
	 \end{equation}

	 \noindent
	 For any situation~$s$ and for any $k\in\nats$, when $U_k^{a,b} \sqsubset s$, we denote by $u_k^s$ the (necessarily unique) situation in $U_k^{a,b}$ such that $u_k^s \sqsubset s$. 
	 Similarly, for any $k\in\nats$, when $V_k^{a,b} \sqsubset s$, we denote by $v_k^s$ the (necessarily unique) situation in $V_k^{a,b}$such that $v_k^s \sqsubset s$.
	 Note that $ V_1^{a,b} \sqsubset U_1^{a,b} \sqsubset V_2^{a,b} \sqsubset \cdots \sqsubset V_n^{a,b} \sqsubset U_n^{a,b} \sqsubset \cdots $.
	 Hence, for any situation~$s$ we can distinguish the following three cases:
	 \begin{itemize}
	 \item
	 The first case is that $V_1^{a,b} \not\sqsubset s$. Then we have that
	 \begin{equation}\label{proof Doob: first case}
	 \martingale^{a,b} (s)=\martingale^{a,b} (t)=\martingale(t).
	 \end{equation}
	  \item
	 The second case is that $V_k^{a,b} \sqsubset s$ and $U_{k}^{a,b} \not\sqsubset s$ for some $k\in\nats$. 
	 Then by applying Equation \eqref{Eq: proof Doob: def M^a,b} for each subsequent step and cancelling out the intermediate terms, which is possible because $\martingale$ is real for any situation $s'\in\situations$ such that $V_{k'}^{a,b} \sqsubseteq s'$ and $U_{k'}^{a,b} \not\sqsubseteq s'$ for some $k'\in\nats$ (this follows readily from the definition of the cuts $V_{k'}^{a,b}$ and $U_{k'}^{a,b}$), we have that
	 \begin{equation}\label{proof Doob: second case}
	 \martingale^{a,b} (s)- \martingale^{a,b} (t) 
	= \sum_{\ell=1}^{k-1} 
	 \left[ \martingale(u_\ell^{s})-\martingale(v_{\ell}^{s}) \right] + \martingale(s)-\martingale(v_{k}^{s}).
	 \end{equation} 

	 % Since $\martingale(s) \geq M > 0$ and $0 < b < \martingale(u_\ell^{s})$ and $0 < \martingale(v_\ell^{s}) < a$ for all $\ell\in\{ 1,...,k\}$, we get that
	 % \begin{equation*}
	 % \martingale^{a,b} (s) > \Big( \frac{b}{a} \Big)^{k-1}  \frac{\martingale(s)}{a} \geq \Big( \frac{b}{a} \Big)^{k-1} \Big( \frac{M}{a} \Big). 
	 % \end{equation*} 
	 \item
	 The third case is that $U_k^{a,b} \sqsubset s$ and $V_{k+1}^{a,b} \not\sqsubset s$ for some $k\in\nats$. Then we have that
	 \begin{equation}\label{proof Doob: third case}
	 \martingale^{a,b} (s)- \martingale^{a,b} (t) 
	=\sum_{\ell=1}^{k} [\martingale(u_\ell^{s})-\martingale(v_{\ell}^{s})],
	 \end{equation} 
	 where, again, we used the fact that $\martingale$ is real for any situation $s'\in\situations$ such that $V_{k'}^{a,b} \sqsubseteq s'$ and $U_{k'}^{a,b} \not\sqsubseteq s'$ for some $k'\in\nats$.
	 % Again, as $0 < b < \martingale_{\ell-1}^{a,b}(u_\ell^{s})$ and $0 < \martingale_{\ell}^{a,b}(v_\ell^{s}) < a$ for all $\ell\in\{ 1,...,k\}$, we get that
	 % \begin{equation*}
	 % \martingale^{a,b} (s) > \Big( \frac{b}{a} \Big)^{k}. 
	 % \end{equation*} 
	\end{itemize}

	\noindent
	That $\martingale^{a,b}(s)$ is non-negative, is trivially satisfied in the first case.
	To see that this is also true for the third case, observe that $0 < b < \martingale(u_\ell^{s})$ and $0\leq\martingale(v_\ell^{s}) < a$ for all $\ell\in\{ 1,...,k\}$.
	This implies that $\martingale(u_\ell^{s})-\martingale(v_\ell^{s}) > b-a > 0$ for all $\ell\in\{ 1,...,k\}$ and therefore directly that $\martingale^{a,b}(s)$ is non-negative because of Equation \eqref{proof Doob: third case}.
	In the second case, it follows from Equations \eqref{proof Doob: first case}, \eqref{proof Doob: second case} and \eqref{proof Doob: third case} that
	\begin{align}\label{proof Doob: eq martingale differences}
	\martingale^{a,b} (s) 
	=\martingale^{a,b} (v_{k}^{s}) + \martingale(s)-\martingale(v_{k}^{s}).
	\end{align}
	We prove by induction that $\martingale^{a,b} (v_{\ell}^{s}) \geq \martingale(v_{\ell}^{s})$ for all $\ell\in\{ 1,...,k\}$, and therefore, by Equation \eqref{proof Doob: eq martingale differences} and because $\martingale(s)$ is non-negative, that $\martingale^{a,b} (s)$ is non-negative.

	If $\ell=1$, then either $v_1^s=t$ or $v_1^s \not= t$.
	If $v_1^s=t$, then $\martingale^{a,b}(v_1^s)=\martingale(v_1^s)$ due to Equation \eqref{proof Doob: first case}.
	If $v_1^s \not= t$, we have, by the definition of $V_1^{a,b}$, that $\martingale(v_1^s) < a$ and $a\leq\martingale(t)=\martingale^{a,b}(v_1^s)$.
	Hence, in both cases, we have that $\martingale^{a,b}(v_1^s) \geq \martingale(v_1^s)$.
	Now suppose that $\martingale^{a,b} (v_{\ell}^{s}) \geq \martingale(v_{\ell}^{s})$ for some $\ell\in\{ 1,...,k-1\}$.
	Then, $\martingale^{a,b} (v_{\ell+1}^{s})=\martingale^{a,b} (v_{\ell}^{s}) + [\martingale(u_{\ell}^{s})-\martingale(v_{\ell}^{s})] \geq \martingale(u_{\ell}^{s}) \geq \martingale(v_{\ell+1}^{s})$, which concludes our induction step.
	So indeed $\martingale^{a,b} (v_{\ell}^{s}) \geq \martingale(v_{\ell}^{s})$ for all $\ell\in\{ 1,...,k\}$.

	Next, we show that $\lupprev{s}(\martingale^{a,b}(s\andstate))\leq\martingale^{a,b}(s)$ for all $s\in\situations$, and hence, combined with its non-negativity --- and therefore bounded belowness --- that $\martingale$ is a non-negative supermartingale.
	Consider any $s\in\situations$.
	If $s\in  [V_{k}^{a,b} , U_{k}^{a,b})$ for some $k\in\nats$, we have, by the definitions of $V_{k}^{a,b}$ and $U_{k}^{a,b}$ and the non-negativity of $\martingale$, that $\martingale(s)$ is real and therefore that
	 \begin{align*}
	 \lupprev{s}(\martingale^{a,b}(s\andstate)) 
	=\lupprev{s}(\martingale^{a,b}(s) + \martingale(s\andstate)-\martingale(s)) 
	 &\overset{\text{\ref{coherence: const add}}}{=} \lupprev{s}(\martingale(s\andstate)) + \martingale^{a,b}(s)-\martingale(s)  \\
	 &\leq \martingale^{a,b}(s),
	 \end{align*}
	 where the last step follows because $\martingale$ is a supermartingale.
	 Otherwise, if $s \not\in [V_{k}^{a,b} , U_{k}^{a,b})$ for all $k\in\nats$, we have that $\lupprev{s}(\martingale^{a,b}(s\andstate))=\lupprev{s}(\martingale^{a,b}(s))=\martingale^{a,b}(s)$, where we have used \ref{coherence: bounds} for the last inequality.
	 Hence, we have that $\lupprev{s}(\martingale^{a,b}(s\andstate))\leq\martingale^{a,b}(s)$ for all $s\in\situations$, and we can therefore infer that $\martingale^{a,b}$ is indeed a non-negative supermartingale.

	 Let us now show that $\martingale^{a,b}$ converges to $+\infty$ on all paths $\omega\in\Gamma(t)$ for which Equation \eqref{Eq: doob conv to infty} holds.
	 Consider such a path $\omega$.
	 % Then $\omega$ goes through all the cuts $V_1^{a,b} \sqsubset U_1^{a,b} \sqsubset V_2^{a,b} \sqsubset \cdots \sqsubset V_n^{a,b} \sqsubset U_n^{a,b} \sqsubset \cdots $. 
	 First, it follows from $\liminf\martingale(\omega) < a$ that there exists some $n_1\in\nats$ such that $\omega^{n_1} \sqsupseteq t$ and $\martingale(\omega^{n_1}) < a$. 
	 Take the first such $n_1$. 
	 Then it follows from the definition of $V_1^{a,b}$ that $\omega^{n_1}\in V_1^{a,b}$. 
	 Next, it follows from $\limsup\martingale (\omega) > b$ that there exists some $m_1\in\nats$ for which $m_1 > n_1$ and $\martingale(\omega^{m_1}) > b$. 
	 Take the first such $m_1$, then it follows from the definition of $U_1^{a,b}$ that $\omega^{m_1}\in U_1^{a,b}$. 
	 Repeating similar arguments over and over again allows us to conclude that $\omega$ goes through all the cuts $V_1^{a,b} \sqsubset U_1^{a,b} \sqsubset V_2^{a,b} \sqsubset \cdots \sqsubset V_n^{a,b} \sqsubset U_n^{a,b} \sqsubset \cdots $.
	 To see that $\martingale^{a,b}- \martingale^{a,b} (t)$ --- and therefore, since $\martingale^{a,b}(t)=\martingale(t)=1$, also $\martingale^{a,b}$ --- converges to $+\infty$ on $\omega$, note that the right hand side in Equation \eqref{proof Doob: second case} is bounded below by $(k-1)(b-a)-a$.
	 Similarly, the right hand side in Equation \eqref{proof Doob: third case} is bounded below by $k(b-a)$.
	 Hence, in both cases $\martingale^{a,b}- \martingale^{a,b} (t)$ goes to $+\infty$ because $k$ goes to $+\infty$.

	To finish, we use the countable set of rational couples $K \coloneqq\{ (a,b)\in\mathbb{Q}^2 : 0<a<b \}$ to define the process $\martingale^\ast$:
	\begin{equation*}
		\martingale^\ast \coloneqq\sum_{(a,b)\in K} w^{a,b} \martingale^{a,b},
	\end{equation*}
	with coefficients $w^{a,b}>0$ that sum to $1$. 
	Hence, $\martingale^\ast$ is a countable convex combination of the non-negative supermartingales $\martingale^{a,b}$. 
	By Lemma~\ref{lemma:positive:countable:linear:combination}, $\martingale^\ast$ is then a non-negative supermartingale.
	It is moreover clear that $\martingale^\ast(t)=\martingale(t)=1$, implying, together with its non-negativity, that $\martingale^\ast$ is a $t$-test supermartingale.
	We show that $\martingale^\ast$ converges in the desired way as described by the proposition.

	If $\martingale$ does not converge to an extended real number on some path $\omega\in\Gamma(t)$, then we have that $\liminf\martingale(\omega) < \limsup\martingale(\omega)$.
	Since $\liminf\martingale(\omega) \geq\inf_{s\in\situations} \martingale(s) \geq 0$, there is at least one couple $(a',b')\in K$ such that $\liminf\martingale(\omega) < a' < b' < \limsup\martingale(\omega)$, and as a consequence $\martingale^{a',b'}$ converges to  $+\infty$ on $\omega$. 
	Then also $\lim w^{a',b'} \martingale^{a',b'}(\omega)=+\infty$ since $w^{a',b'} > 0$. 
	For all other couples $(a,b)\in K \setminus \{(a',b')\}$, we have that $w^{a,b} \martingale^{a,b}$ is non-negative, so $\martingale^\ast$ indeed converges to $+\infty$ on $\omega$.

	Finally, we show that $\martingale^\ast$ converges in $\extreals$ on every path $\omega\in\Gamma(t)$ where $\martingale$ converges to a real number.
	Fix any such $\omega\in\Gamma(t)$.
	Suppose there is an $n\in\nats$ such that $\omega^n \sqsupset t$ and $\martingale(\omega^n)=+\infty$.
	Consider the first such $n$.
	Hence, $\martingale(\omega^{n-1})$ is real if $\omega^{n-1} \sqsupset t$.
	It is also real if $\omega^{n-1}=t$ because $\martingale(t)$ is real, implying that $\martingale(\omega^{n-1})$ is always real.
	Consider now a couple $(a,b)\in K$ such that $\martingale(\omega^{n-1}) < a < b$ (since $\martingale(\omega^{n-1})$ is real, there is at least one).
	It then follows from the definitions of $U_k^{a,b}$ and $V_k^{a,b}$ that $V_k^{a,b} \sqsubseteq \omega^{n-1}$ and $U_k^{a,b} \not\sqsubseteq \omega^{n-1}$ for some $k\in\nats$.
	This implies by the definition of $\martingale^{a,b}$ that $\martingale^{a,b}(\omega^n)=\martingale^{a,b}(\omega^{n-1}) + \martingale(\omega^n)-\martingale(\omega^{n-1})$, which in turn implies that $\martingale^{a,b}(\omega^n)=+\infty$ because $\martingale(\omega^n)=+\infty$, $\martingale^{a,b}(\omega^{n-1}) \geq 0$ and $\martingale(\omega^{n-1})$ is real.
	Then it follows from the definition of $\martingale^{a,b}$ and our convention that $+\infty-\infty=+\infty$, that $\martingale^{a,b}(\omega^m)=+\infty$ for all $m \geq n$.
	Hence, because $w^{a,b}$ is positive, it follows from the definition of $\martingale^\ast$ that $\martingale^\ast(\omega^m)= +\infty$ for all $m \geq n$ and as a consequence, $\martingale^\ast$ converges in $\extreals$.

	So, consider the case where $\martingale(\omega^n)$ is real for all $n\in\natz$ such that $\omega^n \sqsupset t$ (the case where $\martingale(\omega^n)=-\infty$ is impossible because $\martingale$ is bounded below).
	Since $\martingale^{a,b}(s)=\martingale(t)=1$ for all $s \not\sqsupset t$, it then follows from Equation \eqref{Eq: proof Doob: def M^a,b} that $\martingale^{a,b}(\omega^n)$ is real for all $n\in\natz$ and all $(a,b)\in K$.
	Moreover, $\martingale^\ast(\omega^n)$ is then also real for all $n\in\natz$ such that $\omega^n \sqsupseteq t$.
	Indeed, for all $n\in\natz$ such that $\omega^n \sqsupseteq t$, $\martingale(\omega^n)$ are real numbers that converge in $\reals{}$ as $n$ approaches infinity, so $M \coloneqq\sup\{\martingale(\omega^n) \colon n\in\natz, \omega^n \sqsupseteq t\}$ is real, and, since $\martingale$ is non-negative, $M \geq 0$.
	If we note that, according to Equation \eqref{Eq: proof Doob: def M^a,b} and the non-negativity of $\martingale$, 
	\begin{align*}
	\martingale^{a,b}(\omega^{n+1})-\martingale^{a,b}(\omega^{n}) 
	\leq \max\{0, \martingale(\omega^{n+1})-\martingale(\omega^n) \} 
	&\leq \max \{0, \martingale(\omega^{n+1}\} \\
	&\leq \max\{0,M\} \leq M,
	\end{align*}
	for any $n\in\natz$ such that $\omega^n \sqsupseteq t$, then we infer that  
	\begin{align*}
	\martingale^\ast(\omega^n)=\sum_{(a,b)\in K} w^{a,b} \martingale^{a,b}(\omega^n)
	&\leq \sum_{(a,b)\in K} w^{a,b} \left(n M + \martingale(t) \right) \\
	&= n  M + 1,
	\end{align*}
	for all $n\in\natz$ such that $\omega^n \sqsupseteq t$.
	The last expression is clearly real because $M$ is real.
	Together with the non-negativity of $\martingale^\ast$, this implies that $\martingale^\ast(\omega^n)$ is indeed real for all $n\in\natz$ such that $\omega^n \sqsupseteq t$.
	So we have that $\martingale(\omega^n)$, $\martingale^{a,b}(\omega^n)$ and $\martingale^\ast(\omega^n)$ are all real for all $n\in\natz$ such that $\omega^n \sqsupseteq t$ and all $(a,b)\in K$.
	Taking this into account, we now fix any such $n\in\natz$ and any $(a,b)\in K$, and show that
	\begin{align}\label{proof doob: eq infima}
	\inf_{\ell \geq n} [\martingale^{a,b}(\omega^\ell)-\martingale^{a,b}(\omega^n)] \geq\inf_{\ell \geq n} [\martingale(\omega^\ell)-\martingale(\omega^n)]. 
	\end{align}
	First note that, for any $k\in\nats$ and any $v_k\in V_k^{a,b}$, we have that
	\begin{equation}\label{eq: proof doob: 3}
	\martingale^{a,b}(s')\leq\martingale^{a,b}(v_k) \text{ for all } s' \sqsubseteq v_k.
	\end{equation}
	Indeed, if $U_{k'}^{a,b} \sqsubset s'$ and $V_{k'+1}^{a,b} \not\sqsubset s'$ for some $k' < k$, then by Equation \eqref{proof Doob: third case} and the fact that $\martingale(u_p^{v_k})-\martingale(v_p^{v_k}) > b-a > 0$ for all $p\in\{ 1,...,k-1\}$, we have that 
	\begin{align*}
	\martingale^{a,b} (s')- \martingale^{a,b} (t) 
	=\sum_{p=1}^{k'} [\martingale(u_p^{s'})-\martingale(v_{p}^{s'})] 
	 &\leq \sum_{p=1}^{k-1} [\martingale(u_p^{v_k})-\martingale(v_{p}^{v_k})] \\
	 &= \martingale^{a,b} (v_k)- \martingale^{a,b} (t),
	\end{align*}
	and therefore that $\martingale^{a,b}(s')\leq\martingale^{a,b}(v_k)$.
	If $V_{k'}^{a,b} \sqsubset s'$ and $U_{k'}^{a,b} \not\sqsubset s'$ for some $k' < k$, it follows from Equation \eqref{proof Doob: second case}, $\martingale(s')\leq\martingale(u_{k'}^{v_k})$ and again $\martingale(u_p^{v_k})-\martingale(v_p^{v_k}) > b-a > 0$ for all $p\in\{ 1,...,k-1\}$, that 
	\begin{align*}
	\martingale^{a,b} (s')- \martingale^{a,b} (t) 
	 &= \sum_{p=1}^{k'-1} [\martingale(u_p^{s'})-\martingale(v_{p}^{s'})] + \martingale(s')-\martingale(v_{k'}^{s'}) \\
	 &\leq \sum_{p=1}^{k-1} [\martingale(u_p^{v_k})-\martingale(v_{p}^{v_k})] \\
	 &= \martingale^{a,b} (v_k)- \martingale^{a,b} (t),
	\end{align*}
	again implying that $\martingale^{a,b}(s')\leq\martingale^{a,b}(v_k)$.
	Otherwise, if $V_{1}^{a,b} \not\sqsubset s'$, we have that that $\martingale^{a,b}(s')=\martingale^{a,b}(t)$ and hence, again by the fact that $\martingale(u_p^{v_k})-\martingale(v_p^{v_k}) > b-a > 0$ for all $p\in\{ 1,...,k-1\}$, we find that 
	$\martingale^{a,b}(v_k)-\martingale^{a,b}(t) 
		= \sum_{p=1}^{k-1} [\martingale(u_p^{v_k})-\martingale(v_{p}^{v_k})] \geq 0$,
	so $\martingale^{a,b}(s')\leq\martingale^{a,b}(v_k)$.
	All together, this shows that Equation \eqref{eq: proof doob: 3} holds in general.

	For the next part of the proof, for any $k\in\nats$, we use $u_{k}^{\omega}$ and $v_{k}^{\omega}$ to denote the respective situations in $U_{k}^{a,b}$ and $V_{k}^{a,b}$ such that $\omega$ passes through $u_{k}^{\omega}$ and $v_{k}^{\omega}$.
	Now consider any $\ell > n$ such that $\martingale^{a,b}(\omega^\ell)-\martingale^{a,b}(\omega^n) < 0$.
	Then Equation \eqref{eq: proof doob: 3} implies that $V_k^{a,b} \sqsubset \omega^\ell$ and $U_k^{a,b} \not\sqsubseteq \omega^\ell$ for some $k\in\nats$.
	Indeed, assume \emph{ex absurdo} that this is not the case.
	Then there are two remaining possibilities.
	The first is that $U_k^{a,b} \sqsubseteq \omega^\ell$ and $V_{k+1}^{a,b} \not\sqsubset \omega^\ell$ for some $k\in\nats$, implying by Equation \eqref{eq: proof doob: 3} combined with $\omega^n \sqsubset \omega^\ell \sqsubseteq v_{k+1}^{\omega}$ and the definition of $\martingale^{a,b}$, that $\martingale^{a,b}(\omega^\ell)=\martingale^{a,b}(v_{k+1}^{\omega}) \geq \martingale^{a,b}(\omega^n)$.
	The second remaining possibility is that $V_1^{a,b} \not\sqsubset \omega^\ell$, implying by definition of $\martingale^{a,b}$ that $\martingale^{a,b}(\omega^\ell)=\martingale^{a,b}(\omega^n)=\martingale^{a,b}(v_{1}^{\omega})$.
	Both cases contradict $\martingale^{a,b}(\omega^\ell)-\martingale^{a,b}(\omega^n) < 0$, so indeed we have that $V_k^{a,b} \sqsubset \omega^\ell$ and $U_k^{a,b} \not\sqsubseteq \omega^\ell$ for some $k\in\nats$.

	We separate three cases.
	If $V_{k'}^{a,b} \sqsubset \omega^n$ and $U_{k'}^{a,b} \not\sqsubset \omega^n$ for some $k' \leq k$, we have by Equation \eqref{proof Doob: second case} that 
	\begingroup
	\allowdisplaybreaks
	\begin{align*}
	\martingale^{a,b}(\omega^\ell)-\martingale^{a,b}(\omega^n)
	&= \martingale(\omega^\ell)-\martingale(v_k^\omega) + \sum_{p=1}^{ k-1} [\martingale(u_{p}^\omega)-\martingale(v_{p}^\omega)] \\
	&\qquad \quad-\martingale(\omega^n) + \martingale(v_{k'}^\omega)-\sum_{p=1}^{ k'-1} [\martingale(u_{p}^\omega)-\martingale(v_{p}^\omega)] \\
	&=  \martingale(\omega^\ell)  -\martingale(\omega^n) + \sum_{p=k'}^{ k-1} [\martingale(u_{p}^\omega)-\martingale(v_{p+1}^\omega)] \\
	&\geq \martingale(\omega^\ell)-\martingale(\omega^n) + (k-k')(b-a) 
	\geq \martingale(\omega^\ell)-\martingale(\omega^n).
	\end{align*}
	\endgroup
	If $U_{k'}^{a,b} \sqsubset \omega^n$ and $V_{k'+1}^{a,b} \not\sqsubset \omega^n$ for some $k' < k$, then we have that $\martingale(\omega^n) \geq \martingale(v_{k'+1}^{\omega})$, implying by Equation \eqref{proof Doob: second case} and \eqref{proof Doob: third case} that 
	\begin{align*}
	\martingale^{a,b}(\omega^\ell)-\martingale^{a,b}(\omega^n)
	&= \martingale(\omega^\ell)-\martingale(v_k^\omega) + \sum_{p=1}^{ k-1} [\martingale(u_{p}^\omega)-\martingale(v_{p}^\omega)] \\
	&\qquad \qquad \qquad-\sum_{p=1}^{ k'} [\martingale(u_{p}^\omega)-\martingale(v_{p}^\omega)] \\
	&=  \martingale(\omega^\ell)-\martingale(v_{k'+1}^{\omega}) + \sum_{p=(k'+1)}^{ k-1} [\martingale(u_{p}^\omega)-\martingale(v_{p+1}^\omega)] \\
	&\geq \martingale(\omega^\ell)-\martingale(v_{k'+1}^{\omega}) + (k-k'-1)(b-a) \\
	&\geq \martingale(\omega^\ell)-\martingale(\omega^n).
	\end{align*}
	%  definition of $\martingale^{a,b}$ that $\martingale^{a,b}(\omega^\ell)-\martingale^{a,b}(\omega^n)=\martingale(\omega^\ell)-\martingale(\omega^n)$.
	% If $V_k^{a,b} \not\sqsubseteq \omega^n$, meaning that $\omega^n$ precedes $v_k^{\omega}$, 
	% we either have that $U_{k'}^{a,b} \sqsubseteq \omega^n$ and $V_{k'+1}^{a,b} \not\sqsubseteq \omega^n$ for some $k' < k$, implying that $\martingale^{a,b}(\omega^n)=\martingale^{a,b}(u_{k'}^{\omega})=\martingale^{a,b}(v_{k'+1}^{\omega})\leq\martingale^{a,b}(v_{k}^{\omega})$, where the last inequality follows from Equation \eqref{proof Doob: third case}.
	% If not, we have that $V_{k'}^{a,b} \sqsubseteq \omega^n$ and $U_{k'}^{a,b} \not\sqsubseteq \omega^n$ for some $k' < k$, implying that $\martingale^{a,b}(\omega^n) < b\leq\martingale^{a,b}(u_{k-1}^{\omega})=\martingale^{a,b}(v_k^{\omega})$.
	% we have by Equation \eqref{eq: proof doob: 3} that 
	% \begin{align*}
	% \martingale^{a,b}(\omega^\ell)-\martingale^{a,b}(\omega^n)
	% &= \martingale^{a,b}(\omega^\ell)-\martingale^{a,b}(v_k^{\omega}) + \martingale^{a,b}(v_k^{\omega})-\martingale^{a,b}(\omega^n) \\
	% &\geq \martingale^{a,b}(\omega^\ell)-\martingale^{a,b}(v_k^{\omega}) \\
	% &= \martingale(\omega^\ell)-\martingale(v_k^{\omega}),
	% \end{align*}
	% where we used the fact that $\martingale^{a,b}$ is real on $\omega$ and where the last equality follows from Equation \eqref{proof Doob: eq martingale differences}.
	Finally, if $V_{1}^{a,b} \not\sqsubset \omega^n$, then we find in an analogous way as for the previous case (where $U_{k'}^{a,b} \sqsubset \omega^n$ and $V_{k'+1}^{a,b} \not\sqsubset \omega^n$ for some $k' < k$), that $\martingale^{a,b}(\omega^\ell)-\martingale^{a,b}(\omega^n) \geq \martingale(\omega^\ell)-\martingale(\omega^n)$.
	Hence, in all three cases we have that $\martingale^{a,b}(\omega^\ell)-\martingale^{a,b}(\omega^n) \geq \martingale(\omega^\ell)-\martingale(\omega^n)$.
	Since this holds for any $\ell > n$ such that $\martingale^{a,b}(\omega^\ell)-\martingale^{a,b}(\omega^n) < 0$ and since both infima in Equation \eqref{proof doob: eq infima} are obviously smaller or equal than zero (because the term for $\ell=n$ is zero), the desired inequality follows. 
	Since Equation \eqref{proof doob: eq infima} holds for any $(a,b)\in K$, we have that 
	\begin{align}\label{proof doob: eq infima 2}
	\sum_{(a,b)\in K} w^{a,b}\inf_{\ell \geq n} [\martingale^{a,b}(\omega^\ell)-\martingale^{a,b}(\omega^n)] 
	&\geq \sum_{(a,b)\in K} w^{a,b}\inf_{\ell \geq n} [\martingale(\omega^\ell)-\martingale(\omega^n)] \nonumber \\
	&=\inf_{\ell \geq n} [\martingale(\omega^\ell)-\martingale(\omega^n)].
	\end{align}
	Note that both sums above exist because all terms are smaller or equal than zero and the coefficients $w^{a,b}$ are larger than zero. 
	Moreover, $\inf_{\ell \geq n}  \left[ \martingale^\ast(\omega^\ell)-\martingale^\ast(\omega^n) \right]$ is on its turn equal to
	\begin{align*}
	\MoveEqLeft[2]
	\inf_{\ell \geq n} \left( \sum_{(a,b)\in K} w^{a,b} \martingale^{a,b}(\omega^\ell)-\sum_{(a,b)\in K} w^{a,b} \martingale^{a,b}(\omega^n) \right) \\
	&=\inf_{\ell \geq n} \sum_{(a,b)\in K}  [ w^{a,b} \martingale^{a,b}(\omega^\ell)-w^{a,b} \martingale^{a,b}(\omega^n)] \\
	&=\inf_{\ell \geq n} \sum_{(a,b)\in K} w^{a,b} [\martingale^{a,b}(\omega^\ell)-\martingale^{a,b}(\omega^n)] \\
	&\geq \sum_{(a,b)\in K} w^{a,b}\inf_{\ell \geq n} [\martingale^{a,b}(\omega^\ell)-\martingale^{a,b}(\omega^n)],
	\end{align*}
	where the second equality holds because both sums converge to a real number since $\martingale^\ast(\omega^\ell)$ and $\martingale^\ast(\omega^n)$ are both real.
	Combining this with Equation \eqref{proof doob: eq infima 2}, results in
	\begin{align*}
	\inf_{\ell \geq n}  \left[ \martingale^\ast(\omega^\ell)-\martingale^\ast(\omega^n) \right]
	 \geq\inf_{\ell \geq n} [\martingale(\omega^\ell)-\martingale(\omega^n)].
	\end{align*}
	Since this holds for any $n\in\natz$ such that $\omega^n \sqsupseteq t$, this implies that 
	\begin{align*}
	\liminf_{n\to+\infty}\inf_{\ell \geq n}  \left[ \martingale^\ast(\omega^\ell)-\martingale^\ast(\omega^n) \right]
	 \geq \liminf_{n\to+\infty}\inf_{\ell \geq n} [\martingale(\omega^\ell)-\martingale(\omega^n)].
	\end{align*}
	Now, since $\martingale$ converges to a real number on $\omega$, the right hand side is zero.
	Indeed, for any $\epsilon>0$, since $\martingale$ converges to a real number on $\omega$, there is an $N\in\natz$ such that $\vert \martingale(\omega^\ell)-\martingale(\omega^n) \vert\leq\epsilon$ for all $\ell$ and $n$ larger than $N$.
	Then $0 \geq\inf_{\ell \geq n} [\martingale(\omega^\ell)-\martingale(\omega^n)] \geq-\epsilon$ for all $n$ larger than $N$, and therefore also $0 \geq \liminf_{n\to+\infty}\inf_{\ell \geq n} [\martingale(\omega^\ell)-\martingale(\omega^n)] \geq-\epsilon$.
	Since this holds for any $\epsilon > 0$, we indeed find that $\liminf_{n\to+\infty}\inf_{\ell \geq n} [\martingale(\omega^\ell)-\martingale(\omega^n)]=0$.
	Hence, we have that $\liminf_{n\to+\infty}\inf_{\ell \geq n}  \left[ \martingale^\ast(\omega^\ell)-\martingale^\ast(\omega^n) \right]
	 \geq 0$.
	From this, it follows that $\martingale^\ast$ converges on $\omega$ to either a real number or $+\infty$. 
	Indeed, assume \emph{ex absurdo} that it does not.
	Then $\limsup\martingale^\ast(\omega)-\liminf\martingale^\ast(\omega) > \epsilon$ for some $\epsilon > 0$.
	Hence, then for any $N\in\natz$, there would exist an $n \geq N$ such that $\inf_{\ell \geq n}  \left[ \martingale^\ast(\omega^\ell)-\martingale^\ast(\omega^n) \right] <-\epsilon$, and therefore $\liminf_{n\to+\infty}\inf_{\ell \geq n}  \left[ \martingale^\ast(\omega^\ell)-\martingale^\ast(\omega^n) \right] <-\epsilon$, thereby contradicting $\liminf_{n\to+\infty}\inf_{\ell \geq n}  \left[ \martingale^\ast(\omega^\ell)-\martingale^\ast(\omega^n) \right]
	 \geq 0$. 
	We conclude that $\martingale^\ast$ is a $t$-test supermartingale that converges in the desired way as described by the proposition.
\end{proof}

\begin{theorem}[Doob's Convergence Theorem]\label{Theorem: Doob}
Consider any supermartingale $\martingale\in\setofextsupmartb{}$. 
If $\martingale(s)$ is real for some $s\in\situations$, then $\martingale$ converges to a real number strictly almost surely within $\Gamma(s)$.
\end{theorem}
\begin{proof}
As in the proof of Proposition~\ref{Prop: Doob}, we can assume without loss of generality that $\martingale$ is non-negative and that $\martingale(s)=1$.
Furthermore, we know from Proposition~\ref{Prop: Doob} that there is an $s$-test supermartingale $\martingale^\ast$ that converges to $+\infty$ on every path $\omega\in\Gamma(s)$ where $\martingale$ does not converge to an extended real number.
Now, consider the extended real process $\martingale' \coloneqq (\martingale + \martingale^\ast)/2$.
This process is a supermartingale because of Lemma~\ref{lemma:positive:countable:linear:combination}, and it is clearly non-negative because $\martingale$ and $\martingale^\ast$ are.
Moreover, we trivially have that $\martingale'(s)=1$ and therefore, that $\martingale'$ is an $s$-test supermartingale.
Furthermore, consider any path $\omega\in\Gamma(s)$ such that $\martingale(\omega^n)$ does not converge to a real number.
Then either it converges to $+\infty$ or it does not converge in $\extreals$.
In the first case, it follows from the non-negativity of $\martingale^\ast$ that $\martingale'$ also converges to $+\infty$ on $\omega$.
If $\martingale(\omega^n)$ does not converge in $\extreals$, then $\martingale^\ast$ converges to $+\infty$ on $\omega$ and therefore, because $\martingale$ is non-negative, $\martingale'$ also converges to $+\infty$ on $\omega$.
All together, we have that $\martingale'$ is an $s$-test supermartingale that converges to $+\infty$ on every path $\omega\in\Gamma(s)$ where $\martingale$ does not converge to a real number.
\end{proof}

\begin{proposition}\label{prop: liminf can be replaced by lim in definition}
For any $f\in\setofextvariables$ and any $s\in\situations$, we have that
\begin{align*}
\upprevvovkk(f \vert s) =\inf \left\{ \martingale(s) \colon \martingale\in\setofextsupmartb{} \text{ and } \lim \martingale \geq_s f \right\},
\end{align*}
where $\martingale$ is a supermartingale for which $\lim \martingale$ exists within $\Gamma(s)$.
\end{proposition}
\begin{proof}
The inequality `$\leq$' is trivially satisfied since $\liminf\martingale =_s \lim \martingale$ for any supermartingale $\martingale$ such that the limit $\lim \martingale$ exists within $\Gamma(s)$.
It remains to prove the other inequality.
If $\upprevvovkk(f \vert s)=+\infty$, this is trivially satisfied.
If $\upprevvovkk(f \vert s)$ is real or $-\infty$, we induce from Definition \ref{def:upperexpectation2} that, for any real $\alpha > \upprevvovkk(f \vert s)$, there is a supermartingale $\martingale\in\setofextsupmartb{}$ such that $\martingale(s)\leq\alpha$ and $\liminf\martingale \geq_s f$.
Hence, because $\martingale$ is bounded below and $\martingale(s)\leq\alpha$ we have that $\martingale(s)\in\reals{}$.
Then according to Proposition~\ref{Prop: Doob} there is an $s$-test supermartingale that converges to $+\infty$ on all paths $\omega\in\Gamma(s)$ where $\martingale$ does not converge in $\extreals$ and converges in $\extreals$ on all paths $\omega\in\Gamma(s)$ where $\martingale$ converges in $\reals{}$.

Fix any $\epsilon>0$ and consider the process $\martingale'$ that is equal to $\martingale(t) + \epsilon \martingale^\ast(t)$ for all situations $t \sqsupseteq s$ and equal to $\martingale(s) + \epsilon \martingale^\ast(s)\leq\alpha + \epsilon$ for all situations $t \not\sqsupseteq s$.
Then $\martingale'$ is a supermartingale because of Lemma~\ref{lemma:positive:countable:linear:combination} and we have that $\liminf\martingale' \geq_s f$ because $\epsilon \martingale^\ast$ is non-negative and $\liminf\martingale \geq_s f$.
Moreover, for all $\omega\in\Gamma(s)$, this process converges in $\extreals$.
Indeed, for any $\omega\in\Gamma(s)$, if $\martingale$ does not converge in $\extreals$, $\martingale^\ast$ converges to $+\infty$ and hence also $\martingale'$ because $\martingale$ is bounded below and $\epsilon$ is positive.
If $\martingale$ does converge in $\extreals$, it converges either to a real number or to $+\infty$ (convergence to $-\infty$ is impossible because it is bounded below).
If it converges to a real number, $\martingale^\ast$ converges in $\extreals$ and hence $\martingale'$ also converges in $\extreals$.
If it converges to $+\infty$, then so does $\martingale'$ because $\epsilon \martingale^\ast$ is non-negative.
Hence, $\martingale'$ converges in $\extreals$, so the limit $\lim \martingale'(\omega)$ exists for all $\omega\in\Gamma(s)$.
Moreover, recall that $\lim \martingale'=\liminf\martingale' \geq_s f$.
Hence, we have that
\begin{align*}
\inf \left\{ \martingale(s) \colon \martingale\in\setofextsupmartb{} \text{ and } \lim \martingale \geq_s f \right\} 
\leq \martingale'(s)  
\leq \alpha + \epsilon.
\end{align*}
This holds for any $\epsilon > 0$ and any $\alpha > \upprevvovkk(f \vert s)$, which implies that indeed
\begin{align*}
\inf \left\{ \martingale(s) \colon \martingale\in\setofextsupmartb{} \text{ and } \lim \martingale \geq_s f \right\}\leq\upprevvovkk(f \vert s).
\end{align*}
\end{proof}

\begin{theorem}[L\'evy's zero-one law]\label{Theorem: Levy}
	For any $f\in\setofextvariablesb$ and any $s'\in\situations$, the event
	\begin{align*}
		A \coloneqq\Big\{\omega\in\Omega \colon \liminf_{n\to+\infty} \upprevvovkk(f \vert \omega^n) \geq f(\omega) \Big\} \text{ is strictly almost sure within } \Gamma(s').
	\end{align*}
\end{theorem}
\begin{proof}
	 % The proof is heavily inspired by the work of Shafer and Vovk~\cite[Section 6.7]{shafer2011levy}. 
	 % However, they use supermartingales as extended real processes, whereas we use real processes. 
	 % Moreover, they give a direct proof for extended real gambles that are bounded below. 
	 % We will do this in a separate corollary.
	 Let $c\in\reals{}$ be any real constant.
	 Since, for any $s\in\situations$, $\upprevvovkk(\cdot \vert s)$ satisfies~\ref{vovk coherence 6}, we have that $\liminf_{n\to+\infty} \upprevvovkk(f \vert \omega^n) \geq f(\omega)$ if and only if $\liminf_{n\to+\infty} \upprevvovkk(f + c \vert \omega^n) \geq f(\omega) + c$.
	 Therefore, and because f is bounded below, we can assume without loss of generality that $f$ is a gamble such that $\inf f > 0$. 

	We now associate with any couple of rational numbers $0<a<b$ the following recursively constructed sequences of cuts $\{U_k^{a,b}\}_{k\in\natz}$ and $\{V_k^{a,b}\}_{k\in\nats}$. 
	Let $U_0^{a,b} \coloneqq\{ s' \}$ and, for $k\in\nats$,
	\begin{enumerate}
		\item 
		let
		\begin{align*}
			V_k^{a,b} \coloneqq\{ s\in\situations \colon  U_{k-1}^{a,b} \sqsubset s , \ \upprevvovkk(f \vert s) < a \text{ and } (\forall t\in (U_{k-1}^{a,b} , s)) \ \upprevvovkk(f \vert t) \geq a \};
		\end{align*} 
		\item if $V_k^{a,b}$ is non-empty, choose a positive supermartingale $\martingale_{k}^{a,b}\in\setofextsupmartb{}$ such that 
		$\martingale_{k}^{a,b}(s) < a$ and $\liminf\martingale_k^{a,b}  \geq_s f$ for all $s\in V_k^{a,b}$, and let
		\begin{align*}
			U_k^{a,b} \coloneqq\{s\in\situations \colon V_{k}^{a,b} \sqsubset s : \martingale_{k}^{a,b}(s) > b \text{ and } (\forall t\in (V_{k}^{a,b} , s)) \ \martingale_{k}^{a,b}(t) \leq b \};
		\end{align*}
		\item
		if $V_k^{a,b}$ is empty, let $U_k^{a,b} \coloneqq\emptyset$.
	\end{enumerate}
	The cuts $U_k^{a,b}$ and $V_k^{a,b}$ can be partial or complete.
	We now first show that, if $V_k^{a,b}$ is non-empty, there always is a supermartingale $\martingale_{k}^{a,b}$ that satisfies the conditions above. 
	We infer from the definition of the cut $V_k^{a,b}$ that
	\begin{equation*}
	\inf \bigg\{ \martingale(s) :  \martingale\in\setofextsupmartb \text{ and } \liminf\martingale \geq_s f \bigg\} < a \text{ for all } s\in V_k^{a,b}.
	 \end{equation*}
	So, for all $s\in V_k^{a,b}$, we can choose a supermartingale $\martingale_{k,s}^{a,b}$ such that $\martingale_{k,s}^{a,b}(s) < a$ and $\liminf\martingale_{k,s}^{a,b} \geq_s f$. 
	Consider now the extended real process $\martingale_{k}^{a,b}$ defined, for all $ t\in\situations$, by
	\begin{align*}
		\martingale_{k}^{a,b}(t) \coloneqq
		\begin{cases}
		\martingale_{k,s}^{a,b}(t) \ &\text{ if } s \sqsubseteq t \text{ for some } s\in V_k^{a,b}; \\
		a \ &\text{ otherwise.}
		\end{cases}
	\end{align*}
	It is clear that $\martingale_{k}^{a,b}(s) < a \text{ and } \liminf\martingale_{k}^{a,b}  \geq_s f$ for all $s\in V_k^{a,b}$.

	We show that $\martingale_{k}^{a,b}$ is a positive supermartingale.
	For all $s\in V_k^{a,b}$, it follows from Lemma~\ref{lemma: infima of supermartingales} that
	\begin{align}\label{eq: Levy martingale larger than inf f}
	\martingale_{k,s}^{a,b}(t) \geq\inf_{\omega\in\Gamma(t)} \liminf\martingale_{k,s}^{a,b}(\omega) \geq\inf_{\omega\in\Gamma(t)} f(\omega) \geq\inf f > 0 \text{ for all } t \sqsupseteq s.
	\end{align}
	Since also $a > 0$, it follows that $\martingale_{k}^{a,b}$ is positive and therefore bounded below.
	To show that $\lupprev{t}(\martingale_{k}^{a,b}(t\andstate))\leq\martingale_{k}^{a,b}(t)$ for all $t\in\situations$, fix any $t\in\situations$ and consider two cases. 
	If $V_k^{a,b} \sqsubseteq t$, then $\martingale_{k}^{a,b}(t)=\martingale_{k,s}^{a,b}(t)$ and $\martingale_{k}^{a,b}(t\andstate)=\martingale_{k,s}^{a,b}(t\andstate)$ for some $s\in V_k^{a,b}$, and therefore 
	\begin{align*}
	\lupprev{t}(\martingale_{k}^{a,b}(t\andstate))=\lupprev{t}(\martingale_{k,s}^{a,b}(t\andstate))\leq\martingale_{k,s}^{a,b}(t)=\martingale_{k}^{a,b}(t).
	\end{align*}
	If $V_k^{a,b} \not\sqsubseteq t$, then, for any $x\in\statespace$, we have either $tx\in V_k^{a,b}$ and therefore $\martingale_{k}^{a,b}(tx) < a=\martingale_{k}^{a,b}(t)$, or $V_k^{a,b} \not\sqsubseteq tx$, and therefore $\martingale_{k}^{a,b}(tx)=a=\martingale_{k}^{a,b}(t)$.
	Hence, we have that $\martingale_{k}^{a,b}(t\andstate)\leq\martingale_{k}^{a,b}(t)=a$, and therefore, by \ref{coherence: monotonicity} and \ref{coherence: bounds}, that $\lupprev{t}(\martingale_{k}^{a,b}(t\andstate))\leq\martingale_{k}^{a,b}(t)$. 
	We conclude that $\martingale_{k}^{a,b}$ is a positive supermartingale.

	Next, consider the extended real process $\mathscr{T}^{a,b}$ defined by $\mathscr{T}^{a,b}(s) \coloneqq 1$ for all $s \not\sqsupset s'$, and 
	\begin{align*}
	 	\mathscr{T}^{a,b} (s\andstate) \coloneqq 
	 	\begin{cases}
	 		\martingale_{k}^{a,b}(s\andstate) \mathscr{T}^{a,b}(s) / \martingale_{k}^{a,b}(s)   &\text{ if } s\in  [V_{k}^{a,b} , U_{k}^{a,b}) \text{ for some } k\in\nats;\\
	 		\mathscr{T}^{a,b}(s) &\text{ otherwise,}
	 	\end{cases}
	 \end{align*}
	 for all $s \sqsupseteq s'$.
	 We prove that this process is a positive $s'$-test supermartingale that converges to $+\infty$ on all paths $\omega\in\Gamma(s')$ such that
	 \begin{equation}\label{Eq: levy conv to infty}
	 \liminf_{n\to+\infty} \upprevvovkk(f \vert \omega^n) < a < b < f(\omega).
	 \end{equation}
	 That $\mathscr{T}^{a,b}$ is well-defined follows from the fact that, for any $k\in\nats$ and any $s\in  [V_{k}^{a,b} , U_{k}^{a,b})$, $\martingale_{k}^{a,b}(s)$ is positive and moreover real because of the definition of $U_{k}^{a,b}$.
	 Moreover, the process $\mathscr{T}^{a,b}$ is also positive --- and therefore bounded below --- because $\martingale_k^{a,b}(s)$ is real and positive and $\martingale_k^{a,b}(s\andstate)$ is positive, and hence $\martingale_{k}^{a,b}(s\andstate)/ \martingale_{k}^{a,b}(s)$ positive, for any $k\in\nats$ and any $s\in [V_{k}^{a,b}, U_{k}^{a,b})$.
	 Furthermore, if $s\in  [V_{k}^{a,b} , U_{k}^{a,b})$ for some $k\in\nats$, we have that
	 \begin{align*}
	 \lupprev{s}(\mathscr{T}^{a,b}(s\andstate))=\lupprev{s}(\martingale_{k}^{a,b}(s\andstate) \mathscr{T}^{a,b}(s) / \martingale_{k}^{a,b}(s)) 
	 &\overset{\text{\ref{coherence: homog for ext lambda}}}{=} \lupprev{s}(\martingale_{k}^{a,b}(s\andstate)) \mathscr{T}^{a,b}(s) / \martingale_{k}^{a,b}(s)  \\
	 &\leq \mathscr{T}^{a,b}(s).
	 \end{align*}
	 If this is not the case, we have that $\lupprev{s}(\mathscr{T}^{a,b}(s\andstate))=\lupprev{s}(\mathscr{T}^{a,b}(s))=\mathscr{T}^{a,b}(s)$ because of \ref{coherence: bounds}.
	 Hence, we have that $\lupprev{s}(\mathscr{T}^{a,b}(s\andstate))\leq\mathscr{T}^{a,b}(s)$ for all $s\in\situations$, which together with the fact that $\mathscr{T}^{a,b}(s')=1$, allows us to conclude that $\mathscr{T}^{a,b}$ is indeed a positive $s'$-test supermartingale.
	 % As, for all $s\in\situations$, $\mu \mathscr{T}^{a,b} (s)$ is either equal to $\mu \martingale_{k-1}^{a,b}(s) > 0$ for some $k\in\nats$, or equal to $1$, we get that $\mu \mathscr{T}^{a,b} (s)$ is positive $\frac{\mathscr{T}^{a,b} (s\andstate) }{\mathscr{T}^{a,b} (s)} > 0$. Consequently, $\mathscr{T}^{a,b}$ is finite ($\lambda_{\martingale_{k-1}^{a,b}}$ is real-valued) and positive.
	 % Also, for all $s\in\situations$, $\lupprev(\frac{\mathscr{T}^{a,b} (s\andstate) }{\mathscr{T}^{a,b} (s)})$ is either equal to $\lupprev(\lambda_{\martingale_{k-1}^{a,b}}(s\andstate)) \leq 1$ or equal to $1$, such that $\mathscr{T}^{a,b}$ is a supermartingale and a test supermartingale.

	 Next, we show that $\mathscr{T}^{a,b}$ converges to $+\infty$ on all paths $\omega\in\Gamma(s')$ for which \eqref{Eq: levy conv to infty} holds.
	 Consider such a path $\omega$.
	 Then $\omega$ goes through all the cuts $U_0^{a,b} \sqsubset V_1^{a,b} \sqsubset U_1^{a,b} \sqsubset ... \sqsubset V_n^{a,b} \sqsubset U_n^{a,b} \sqsubset ... $. 
	 Indeed, it is trivial that $\omega$ goes through $U_0^{a,b}=\{s'\}$. 
	 Furthermore, it follows from $\liminf_{n\to+\infty} \upprevvovkk(f \vert \omega^n) < a$ that there is an $n_1\in\nats$ such that $\omega^{n_1} \sqsupset s'$ and $\upprevvovkk(f \vert \omega^{n_1}) < a$. 
	 Take the first such $n_1\in\nats$. 
	 Then it follows from the definition of $V_1^{a,b}$ that $\omega^{n_1}\in V_1^{a,b}$. 
	 Next, it follows from $\liminf_{n\to+\infty} \martingale_1^{a,b} (\omega^n) \geq f(\omega) > b$ that there exists some $m_1\in\nats$ for which $m_1 > n_1$ and $\martingale_1^{a,b} (\omega^{m_1}) > b$. 
	 Take the first such $m_1$.
	 Then it follows from the definition of $U_1^{a,b}$ that $\omega^{m_1}\in U_1^{a,b}$. 
	 Repeating similar arguments over and over again allows us to conclude that $\omega$ indeed goes through all the cuts $U_0^{a,b} \sqsubset V_1^{a,b} \sqsubset U_1^{a,b} \sqsubset ... \sqsubset V_n^{a,b} \sqsubset U_n^{a,b} \sqsubset ... $.

	 In what follows, we use the following notation. 
	 For any situation~$s$ and for any $k\in\natz$, when $U_k^{a,b} \sqsubset s$, we denote by $u_k^s$ the (necessarily unique) situation in $U_k^{a,b}$ such that $u_k^s \sqsubset s$; observe that $u_0^s=\Box$. 
	 Similarly, for any $k\in\nats$, when $V_k^{a,b} \sqsubset s$, we denote by $v_k^s$ the (necessarily unique) situation in $V_k^{a,b}$such that $v_k^s \sqsubset s$.

	 For any situation~$s$ on a path $\omega\in\Gamma(s')$ satisfying \eqref{Eq: levy conv to infty} we now have one of the following cases:
	 \begin{enumerate}
	 \item
	 The first case is that $s \not\sqsupset V_1^{a,b}$. Then we have 
	 \begin{equation*}
	 \mathscr{T}^{a,b} (s)=\mathscr{T}^{a,b} (s')=1.
	 \end{equation*}
	  \item
	 The second case is that $s\in (V_k^{a,b} , U_{k}^{a,b}]$ for some $k\in\nats$. Then we have
	 \begin{equation*}
	 \mathscr{T}^{a,b} (s)=\Bigg( \prod_{\ell=1}^{k-1} \frac{\martingale_{\ell}^{a,b}(u_\ell^{s})}{\martingale_{\ell}^{a,b}(v_{\ell}^{s})} \Bigg) \frac{\martingale_{k}^{a,b}(s)}{\martingale_{k}^{a,b}(v_{k}^{s})}.
	 \end{equation*} 
	 Since $\martingale_k^{a,b}(s) \geq\inf f > 0$ because of Equation \eqref{eq: Levy martingale larger than inf f}, and, for all $\ell\in\{ 1,...,k\}$, $\martingale_{\ell}^{a,b}(u_\ell^{s}) > b > 0$ and $0 < \martingale_{\ell}^{a,b}(v_\ell^{s}) < a$, we get
	 \begin{equation*}
	 \mathscr{T}^{a,b} (s) \geq \Big( \frac{b}{a} \Big)^{k-1}  \frac{\martingale_{k}^{a,b}(s)}{a} \geq \Big( \frac{b}{a} \Big)^{k-1} \Big( \frac{\inf f}{a} \Big). 
	 \end{equation*} 
	 \item
	 The third case is that $s\in (U_k^{a,b} , V_{k+1}^{a,b}]$ for some $k\in\nats$. Then we have
	 \begin{equation*}
	 \mathscr{T}^{a,b} (s)=\prod_{\ell=1}^{k} \frac{\martingale_{\ell}^{a,b}(u_\ell^{s})}{\martingale_{\ell}^{a,b}(v_{\ell}^{s})}.
	 \end{equation*} 
	 As for all $\ell\in\{ 1,...,k\}$, $\martingale_{\ell-1}^{a,b}(u_\ell^{s}) > b > 0$ and $0 < \martingale_{\ell}^{a,b}(v_\ell^{s}) < a$, we find that
	 \begin{equation*}
	 \mathscr{T}^{a,b} (s) > \Big( \frac{b}{a} \Big)^{k}. 
	 \end{equation*} 
	\end{enumerate}
	Because $\inf f > 0$ and $\frac{b}{a} > 1$, and because $\omega$ goes through all the cuts, we conclude that indeed $\lim_{n\to+\infty} \mathscr{T}^{a,b} (\omega^n)=+\infty$.

	To finish, we use the countable set of rational couples $K \coloneqq\{ (a,b)\in\mathbb{Q}^2 : 0<a<b \}$ to define the process $\mathscr{T}$:
	\begin{equation*}
		\mathscr{T} \coloneqq\sum_{(a,b)\in K} w^{a,b} \mathscr{T}^{a,b},
	\end{equation*}
	with coefficients $w^{a,b}>0$ that sum to $1$. 
	Hence, $\mathscr{T}$ is a countable convex combination of the positive $s'$-test supermartingales $\mathscr{T}^{a,b}$. 
	By Lemma~\ref{lemma:positive:countable:linear:combination}, $\mathscr{T}$ is then also a supermartingale.
	Moreover, it is positive, because all $\mathscr{T}^{a,b}$ are positive and it is clear that $\mathscr{T}(s')=1$.
	Hence, $\mathscr{T}$ is a positive $s'$-test supermartingale.
	Moreover, $\mathscr{T}$ converges to $+\infty$ on the paths $\omega\in\Gamma(s')$ where $\liminf_{n\to+\infty} \upprevvovkk(f \vert \omega^n) < f(\omega)$. 
	Indeed, consider such a path $\omega$. 
	Then since $f(\omega) \geq\inf f > 0$, there is at least one couple $(a',b')\in K$ such that $\liminf_{n\to+\infty} \upprevvovkk(f \vert \omega^n) < a' < b' < f(\omega)$, and as a consequence $\lim_{n\to+\infty} \mathscr{T}^{a',b'}(\omega^n)=+\infty$. 
	Then also $\lim_{n\to+\infty} w^{a',b'} \mathscr{T}^{a',b'}(\omega^n)=+\infty$ since $w^{a',b'} > 0$. 
	For all other couples $(a,b)\in K \setminus (a',b')$, we have $w^{a,b} \mathscr{T}^{a,b} > 0$, so $\mathscr{T}$ indeed converges to $+\infty$ on $\omega$.
	\end{proof}

	% \begin{corollary}\label{Corollary: Levy}
	% For any $f\in\setofextvariablesb$ and any $s\in\situations$, the event
	% \begin{align*}
	% 	A \coloneqq\Big\{\omega\in\Gamma(s) \colon \liminf_{n\to+\infty} \upprevvovkk(f \vert \omega^n) \geq f(\omega) \Big\} \text{ is strictly almost sure within $\Gamma(s)$.}
	% \end{align*}
	% \end{corollary}
	% \begin{proof}
	% Again, consider the restricted tree that follows $s$ and apply Proposition~\ref{Theorem: Levy} to this tree.
	% Then it follows that there is a non-negative extended real-valued map $\martingale$ defined on all situations $t \sqsupseteq s$, such that $\martingale(s)=1$ and $\lupprev{t}(\martingale(t\andstate))\leq\martingale(t)$ for all $t \sqsupseteq s$ and that moreover converges to $+\infty$ on all paths $\omega\in\Gamma(s)$ where $\liminf_{n\to+\infty} \upprevvovkk(f \vert \omega^n) < f(\omega)$.
	% Hence, it suffices to extend the domain of $\martingale$, by letting it be equal to $\martingale(s)$ for all situations that do not follow $s$, to obtain an $s$-test supermartingale defined on all situations that converges to $+\infty$ on all paths $\omega\in\Gamma(s)$ where $\liminf_{n\to+\infty} \upprevvovkk(f \vert \omega^n) < f(\omega)$.
	% \end{proof}

\begin{proposition}\label{prop: inf becomes min in definition}
For any $f\in\setofextvariablesb$, the infimum in Equation \eqref{Eq: lemmaMonotoneConvergenceSAS 1} is attained.
\end{proposition}
\begin{proof}
Consider any $f\in\setofextvariablesb$ and the extended real process $\process$ defined by $\process(t) \coloneqq\upprevvovkk(f \vert t)$ for all $t\in\situations$.
Then $\process$ is a supermartingale because of Corollary~\ref{corollary: upprevvovk is supermartingale}.
Moreover, because of Theorem \ref{Theorem: Levy} we have that $\liminf\process \geq_s f$ strictly almost surely.
Since $\process(s)=\upprevvovkk(f \vert s)$, this concludes the proof.
\end{proof}

\section{Continuity of $\upprevvovkk$}\label{Sect: Continuity}
Our definition of an upper expectation, characterised by axioms \ref{coherence: const is const}--\ref{coherence: monotonicity}, implies that, if the state space is finite, it satisfies continuity with respect to non-decreasing sequences that are bounded below (see Proposition~\ref{Prop: local continuity wrt non-decreasing seq}).
Hence, since we assumed the local models $\lupprev{s}$ to be upper expectations on a finite state space, they satisfy this particular continuity.
We now prove that this also holds for the global model $\upprevvovkk$.

% We have always required that, in additon to the coherence axioms \ref{sep coherence 1}-\ref{sep coherence 3}, the local models $\lupprev{s}$ should satisfy \ref{ext coherence continuity}.
% As shown in Section \ref{sect: upprev for extended}, this axiom has a clear link with models using upper envelopes of linear expectations when we consider \emph{extended} real variables.
% Apart from that, it also imposes additional structure that results in stronger mathematical properties.
% The most significant is that the game-theoretic upper expectation $\upprevvovkk$ also satisfies continuity with respect to non-decreasing sequences of extended real variables that are bounded below, just like the local models $\lupprev{s}$.
% We prove this in the theorem below.

\begin{theorem}\label{theorem: upward convergence game}
For any $s\in\situations$ and any non-decreasing sequence $\{f_n\}_{n\in\natz}$ in $\setofextvariablesb$ that converges point-wise to a variable $f\in\setofextvariablesb$, we have that
\begin{equation*}
\upprevvovkk(f \vert s)=\lim_{n\to+\infty} \upprevvovkk(f_n \vert s).
\end{equation*} 
\end{theorem}
\begin{proof}
The idea of this proof goes back to \cite[Theorem 6.6]{ShaferVovk2014ITIP}.
% We shall prove the theorem for $s=\Box$ in order not to overload the proof with unnecessary notation.
% The proof is trivially extended towards any $s\in\situations$.
As $f_0\in\setofextvariablesb$ is bounded below and the sequence $\{f_n\}_{n\in\natz}$ is non-decreasing, there is an $M\in\reals{}$ such that $f_n \geq M$ for all $n\in\natz$ and therefore, $f$ is also bounded below by $M$. 
Hence, since $\upprevvovkk$ is constant additive [\ref{vovk coherence 6}], we can assume without loss of generality that $f$ and all $f_n$ are non-negative.

% As $\upprevvovkk$ is monotone \ref{coherence 4}* and $\{ f_n \}_{n\in\natz}$ is a non-decreasing sequence, $\{ \upprevvovkk(f_n) \}_{n\in\natz}$ is a non-decreasing sequence of extended reals and therefore
% \begin{equation*}
% 	\sup_{n\in\natz} \upprevvovkk(f_n)=\lim_{n\to+\infty} \upprevvovkk(f_n).
% \end{equation*} 
That $\lim_{n\to+\infty} \upprevvovkk(f_n \vert s)$ exists, follows from the non-decreasing character of $\{f_n\}_{n\in\natz}$ and \ref{vovk coherence 4}.
Moreover, we have that $\upprevvovkk(f \vert s) \geq \lim_{n\to+\infty} \upprevvovkk(f_n \vert s)$ because $f \geq f_n$ [since $\{f_n\}_{n\in\natz}$ is non-decreasing] and because $\upprevvovkk$ satisfies~\ref{vovk coherence 4}.
It remains to prove the converse inequality.

For any $n\in\natz$, consider the extended real process $S_n$, defined by $S_n(t) \coloneqq\upprevvovkk(f_n \vert t)$ for all $t\in\situations$ and the extended real process $S$ defined by the limit $S(t) \coloneqq\lim_{n\to+\infty}S_n(t)$ for all $t\in\situations$.
This limit exists because $\{S_n(t)\}_{n\in\natz}$ is a non-decreasing sequence for all $t\in\situations$, due to the monotonicity [\ref{vovk coherence 4}] of $\upprevvovkk$.
As $f_n$ is non-negative for all $n\in\natz$, $S_n$ is non-negative for all $n\in\natz$ because of \ref{vovk coherence 5} and therefore $S$ is also non-negative. 
As a result, $S$ and all $S_n$ are non-negative extended real processes.

It now suffices to prove that $S$ is a supermartingale such that $\liminf S \geq_s f$ s.a.s. because it will then follow from Lemma~\ref{lemmaMonotoneConvergenceSAS} that
\begin{align*}
	\upprevvovkk(f \vert s) 
	&=\inf \Big\{ \martingale(s) :  \martingale\in\setofextsupmartb \text{ and } \liminf\martingale \geq_s f \text{ s.a.s.} \Big\} \\
	&\leq S(s)=\lim_{n\to+\infty}\upprevvovkk(f_n \vert s).
\end{align*}
This is what we now set out to do.

We first show that $S$ is a supermartingale.
$S$ is bounded below because it is non-negative.
Furthermore, for all situations $t\in\situations$, we already know that 
$\{S_n(t\andstate)\}_{n\in\natz}$ is a non-decreasing sequence that converges to $S(t\andstate)$.
Since $S_n$ and $S$ are non-negative, we also have that $S_n(t\andstate), S(t\andstate)\in\setofgenextvariablesb(\statespace)$.
Then, according to \ref{ext coherence continuity} we have that 
\begin{align}\label{Eq: upward monotone game}
\lupprev{t}(S(t\andstate))=\lim_{n\to+\infty} \lupprev{t}(S_n(t\andstate)) \text{ for all } t\in\situations. 
\end{align}
$S_n$ is a supermartingale for all $n\in\natz$ because of Corollary~\ref{corollary: upprevvovk is supermartingale}, so it follows that $\lupprev{s}(S_n(t\andstate)) \leq S_n(t)$ for all $n\in\natz$ and all $t\in\situations$.
This implies, together with Equation \eqref{Eq: upward monotone game}, that
\begin{align*}
 \lupprev{t}(S(t\andstate))\leq\lim_{n\to+\infty} S_n(t)=S(t) \text{ for all } t\in\situations. 
\end{align*}
Hence, $S$ is a supermartingale.

To prove that $\liminf S \geq_s f$ s.a.s., we will use L\'evy's zero-one law.
It follows from Theorem~\ref{Theorem: Levy} that, for all $n\in\natz$, there is an $s$-test supermartingale $\martingale_n$ that converges to $+\infty$ on the event 
\begin{align*} \label{eq: proof upward convergence eq1}
	A_n \coloneqq\Big\{ \omega\in\Gamma(s) \colon \liminf_{m \to +\infty} \upprevvovkk(f_n \vert \omega^m) &< f_n(\omega) \Big\}.
\end{align*}
Now, consider the extended real process $\martingale$, defined by
\begin{equation*}
	\martingale(t) \coloneqq\sum_{n\in\nats} \lambda_n \martingale_n(t) \text{ for all } t\in\situations,
\end{equation*}
where the coefficients $\lambda_n > 0$ sum to $1$.
Then it follows from Lemma~\ref{lemma:positive:countable:linear:combination} that $\martingale$ is again a non-negative supermartingale.
Moreover, it is clear that $\martingale(s)=1$ and hence, $\martingale$ is an $s$-test supermartingale.

% We show that $\martingale$ is also a test supermartingale.
% It is clear that $\martingale(\Box)=1$ and $\martingale$ is non-negative because, for all $n\in\natz$, $\martingale_n$ is non-negative and $\lambda_n$ is positive.
% Moreover,
% \begin{align*}
% \lupprev{\situation{1}{k}}(\martingale(\situation{1}{k} \cdot)) 
% &= \lupprev{\situation{1}{k}}(\sum_{n\in\nats} \lambda_n \martingale_n(\situation{1}{k} \cdot)) \\
% &\leq \sum_{n\in\nats} \lupprev{\situation{1}{k}}( \lambda_n \martingale_n(\situation{1}{k} \cdot)) \\
% &= \sum_{n\in\nats} \lambda_n \lupprev{\situation{1}{k}}(\martingale_n(\situation{1}{k} \cdot)) \\
% &\leq \sum_{n\in\nats} \lambda_n \martingale_n(\situation{1}{k})=\martingale(\situation{1}{k}),
% \end{align*}
% for all situations $\situation{1}{k}\in\situations$, where the second step follows from Corollary~\ref{corollary: countable subadditivity for upper exp} --- which we can apply because $\{ \lambda_n \martingale_n(\situation{1}{k} \cdot) \}_{n\in\nats}$ is non-negative --- the third from non-negative homogeneity \ref{ext coherence 3} and the fourth from $\martingale_n$ being a supermartingale for all $n\in\natz$.
% Hence, $\martingale$ is a supermartingale and therefore a test supermartingale.

We show that $\martingale$ converges to $+\infty$ on all paths $\omega\in\Gamma(s)$ such that $\liminf_{m \to +\infty} S(\omega^m) < f (\omega)$.
Clearly, $\martingale$ converges to $+\infty$ on $\cup_{n\in\natz} A_n \eqqcolon A$.
Consider now any path $\omega\in\Gamma(s)$ for which $\liminf_{m \to +\infty} S(\omega^m) < f (\omega)$.
Since, as we explained before, $S_n(t)$ is non-decreasing in $n$ for all $t\in\situations$, we have that, for all $m\in\natz$, $\sup_{n\in\natz} S_n(\omega^m)=\lim_{n\to+\infty} S_n(\omega^m)=S(\omega^m)$. Since $\liminf_{m \to +\infty} S(\omega^m) < f (\omega)$, this implies that
\begin{align*}
 \liminf_{m \to +\infty} \sup_{n\in\natz} S_n(\omega^m) 
 < \lim_{n\to+\infty} f_n(\omega).
\end{align*}
Since also $\sup_{n\in\natz} \liminf_{m \to +\infty} S_n(\omega^m)\leq\liminf_{m \to +\infty} \sup_{n\in\natz} S_n(\omega^m)$ [because we obviously have that $S_n(\omega^m)\leq \sup_{n\in\natz} S_n(\omega^m)$ for all $n,m \in\natz$], we infer that
\begin{align}\label{eq: proof upward convergence eq2}
	\sup_{n\in\natz} \liminf_{m \to +\infty} \upprevvovkk(f_n \vert \omega^m) &< \lim_{n\to+\infty} f_n(\omega).
\end{align}
Then there is some $n_\omega$ such that
\begin{equation*}
\sup_{n\in\natz} \liminf_{m \to +\infty} \upprevvovkk(f_n \vert \omega^m) <  f_{n_\omega}(\omega),
\end{equation*}
and therefore, we see that also
\begin{equation*}
\liminf_{m \to +\infty} \upprevvovkk(f_{n_\omega} \vert \omega^m) <  f_{n_\omega}(\omega).
\end{equation*}
So $\omega\in A_{n_\omega} \subseteq A$ and, as a consequence, $\martingale$ converges to $+\infty$ on $\omega$.
Hence, the $s$-test supermartingale $\martingale$ converges to $+\infty$ on all paths $\omega\in\Gamma(s)$ such that $\liminf_{m \to +\infty} S(\omega^m) < f (\omega)$, and therefore $\liminf S \geq_s f$ strictly almost surely.
\end{proof}

\begin{corollary}\label{corollary: homog wrt +infty}
For any situation $s\in\situations$ and any non-negative $f\in\setofextvariables$, we have that
$\upprevvovkk((+\infty) f \vert s)=(+\infty) \upprevvovkk(f \vert s)$
\end{corollary}
\begin{proof}
It follows from \ref{vovk coherence 3} that $\upprevvovkk(\lambda f \vert s)=\lambda \upprevvovkk(f \vert s)$ for any real $\lambda > 0$.
Hence, if we let $\{\lambda_n\}_{n\in\natz}$ be a non-decreasing sequence of positive reals that converges to $+\infty$, we have that 
\begin{align}\label{Eq: corollary: homog wrt +infty}
\lim_{n\to+\infty} \upprevvovkk(\lambda_n f \vert s)=\lim_{n\to+\infty} \lambda_n \upprevvovkk(f \vert s)=(+\infty) \upprevvovkk(f \vert s).
\end{align}
Moreover, it is clear that $\{\lambda_n f \}_{n\in\natz}$ is a non-decreasing sequence of variables in $\setofextvariablesb$ that converges to $(+\infty) f$, which implies by Theorem \ref{theorem: upward convergence game} that the left hand side of Equation \eqref{Eq: corollary: homog wrt +infty} is equal to $\upprevvovkk((+\infty) f \vert s)$.
\end{proof}

\begin{corollary}\label{corollary: Vovk is an upper expectation}
For any $s\in\situations$, the map $\upprevvovkk(\cdot \vert s) \colon \setofextvariables \to \extreals$ satisfies~\ref{coherence: const is const}--\ref{coherence: monotonicity} on $\setofextvariablesb$.
\end{corollary}
\begin{proof}
 \ref{coherence: const is const} follows from \ref{vovk coherence 5}.
\ref{coherence: sublinearity} and \ref{coherence: monotonicity} respectively follow from \ref{vovk coherence 2} and \ref{vovk coherence 4}.
\ref{coherence: homog for ext lambda} follows from \ref{vovk coherence 3} for real $\lambda$, and from Corollary~\ref{corollary: homog wrt +infty} for $\lambda=+\infty$.
\end{proof}

\begin{lemma}[Fatou's Lemma]\label{lemma: Fatou general}
For any situation $s\in\situations$ and any sequence $\{f_n\}_{n\in\natz}$ in $\setofextvariablesb$ that is uniformly bounded below, we have that
\begin{align*}
\upprevvovkk(f \vert s)\leq\liminf_{n\to+\infty} \upprevvovkk(f_n \vert s) \text{ where } f \coloneqq\liminf_{n\to+\infty} f_n.
\end{align*}
\end{lemma}

\begin{proof}
Consider the variable $g_k$ defined by $g_k(\omega) \coloneqq\inf_{n \geq k}{f_n}(\omega)$ for any $k\in\natz$ and all $\omega\in\Omega$.
Then clearly $f=\lim_{k \to +\infty}{g_k}$. 
Furthermore, $\{g_k\}_{k\in\natz}$ is a non-decreasing sequence in $\setofextvariablesb$ because $\{f_n\}_{n\in\natz}$ is uniformly bounded below.
Hence, we can use Theorem \ref{theorem: upward convergence game} to find that
\begin{equation*}
\upprevvovkk(f \vert s)=\lim_{k \to +\infty} \upprevvovkk(g_k \vert s)=\liminf_{k \to +\infty} \upprevvovkk(g_k \vert s) \leq  \liminf_{k \to +\infty} \upprevvovkk(f_k \vert s),
\end{equation*}
where the inequality follows because, for all $k\in\natz$, $g_k \leq f_k$ and therefore, because of \ref{vovk coherence 4}, also $\upprevvovkk(g_k \vert s)\leq\upprevvovkk(f_k \vert s)$.
\end{proof}

\begin{theorem}
\label{Lemma Continuity w.r.t. lower cuts}
Consider any $s\in\situations$, any $f\in\setofextvariables$ and, for every $\alpha\in\reals{}$, the variable $f_\alpha\in\setofextvariables$ defined by 
$f_\alpha(\omega) \coloneqq\max\{f(\omega),\alpha\} \text{ for all } \omega\in\Omega$.
Then  
\begin{align*}
\lim_{\alpha \to -\infty} \upprevvovkk(f_\alpha \vert s)=\upprevvovkk(f \vert s).
\end{align*}
\end{theorem}
\begin{proof}
% Again, we shall assume that $s=\Box$ for simplicity.
$\upprevvovkk(f_\alpha \vert s)$ is non-decreasing in $\alpha$ because $f_\alpha$ is non-decreasing in $\alpha$ and because $\upprevvovkk$ is monotone [\ref{vovk coherence 4}], and therefore $\lim_{\alpha \to -\infty} \upprevvovkk(f_\alpha \vert s)$ exists.
Moreover, $f_\alpha \geq f$ for all $\alpha\in\reals{}$, implying, by the monotonicity [\ref{vovk coherence 4}] of $\upprevvovkk$, that $\lim_{\alpha \to -\infty} \upprevvovkk(f_\alpha \vert s) \geq \upprevvovkk(f \vert s)$.
It therefore only remains to prove the converse inequality.

If $\upprevvovkk(f \vert s)=+\infty$, then $\lim_{\alpha \to -\infty} \upprevvovkk(f_\alpha \vert s)\leq\upprevvovkk(f \vert s)$ holds trivially.
If $\upprevvovkk(f \vert s) < +\infty$, fix any real $c > \upprevvovkk(f \vert s)$.
Then it follows from the definition of $\upprevvovkk(f \vert s)$ that there is some supermartingale $\martingale\in\setofextsupmartb$ such that $\martingale(s) \leq c$ and $\liminf\martingale \geq_s f$.
Since $\martingale$ is bounded below, it immediately follows that there is some $\alpha^\ast\in\reals{}$ such that $\liminf\martingale \geq \alpha^\ast$, and hence also $\liminf\martingale \geq \alpha$ for all $\alpha\leq\alpha^\ast$.
Fix any such $\alpha\leq\alpha^\ast$.
Then it is clear that moreover $\liminf\martingale \geq_s f_\alpha$ and it therefore follows from the definition of $\upprevvovkk(f_\alpha \vert s)$ that $\upprevvovkk(f_\alpha \vert s)\leq\martingale(s) \leq c$.
Consequently, we also have that $\lim_{\alpha \to -\infty} \upprevvovkk(f_\alpha \vert s) \leq c$, and since this holds for any $c > \upprevvovkk(f \vert s)$, we conclude that indeed $\lim_{\alpha \to -\infty} \upprevvovkk(f_\alpha \vert s)\leq\upprevvovkk(f \vert s)$.
\end{proof}

\section{Continuity of $\upprevvovkk$ with respect to $n$-measurable variables}\label{Sect: Continuity wrt n-measurables}
% For any $f\in\setofgenextvariables(\mathscr{Y})$ and any $\alpha, \beta\in\extreals$ such that $\alpha\leq\beta$ we use the respective notations $f^\alpha$ and $f^{\alpha,\beta}$ to denote the extended real variables on $\mathscr{Y}$ defined by 
% \begin{align*}
% f^\alpha(y) &\coloneqq
% \begin{cases}
% \alpha \text{ if } f(y) < \alpha; \\
% f(y) \text{ otherwise,}
% \end{cases}
% \text{for all } y\in\mathscr{Y} \text{ and,} \\
% f^{\alpha,\beta}(y) &\coloneqq 
% \begin{cases}
% \alpha \text{ if } f(y) < \alpha; \\
% f(y) \text{ if } \alpha \leq f(y)\leq\beta; \\
% \beta \text{ if } f(y) > \beta,
% \end{cases}
% \text{for all } y\in\mathscr{Y}. 
% \end{align*} 
\begin{lemma}\label{Lemma: convergence of n-measurables}
Consider any global variable $h\in\setofextvariablesb$ taking values in the natural numbers.
If $h(\omega)=h(\tilde{\omega})$ for any $\omega\in\samplespace$ and any $\tilde{\omega}\in\Gamma(\omega^{h(\omega)})$, we have that $\sup h \coloneqq\sup_{\omega\in\samplespace} h(\omega)$ is real.
\end{lemma}
\begin{proof}
Assume \emph{ex absurdo} that $\sup h=+\infty$. 
Then we have that
\begin{equation*}
 \sup_{\omega\in\samplespace} h(\omega)=\sup_{x_{1}\in\statespace} \sup_{\omega\in\Gamma(x_{1})} h(\omega)=+\infty.
\end{equation*}
Since $\statespace$ is finite, there is clearly some $x_{1}^\ast\in\statespace$ for which $\sup_{\omega\in\Gamma(x_{1}^\ast)} h(\omega)=+\infty$. 
Similarly, we also find that 
\begin{equation*}
 \sup_{\omega\in\Gamma(x_{1}^\ast)} h(\omega)=\sup_{x_{2}\in\statespace} \sup_{\omega\in\Gamma(x_{1}^\ast x_{2})} h(\omega)=+\infty.
\end{equation*}
Since $\statespace$ is finite, there is again some $x_{2}^\ast\in\statespace$ for which $\sup_{\omega\in\Gamma(x_{1}^\ast x_{2}^\ast)} h(\omega)=+\infty$.
We can continue in this way and construct a path $\omega_\ast=x_{1}^\ast x_{2}^\ast ... x_{n}^\ast ...$ for which 
\begin{equation*}
	\sup_{\omega\in\Gamma(\omega_\ast^n)} h(\omega)=+\infty \text{ for all } n\in\natz.
\end{equation*}
However, $h$ takes values in the natural numbers, so $h(\omega_\ast)$ is real. 
This implies, together with $h(\omega)=h(\omega_\ast)$ for any $\omega\in\Gamma(\omega_\ast^{h(\omega_\ast)})$, that
\begin{equation*}
	\sup \set*{h(\omega) \colon \omega\in\Gamma\big(\omega_\ast^{h(\omega_\ast)}\big)} 
	= \sup \set*{h(\omega_\ast) \colon \omega\in\Gamma\big(\omega_\ast^{h(\omega_\ast)}\big)}=h(\omega_\ast) < +\infty.
\end{equation*}
This is a contradiction, and therefore $\sup h$ is real.
\end{proof}

\begin{proposition}\label{Prop: cont. wrt non-increasing n-measurables}
For any $s\in\situations$ and any non-increasing sequence $\{f_n\}_{n\in\natz}$ of $n$-measurable gambles that converges point-wise to a variable $f\in\setofextvariables$, we have that
\begin{equation*}
\upprevvovkk(f \vert s)=\lim_{n\to+\infty} \upprevvovkk(f_n \vert s).
\end{equation*} 
\end{proposition}
\begin{proof}
The idea behind the proof of this theorem originates from \cite{DeBock2014Continuity}, where a version of Definition \ref{def:upperexpectation2} with real supermartingales is used.
We here adapt it to our setting.
Fix any $s\in\situations$.
Because $\{f_n\}_{n\in\natz}$ is non-increasing and $\upprevvovkk$ is monotone [\ref{vovk coherence 4}], $\upprevvovkk(f_n \vert s)$ is also non-increasing in $n$ and hence the limit $\lim_{n\to+\infty} \upprevvovkk(f_n \vert s)$ exists.
Moreover, because $f_n \geq f$ for all $n\in\natz$ and again $\upprevvovkk$ is monotone [\ref{vovk coherence 4}], we have that $\lim_{n\to+\infty} \upprevvovkk(f_n \vert s) \geq \upprevvovkk(f \vert s)$.
Hence, we are left to show the other inequality.
If $\upprevvovkk(f \vert s)=+\infty$, the inequality trivially holds.
If not, we will prove that, for all $\alpha\in\reals{}$ such that $\upprevvovkk(f \vert s) < \alpha$,  $\lim_{n\to+\infty} \upprevvovkk(f_n \vert s)\leq\alpha$.

So consider any $\alpha\in\reals{}$ such that $\upprevvovkk(f \vert s) < \alpha$.
Then there is a supermartingale $\martingale\in\setofextsupmartb{}$ such that $\martingale(s)\leq\alpha$ and $\liminf\martingale \geq_s f$.
Fix any $\epsilon > 0$ and any $\omega\in\Gamma(s)$.
Note that $f$ is bounded above because $\{f_n\}_{n\in\natz}$ is a non-increasing sequence of gambles.
Hence, because $\{f_n\}_{n\in\natz}$ converges point-wise to $f$, there is, for every real $\beta > f(\omega)$, an index $n'$ such that $\beta \geq f_n(\omega)$ for all $n \geq n'$.
So if $\liminf\martingale(\omega)$ is real, there is an index $M(\omega)\in\natz$ such that $\liminf\martingale(\omega) + \epsilon \geq f_n(\omega)$ for all $n \geq M(\omega)$.
Moreover, because $\liminf\martingale(\omega)$ is real, there exists a second index $N(\omega)\in\natz$, such that $\martingale(\omega^n) + \epsilon \geq \liminf\martingale(\omega)$ for all $n \geq N(\omega)$.
Hence, for all $n \geq \max\{M(\omega), N(\omega)\}$, we have that $\martingale(\omega^n) + 2\epsilon \geq f_n(\omega)$.

If $\liminf\martingale(\omega)$ is not real, it can only be equal to $+\infty$ because of the bounded belowness of $\martingale$.
Since $f_0$ is a gamble, there clearly is an index $M(\omega)$ such that $\martingale(\omega^n) \geq f_0(\omega)$ for all $n \geq M(\omega)$, which by the non-increasing character of $\{f_n\}_{n\in\natz}$ implies that $\martingale(\omega^n) \geq f_n(\omega)$ for all $n \geq M(\omega)$.
Hence, we conclude that for all $\omega\in\Gamma(s)$ and any $n'\in\natz$, there is a natural number $n \geq n'$ such that $\martingale(\omega^n) + 2\epsilon \geq f_n(\omega)$.

Let $\ell$ be the length of the string $s$.
Now consider the variable $h\in\setofextvariablesb$ defined by $h(\omega) \coloneqq\inf\{k \geq \ell \colon \martingale(\omega^k) \geq f_k(\omega)-2\epsilon \Big\}$ for all $\omega\in\Gamma(s)$ and $h(\omega) \coloneqq\ell$ for all $\omega \not\in\Gamma(s)$.
Then it is clear from the argument above, that $h$ takes values in the natural numbers.
Moreover, for any $n\in\natz$, $\omega\in\samplespace$ and $\tilde{\omega}\in\Gamma(\omega^n)$, we have that $\omega^n=\tilde{\omega}^n$ and therefore $\martingale(\omega^n)=\martingale(\tilde{\omega}^n)$ and moreover $f_n(\omega)=f_n(\tilde{\omega})$ because $f_n$ is $n$-measurable. 
Hence, for any $n\in\natz$, any $\omega\in\samplespace$ and any $\tilde{\omega}\in\Gamma(\omega^n)$, $f_n(\omega)-\martingale(\omega^n)=f_n(\tilde{\omega})-\martingale(\tilde{\omega}^n)$.
This implies that $h(\omega)=h(\tilde{\omega})$  for any  $\omega\in\Gamma(s)$ and all $\tilde{\omega}\in\Gamma(\omega^{h(\omega)})$.
That $h(\omega)=h(\tilde{\omega})$ holds is obviously also true for any two paths $\omega$ and $\tilde{\omega}$ outside of $\Gamma(s)$.
Hence, the conditions for Lemma~\ref{Lemma: convergence of n-measurables} are satisfied and we can therefore infer that $\sup h < +\infty$.

Let $U \coloneqq\{t\in\situations \colon (\exists \omega\in\samplespace)  t=\omega^{h(\omega)} \}$.
Then $U$ is a cut because $h(\omega)=h(\tilde{\omega})$  for any  $\omega\in\samplespace$ and all $\tilde{\omega}\in\Gamma(\omega^{h(\omega)})$.
It is also a complete cut because $h$ takes values in the natural numbers.
For any situation $t \sqsupseteq U$, let us write $u(t)$ to denote the unique situation in $U$ such that $u(t) \sqsubseteq t$.
We then let $\martingale_U$ be the extended real process defined by
\begin{align*}
\martingale_U(t) \coloneqq
\begin{cases}
\martingale(t) &\text{ if } U \not\sqsubset t; \\
\martingale(u(t)) &\text{ otherwise,}
\end{cases}
\text{ for all } t\in\situations.
\end{align*}
The process $\martingale_U$ is bounded below because $\martingale$ is bounded below.
Moreover, we have that
\begin{align*}
\martingale_U(t\andstate) \coloneqq
\begin{cases}
\martingale(t\andstate) &\text{ if } U \not\sqsubseteq t; \\
\martingale(u(t)) &\text{ otherwise,}
\end{cases}
\text{ for all } t\in\situations.
\end{align*}
Since, for any $t\in\situations$, $\lupprev{t}(\martingale(t\andstate))\leq\martingale(t)$ --- because $\martingale$ is a supermartingale --- and $\lupprev{t}(c)=c$ for all $c\in\extreals_{\geq 0}$ because of \ref{coherence: bounds}, it follows that $\lupprev{t}(\martingale_U(t\andstate))\leq\martingale_U(t)$ for all $t\in\situations$.
Hence, $\martingale_U$ is also a supermartingale.

For any $\omega\in\samplespace$, we now let $u(\omega)$ be the unique situation in $U$ such that $\omega\in\Gamma(u(\omega))$.
Clearly $u(\omega)=\omega^{h(\omega)}$ for all $\omega\in\samplespace$, implying that
\begin{align*}
\lim_{m \to +\infty} \martingale_U(\omega^m)=\lim_{m \to +\infty} \martingale(u(\omega))=\martingale(\omega^{h(\omega)}) \text{ for all } \omega\in\samplespace.
\end{align*} 
Therefore, by definition of $h$, we have that
\begin{align*}
	f_{h(\omega)}(\omega)-\lim_{m \to +\infty} \martingale_U(\omega^m)=f_{h(\omega)}(\omega)-\martingale(\omega^{h(\omega)}) \leq 2 \epsilon \text{ for all } \omega\in\Gamma(s).
\end{align*}
Since $\{f_n\}_{n\in\natz}$ is non-increasing and $h(\omega)\leq\sup h$ for all $\omega\in\samplespace$, we have that $f_{\sup h}(\omega) \leq f_{h(\omega)}(\omega)$.
Note that $f_{\sup h}$ is an element of $\{f_n\}_{n\in\natz}$ because $\sup h$ is a natural number as shown above.
So we have that 
\begin{align*}
	f_{\sup h}(\omega)-\lim_{m \to +\infty} \martingale_U(\omega^m) \leq 2 \epsilon \text{ for all } \omega\in\Gamma(s).
\end{align*}
Hence, $f_{\sup h}-2 \epsilon \leq_s  \lim \martingale_U$, which implies, together with $\martingale_U\in\setofextsupmartb{}$ and \ref{vovk coherence 6}, that $\upprevvovkk(f_{\sup h} \vert s)-2 \epsilon\leq\martingale_U(s)$.
Moreover, $\martingale_U(s)=\martingale(s)$ because clearly $U \not\sqsubset s$, implying that $\upprevvovkk(f_{\sup h} \vert s)-2 \epsilon\leq\martingale(s)\leq\alpha$.
Since $\{f_n\}_{n\in\natz}$ is non-increasing and $\upprevvovkk$ is monotone [\ref{vovk coherence 4}], this on its turn results in $\lim_{n\to+\infty} \upprevvovkk(f_{n} \vert s)-2 \epsilon\leq\alpha$.
This inequality holds for any $\epsilon > 0$ and any real $\alpha > \upprevvovkk(f \vert s)$, so we conclude that indeed $\lim_{n\to+\infty} \upprevvovkk(f_{n} \vert s)\leq\upprevvovkk(f \vert s)$.
\end{proof}

\begin{corollary}\label{Corollary: cont. wrt non-increasing fin-measurables}
For any $s\in\situations$ and any non-increasing sequence $\{f_n\}_{n\in\natz}$ of finitary gambles that converges point-wise to a variable $f\in\setofextvariables$, we have that
\begin{equation*}
\upprevvovkk(f \vert s)=\lim_{n\to+\infty} \upprevvovkk(f_n \vert s).
\end{equation*} 
\end{corollary}
\begin{proof}
Consider such a sequence $\{f_n\}_{n\in\natz}$ that converges point-wise to some $f\in\setofextvariables$.
We show, in an analogous way as we did for the proof of Lemma~\ref{lemma: limits of n-measurables are equal to limits of finitary}, that there is a non-increasing sequence $\{f'_n\}_{n\in\natz}$ of $n$-measurable gambles that converges point-wise to $f$ and such that moreover $\lim_{n\to+\infty} \upprevvovkk(f'_n \vert s)$ is equal to $\lim_{n\to+\infty} \upprevvovkk(f_n \vert s)$.
Then, by Proposition~\ref{Prop: cont. wrt non-increasing n-measurables}, we indeed have that $\lim_{n\to+\infty} \upprevvovkk(f_n \vert s)=\lim_{n\to+\infty} \upprevvovkk(f'_n \vert s)=\upprevvovkk(f \vert s)$.

Let $f'_0 \coloneqq c$ for some $c\in\reals{}$ and $\gamma(1) \coloneqq 0$.
Let $\{f'_n\}_{n\in\natz}$ be defined by the following recursive expressions:
\begin{flalign*}
f'_n &\coloneqq
\begin{cases}
f_{\gamma(n)} &\text{ if $f_{\gamma(n)}$ is $n$-measurable}; \\
f'_{n-1} &\text{ otherwise,}
\end{cases} \\
\text{ and } \qquad \qquad \qquad \qquad \qquad & & \\
\gamma(n+1) &\coloneqq
\begin{cases}
\gamma(n)+1 &\text{ if $f_{\gamma(n)}$ is $n$-measurable};\\
\gamma(n) &\text{ otherwise.}
\end{cases}
\end{flalign*}
for all $n\in\nats$.
The original sequence $\{f_n\}_{n\in\natz}$ is a subsequence of $\{f'_n\}_{n\in\natz}$. 
Moreover, the additional elements in $\{f'_n\}_{n\in\natz}$ clearly do not change the limit behaviour, nor does it change the non-increasing character.
Hence, the limits $\lim_{n\to+\infty} f'_n$ and $\lim_{n\to+\infty} f_n$ are equal.
The same argument holds to show that $\lim_{n\to+\infty} \upprevvovkk(f'_n \vert s)$ equals $\lim_{n\to+\infty} \upprevvovkk(f_n \vert s)$.
It is moreover clear that $\{f'_n\}_{n\in\natz}$ is a sequence of gambles because $\{f_n\}_{n\in\natz}$ is a sequence of gambles.
Hence, we are left to show that the variables $f'_n$ are $n$-measurable. 
We will do this by induction.
$f'_0=c$ is clearly $0$-measurable.
To prove the induction step, suppose that $f'_{n-1}$ is $(n-1)$-measurable for some $n\in\nats$.
Then either we have that $f_{\gamma(n)}$ is $n$-measurable, which directly implies that $f'_n=f_{\gamma(n)}$ is $n$-measurable.
Otherwise, $f'_n$ is equal to $f'_{n-1}$ implying that $f'_n$ is $(n-1)$-measurable and therefore automatically $n$-measurable.
This concludes the induction step.
\end{proof}

For any $f\in\setofextvariables{}$ and any $c \in \reals{}$, we let $f^{\wedge c}$ be the variable defined by $f^{\wedge c}(\omega) \coloneqq \min \set{f(\omega), c}$ for all $\omega \in \samplespace{}$.
For any countable net $\{f_{(m,n)}\}_{m,n\in\natz}$ of gambles, we say that the gamble $f \coloneqq \lim_{(m,n) \to (+\infty,+\infty)} f_{(m,n)}$ is the Moore-Smith limit of $\{f_{(m,n)}\}_{m,n\in\natz}$ if, for all $\omega \in \samplespace{}$ and all $\epsilon > 0$, there is a couple $(m^\ast, n^\ast) \in \natz^2$ such that $\vert f_{(m,n)}(\omega) - f(\omega) \vert \leq \epsilon$ for all $m \geq m^\ast$ and all $n \geq n^\ast$.
We trivially extend this definition to any $f\in\setofextvariablesb{}$ by additionaly requiring that, for all $\omega \in \samplespace{}$ such that $f(\omega) = + \infty$ and for any $\alpha>0$, there is a couple $(m^\ast, n^\ast) \in \natz^2$ such that $f_{(m,n)}(\omega) \geq \alpha$ for all $m \geq m^\ast$ and all $n \geq n^\ast$.

\begin{lemma}\label{lemma: Moore-smith}
Consider any sequence $\{f_n\}_{n\in\natz}$ in $\setofgambles{}$ that converges point-wise to some extended real variable $f$ that is bounded below.
Then $\lim_{(m,n) \to (+\infty,+\infty)} f_n^{\wedge m} = f$.
\end{lemma}
\begin{proof}
Consider any $\omega \in \samplespace{}$.
If $f(\omega) \in \reals{}$, fix any $\epsilon > 0$.
Then there is an $n^\ast \in \natz$ such that $\vert f_{n}(\omega) - f(\omega) \vert \leq \epsilon$ for all $n \geq n^\ast$.
Hence, for any two $n_1, n_2 \geq n^\ast$, we have that $\vert f_{n_1}(\omega) - f_{n_2}(\omega) \vert \leq 2\epsilon$.
Now choose any $m^\ast \geq f_{n^\ast}(\omega) + 2\epsilon$.
Then clearly $f_n^{\wedge m}(\omega) = f_n(\omega)$ for all $n \geq n^\ast$ and all $m \geq m^\ast$, implying that $\vert f_n^{\wedge m}(\omega) - f(\omega) \vert \leq \epsilon$ for all $n \geq n^\ast$ and all $m \geq m^\ast$.
If $f(\omega) = + \infty$, fix any $\alpha > 0$.
Then there is an $n^\ast \in \natz$ such that $f_{n}(\omega) \geq \alpha$ for all $n \geq n^\ast$.
If we now take $m^\ast \geq \alpha$, then clearly also $f_{n}^{\wedge m}(\omega) \geq \alpha$ for all $n \geq n^\ast$ and all $m \geq m^\ast$.
Hence, we can conclude that indeed $\lim_{(m,n) \to (+\infty,+\infty)} f_n^{\wedge m} = f$.
\end{proof}

\begin{proposition}\label{Prop: limit of n-meas}
For any $s\in\situations$ and any $f\in\setoflimitsoffinmeasb{}$, there is a sequence $\{f_n\}_{n\in\natz}$ of $n$-measurable gambles that is uniformly bounded below and that converges point-wise to $f$, for which 
\begin{equation*}
\lim_{n\to+\infty} \upprevvovkk(f_n \vert s)=\upprevvovkk(f \vert s)
\end{equation*}
and $f_n \leq \sup f$ for all $n \in \natz{}$.
If \/ $\inf f \in \reals{}$, then we can moreover guarantee that $\inf f \leq f_n$ for all $n \in \natz{}$.
\end{proposition}
\begin{proof}
If $\inf f = +\infty$, meaning that $f = \sup f = +\infty$, it suffices to consider the increasing sequence of gambles $\{n\}_{n \in \nats{}}$ and apply \ref{vovk coherence 5} to see that the proposition holds.
So suppose that $\inf f$ is real. 
We only have to show that there is a sequence of $n$-measurable gambles $\{f_n\}_{n\in\natz}$ that converges point-wise to $f$ such that $\inf f \leq f_n \leq \sup f$ for all $n \in \natz{}$ and $\limsup_{n \to +\infty} \upprevvovkk{}(f_n \vert s) \leq \upprevvovkk{}(f \vert s)$.
Indeed, since $\{f_n\}_{n\in\natz}$ is then uniformly bounded below because $\inf f$ is real, we can then apply Lemma \ref{lemma: Fatou general} to find that also $\liminf_{n \to +\infty} \upprevvovkk{}(f_n \vert s) \geq \upprevvovkk{}(f \vert s)$, and therefore that $\lim_{n \to +\infty} \upprevvovkk{}(f_n \vert s) = \upprevvovkk{}(f \vert s)$.

We assume without loss of generality that $\inf f = 0$.
Indeed, if we prove the proposition for such a variable, we can generalise it towards any extended variable $f^\prime$ such that $\inf f^\prime$ is real because of \ref{vovk coherence 6}.

According to Lemma \ref{lemma: limits of n-measurables are equal to limits of finitary}, there is a sequence of $n$-measurable gambles $\{g_n\}_{n\in\natz}$ such that $\inf f = 0 \leq g_n \leq \sup f$ for all $n \in \natz{}$ and such that $\lim_{n \to +\infty} g_n = f$.
Consider the net $\{g_n^{\wedge m}\}_{m,n\in\natz}$ and the sequence $\{f^{\wedge m}\}_{m \in \natz}$.
Then it is clear that, for any $m \in \natz$, $\{g_n^{\wedge m}\}_{n\in\natz}$ is a sequence of $n$-measurable gambles that converges point-wise to $f^{\wedge m}$.
What we will do next, for every $m \in \natz{}$, is to combine the gambles in $\{g_n^{\wedge m}\}_{n\in\natz}$ to obtain a new finitary gamble $g_{h_m}$, such that $\upprevvovkk{}(g_{h_m} \vert s)  \leq \upprevvovkk{}(f^{\wedge m} \vert s) + \sfrac{1}{m}$ and $0 \leq g_{h_m} \leq \sup f$.
Furthermore, $\{g_{h_m}\}_{m \in \natz{}}$ will converge pointwise to $f$.

% If $\upprevvovkk{}(f \vert s) = +\infty$, we trivially have that $\limsup_{n \to +\infty} \upprevvovkk{}(g_n \vert s) \leq \upprevvovkk{}(f \vert s)$.
% So assume that $\upprevvovkk{}(f \vert s)$ is real and fix any real $\epsilon>0$.
For any $m \in \natz$, since $f^{\wedge m}$ is a gamble [because $f$ is bounded below], we have by \ref{vovk coherence 5} that $\upprevvovkk{}(f^{\wedge m} \vert s)$ is real. 
Hence, there is a supermartingale $\martingale_m \in \setofextsupmartb{}$ such that $\martingale_m(s) \leq \upprevvovkk{}(f^{\wedge m} \vert s) + \sfrac{1}{3m} < +\infty$ and $\liminf \martingale_m \geq_s f^{\wedge m}$. 
Then, because $f^{\wedge m}$ is a gamble, there is, for any $\omega \in \Gamma(s)$, an index $N_m(\omega) \in \natz$ such that $\martingale_m(\omega^n) \geq f^{\wedge m}(\omega) - \sfrac{1}{3m}$ for all $n \geq N_m(\omega)$. 
Since $\{g_n^{\wedge m}\}_{n\in\natz}$ converges point-wise to $f^{\wedge m}$ and $f^{\wedge m}$ is a gamble, there is, for any $\omega \in \Gamma(s)$, an index $M_m(\omega) \in \natz$ such that $\vert f^{\wedge m}(\omega) - g_n^{\wedge m}(\omega) \vert \leq \sfrac{1}{3m}$ for all $n \geq M_m(\omega)$.
Therefore, we have, for any $\omega \in \Gamma(s)$, that
\begin{equation}\label{Eq: proof cont of n-measurables: g_n are real}
\martingale_m(\omega^n) \geq g_n^{\wedge m}(\omega) - \sfrac{2}{3m} \text{ for all } n \geq \max\{N_m(\omega), M_m(\omega)\}.
\end{equation}

Let $\ell$ be the length of the situation $s$. 
We now build an increasing sequence $\{h_m\}_{m \in \natz}$ of non-negative global variables, where $h_0 \coloneqq \ell$ and, for all $m \in \nats{}$, 
\begin{align*}
h_m(\omega) \coloneqq 
\begin{cases}
\inf \Big\{k \in \natz \colon k \geq h_{m-1}(\omega) + 1 \text{ and } \martingale_m(\omega^k) \geq g_k^{\wedge m}(\omega) - \sfrac{2}{3m} \Big\} \text{ if } \omega \in \Gamma(s) \\
\ell + m  \ \text{ otherwise,}
\end{cases}
\end{align*}
for all $\omega \in \samplespace{}$.
% Hence, for every $\omega \in \Gamma(s)$, the function $h_1$ returns the first index $n > \ell$ for which $\martingale(\omega^n) \geq g_n(\omega) - \sfrac{2}{3m}$, the function $h_2$ returns the second index $n > \ell$ for which $\martingale(\omega^n) \geq g_n(\omega) - \sfrac{2}{3m}$, and so on. 
Observe that, for any $m \in \nats{}$, $h_m$ takes values in the natural numbers because Equation \eqref{Eq: proof cont of n-measurables: g_n are real} holds for any $\omega \in \Gamma(s)$ and \emph{all} $n$ larger than $\max\{N_m(\omega), M_m(\omega)\}$.
We show that, for all $m \in \natz$,
\begin{equation}\label{eq: equal indices}
 h_m(\omega) = h_m(\tilde{\omega}) \text{ for any } \omega \in \samplespace{} \text{ and all } \tilde{\omega} \in \Gamma(\omega^{h_m(\omega)}),
\end{equation}
implying that $h_m$ satisfies the conditions in Lemma \ref{Lemma: convergence of n-measurables}.

Equation~\eqref{eq: equal indices} holds trivially for any $m \in \natz{}$, any $\omega \in \samplespace{}\setminus{\Gamma(s)}$ and any $\tilde{\omega} \in \Gamma(\omega^{h_m(\omega)})$.
Indeed, since $h_m(\omega) \geq \ell$ and $\omega \not\in \Gamma(s)$, we also have that $\tilde{\omega} \not\in \Gamma(s)$, implying by the definition of $h_m$ that $h_m(\omega) = h_m(\tilde{\omega}) = \ell + m$.

We now show by induction that Equation~\eqref{eq: equal indices} also holds for any $m \in \natz{}$, any $\omega \in \Gamma(s)$ and any $\tilde{\omega} \in \Gamma(\omega^{h_m(\omega)})$.
It is clear that it holds for $m=0$.
Now suppose that it holds for some $(m-1) \in \natz$ and consider any $\omega \in \Gamma(s)$ and any $\tilde{\omega} \in \Gamma(\omega^{h_m(\omega)})$.
Note that, because $h_m(\omega) > h_{m-1}(\omega)$, we also have that $\tilde{\omega} \in \Gamma(\omega^{h_{m-1}(\omega)})$ and hence, by the inductive hypothesis, that $h_{m-1}(\omega) = h_{m-1}(\tilde{\omega})$.
Moreover, for all $n \leq h_n(\omega)$, we have that $\martingale_m(\omega^{n}) = \martingale_m(\tilde{\omega}^{n})$ because $\omega^{n} = \tilde{\omega}^{n}$, and also $g_{n}^{\wedge m}(\omega) = g_{n}^{\wedge m}(\tilde{\omega})$ since $g_{n}^{\wedge m}$ is $n$-measurable.
Hence, for any $n \leq h_m(\omega)$, we have that $g_n^{\wedge m}(\omega) - \martingale_m(\omega^n) = g_n^{\wedge m}(\tilde{\omega}) - \martingale_m(\tilde{\omega}^n)$, implying, together with $h_{m-1}(\omega) = h_{m-1}(\tilde{\omega})$ and the definition of $h_m$, that $h_m(\omega) = h_m(\tilde{\omega})$.
This concludes our proof that Equation~\eqref{eq: equal indices} holds, which allows us to use Lemma \ref{Lemma: convergence of n-measurables} to obtain that $\sup h_m < +\infty$ for all $m \in \natz$.
% because $\martingale_m(\omega^{n}) \geq f_{n}^{\wedge m}(\omega) - \sfrac{2}{3m}$ by definition of $n$, this implies that $\martingale_m(\tilde{\omega}^{n}) \geq f_{n}^{\wedge m}(\tilde{\omega}) - \sfrac{2}{3m}$
% For any $n\in\natz$, any $\omega \in \samplespace{}$ and any $\tilde{\omega} \in \Gamma(\omega^n)$, we have that $\omega^n = \tilde{\omega}^n$ and therefore $\martingale_m(\omega^n) = \martingale_m(\tilde{\omega}^n)$ and moreover $g_n^{\wedge m}(\omega) = g_n^{\wedge m}(\tilde{\omega})$ because $g_n^{\wedge m}$ is $n$-measurable. 
% Hence, for any $n\in\natz$, any $\omega \in \samplespace{}$ and any $\tilde{\omega} \in \Gamma(\omega^n)$, $g_n^{\wedge m}(\omega) - \martingale_m(\omega^n) = g_n^{\wedge m}(\tilde{\omega}) - \martingale_m(\tilde{\omega}^n)$.
% This means that, for all $m \in \natz$,
% \begin{equation}\label{eq: equal indices}
%  h_m(\omega) = h_m(\tilde{\omega}) \text{ for any } \omega \in \Gamma(s) \text{ and all } \tilde{\omega} \in \Gamma(\omega^{h_m(\omega)}).
% \end{equation}
% The above equality is obviously also true for paths $\omega$ outside of $\Gamma(s)$, and hence, the conditions for Lemma \ref{Lemma: convergence of n-measurables} are satisfied and hence, $\sup h_m < +\infty$ for all $i \in \nats{}$.
So, $\{h_m\}_{m \in \natz}$ is an increasing sequence of non-negative gambles.

Also note that $h_m$ is $(\sup h_m)$-measurable for all $m \in \natz$.
Indeed, for any $\omega \in \Gamma(s)$, consider any $m \in \natz$ and any $\tilde{\omega} \in \Gamma(\omega^{\sup h_m})$, then clearly also $\tilde{\omega} \in \Gamma(\omega^{h_m(\omega)})$ because $h_m(\omega) \leq \sup h_m$, and therefore, by Equation \eqref{eq: equal indices}, we have that $h_m(\omega) =  h_m(\tilde{\omega})$.
Keeping in mind that $h_m$ is constant outside of $\Gamma(s)$, together with the fact that $\sup h_m \geq \ell$, we can indeed conclude that $h_m$ is $(\sup h_m)$-measurable.

Now, consider the sequence of real variables $\{g_{h_m}\}_{m \in \natz}$ defined by
\begin{equation}\label{eq: def sequence of n-measurable gambles}
g_{h_m}(\omega) \coloneqq g_{h_m(\omega)}^{\wedge m}(\omega) \ \text{ for all } \omega \in \samplespace{} \text{ and all } m \in \natz.
\end{equation}
Observe that $0 \leq g_{h_m} \leq \sup f$ for all $m \in \natz{}$ because $0 \leq g_n \leq \sup f$ for all $n \in \natz{}$.
Moreover, we have that, for any $m \in \natz$, $g_{h_m}$ is a $(\sup h_m)$-measurable gamble.
Indeed, consider any $m \in \natz$, any $\omega \in \samplespace{}$ and any $\tilde{\omega} \in \Gamma(\omega^{(\sup h_m)})$.
Since $h_m$ is $(\sup h_m)$-measurable, we have that $h_m(\omega) = h_m(\tilde{\omega})$ and hence $g_{h_m}(\omega) = g_{h_m(\omega)}^{\wedge m}(\omega) = g_{h_m(\tilde{\omega})}^{\wedge m}(\omega)$.
Now, $g_{h_m(\tilde{\omega})}^{\wedge m}$ is by definition $h_m(\tilde{\omega})$-measurable and therefore surely $(\sup h_m)$-measurable.
Hence, it follows that $g_{h_m(\tilde{\omega})}^{\wedge m}(\omega) = g_{h_m(\tilde{\omega})}^{\wedge m}(\tilde{\omega}) = g_{h_m}(\tilde{\omega})$ because $\tilde{\omega} \in \Gamma(\omega^{(\sup h_m)})$.
As a consequence, we find that $g_{h_m}(\omega) = g_{h_m}(\tilde{\omega})$, implying that $g_{h_m}$ is indeed $(\sup h_m)$-measurable.
That it is a gamble, follows directly from its $(\sup h_m)$-measurability and its real valuedness, which on its turn follows from the fact that all $g_n$ are gambles.
Indeed, $g_{h_m}$ can then only take $\vert \statespace{} \vert ^{(\sup h_m)}$ different real values, which by the finiteness of $\statespace{}$ implies that $g_{h_m}$ is bounded.

%  situation $s' \in \statespace{}_{1: \sup h'_i}$.
% Since $h_m$ is $(\sup h'_i)$-measurable, we have that $g_{h_m}(\omega) = f_{h_m(\omega)}(\omega) = f_{h_m(s')}(\omega)$ for all $\omega \in \Gamma(s')$.
% Now, $f_{h_m(s')}$  is $h_m(s')$-measurable and therefore surely $(\sup h'_i)$-measurable because $h_m(s') \leq \sup h_m = \sup h'_i$.
% Hence, $f_{h_m(s')}$ --- and therefore also $g_{h_m}$ --- is constant for all paths $\omega \in \Gamma(s')$.
% This holds for any $s' \in \statespace{}_{1: \sup h_m}$ and as a consequence, $g_{h_m}$ is $(\sup h_m)$-measurable.

Next, we show that $\{g_{h_m}\}_{m \in \natz}$ converges point-wise to $f$.
Fix any $\omega \in \Omega$.
Suppose that $f(\omega)$ is real.
Then it follows from Lemma \ref{lemma: Moore-smith} and the definition of the Moore-Smith limit, that, for any $\epsilon > 0$, there is a couple $(m^\ast, n^\ast) \in \natz^2$ such that $\vert g_n^{\wedge m}(\omega) - f(\omega) \vert \leq \epsilon$ for all $m \geq m^\ast$ and all $n \geq n^\ast$.
In particular, if $p^\ast \coloneqq \max\{m^\ast,n^\ast\}$, we have that $\vert g_n^{\wedge m}(\omega) - f(\omega) \vert \leq \epsilon$ for all $m \geq p^\ast$ and all $n \geq p^\ast$.
Since $h_m(\omega)$ is increasing in $m$, this implies that there also is an $m^\prime \in \natz$ such that $\vert g_{h_m}(\omega) - f(\omega) \vert = \vert g_{h_m(\omega)}^{\wedge m}(\omega) - f(\omega) \vert \leq \epsilon$ for all $m \geq m^\prime$.
Hence, $\lim_{m \to +\infty} g_{h_m}(\omega) = f(\omega)$.
Analogously, we can prove that $\lim_{m \to +\infty} g_{h_m}(\omega) = f(\omega)$ holds if $f(\omega) = +\infty$.
Hence, we indeed conclude that $\{g_{h_m}\}_{m \in \natz}$ converges point-wise to $f$.

Finally, we show that such that $\upprevvovkk{}(g_{h_m} \vert s)  \leq \upprevvovkk{}(f^{\wedge m} \vert s) + \sfrac{1}{m}$ for all $m \in \natz{}$.
To do this, we define the following sequence of cuts:
\begin{equation*}
U_{h_m} \coloneqq \{t \in \situations \colon (\exists \omega \in \samplespace{})  t = \omega^{h_m(\omega)} \} \text{ for all } m \in \natz.
\end{equation*}
Note that $U_{h_m}$ is indeed a cut because of Equation \eqref{eq: equal indices}.
Moreover, it is a complete cut because $h_m(\omega)$ is real for all $\omega \in \samplespace{}$.
For any situation $t \sqsupseteq U_{h_m}$, let us write $u_{h_m}(t)$ to denote the unique situation in $U_{h_m}$ such that $u_{h_m}(t) \sqsubseteq t$.
We use a similar notation for paths;
for any $\omega \in \samplespace{}$, $u_{h_m}(\omega)$ is the unique situation in $U_{h_m}$ such that $\omega \in \Gamma(u_{h_m}(\omega))$.
Then it should be clear that $u_{h_m}(\omega) = \omega^{h_m(\omega)}$ for all $\omega \in \samplespace{}$.
We will use this later on.
Using the cuts $U_{h_m}$, we define, for all $m \in \natz$, the extended real process $\martingale_{h_m}$ by
\begin{align*}
\martingale_{h_m}(t) \coloneqq
\begin{cases}
\martingale_m(t) &\text{ if } U_{h_m} \not\sqsubset t; \\
\martingale_m(u_{h_m}(t)) &\text{ otherwise,}
\end{cases}
\text{ for all } t\in\situations.
\end{align*}

Fix any $m \in \natz$.
Since $\martingale_m$ is bounded below, $\martingale_{h_m}$ is also bounded below.
Moreover, we have that
\begin{align*}
\martingale_{h_m}(t \cdot) \coloneqq
\begin{cases}
\martingale_m(t \cdot) &\text{ if } U_{h_m} \not\sqsubseteq t; \\
\martingale_m(u_{h_m}(t)) &\text{ otherwise,}
\end{cases}
\text{ for all } t\in\situations.
\end{align*}
Since, for any $t \in \situations$, $\lupprev{t}(\martingale_m(t \cdot)) \leq \martingale_m(t)$ --- because $\martingale_m$ is a supermartingale --- and $\lupprev{t}(c) = c$ for all $c \in \extreals{}_{\geq 0}$ because of \ref{coherence: bounds}, we have that $\lupprev{t}(\martingale_{h_m}(t \cdot)) \leq \martingale_{h_m}(t)$ for all $t \in \situations$.
Hence, $\martingale_{h_m} \in \setofextsupmartb{}$.
Furthermore, for all $\omega \in \samplespace{}$,
\begin{align*}
\lim_{n \to +\infty} \martingale_{h_m}(\omega^n) = \lim_{n \to +\infty} \martingale_m(u_{h_m}(\omega)) = \martingale_m(\omega^{h_m(\omega)}),
\end{align*}
where we used the fact that $u_{h_m}(\omega) = \omega^{h_m(\omega)}$. 
Therefore, by definition of $h_m$, we have that
\begin{align*}
	g_{h_m}(\omega) - \lim_{n \to +\infty} \martingale_{h_m}(\omega^n) 
	= g_{h_m(\omega)}^{\wedge m}(\omega) - \martingale_m(\omega^{h_m(\omega)}) \leq \sfrac{2}{3m} \text{ for all } \omega \in \Gamma(s).
\end{align*}
Hence, $g_{h_m} - \sfrac{2}{3m} \leq_s  \lim \martingale_{h_m}$, which implies, together with $\martingale_{h_m} \in \setofextsupmartb{}$ and \ref{vovk coherence 6}, that $\upprevvovkk{}(g_{h_m} \vert s) - \sfrac{2}{3m} \leq \martingale_{h_m}(s)$.
Since moreover $\martingale_{h_m}(s) = \martingale_m(s)$ because clearly $U_{h_m} \not\sqsubset s$, we have that $\upprevvovkk{}(g_{h_m} \vert s) - \sfrac{2}{3m} \leq \martingale_m(s) \leq \upprevvovkk{}(f^{\wedge m} \vert s) + \sfrac{1}{3m}$, so $\upprevvovkk{}(g_{h_m} \vert s) \leq \upprevvovkk{}(f^{\wedge m} \vert s) + \sfrac{1}{m}$.

This holds for every $m \in \natz$, so $\limsup_{m \to +\infty} \upprevvovkk{}(g_{h_m} \vert s) \leq \limsup_{m \to +\infty} \upprevvovkk{}(f^{\wedge m} \vert s)$, implying by \ref{vovk coherence 4} and the fact that $f^{\wedge m} \leq f$ for all $m \in \natz{}$, that $\limsup_{m \to +\infty} \upprevvovkk{}(g_{h_m} \vert s) \leq \upprevvovkk{}(f \vert s)$.
Hence, we have found a sequence $\{g_{h_m}\}_{m \in \natz}$ of $(\sup h_m)$-measurable gambles that converges point-wise to $f$ such that $\inf f = 0 \leq g_{h_m} \leq \sup f$ for all $m \in \natz{}$ and $\limsup_{m \to +\infty} \upprevvovkk{}(g_{h_m} \vert s) \leq \upprevvovkk{}(f \vert s)$.
The last step consists of transforming this sequence of $(\sup h_m)$-measurable gambles into a sequence of $n$-measurable gambles.
This can easily be done as follows.
Let $f_0 \coloneqq c$ for some arbitrary $0 \leq c \leq \sup f$.
Now let 
\begin{equation*}
f_n \coloneqq 
\begin{cases}
g_{h_m} &\text{ if } n = \sup h_m \text{ for some } m \in \natz; \\
f_{n-1} &\text{ otherwise, }
\end{cases}
\text{ for all } n\in\natz.
\end{equation*}
This definition is valid because, for any $n\in\natz$, there is only one $m \in \natz$ such that $n = \sup h_m$ because $\{h_m\}_{m \in \natz}$, and therefore also $\{\sup h_m\}_{m \in \natz}$, is increasing.
It also follows from this argument that $\{f_n\}_{n\in\natz}$ visits all the variables in $\{g_{h_m}\}_{m \in \natz}$ in the same order.
Hence, we have that $\lim_{n \to +\infty} f_n = \lim_{m \to +\infty} g_{h_m} = f$, that $\limsup_{n \to +\infty} \upprevvovkk{}(f_n \vert s) = \limsup_{m \to +\infty} \upprevvovkk{}(g_{h_m} \vert s) \leq \upprevvovkk{}(f \vert s)$, and that $\{f_n\}_{n\in\natz}$ is a sequence of gambles such that $0 \leq f_n \leq \sup f$ for all $n \in \natz{}$.
Let us moreover show by induction that it is a sequence of $n$-measurable gambles.
$f_0 = c$ is clearly $0$-measurable.
To prove the induction step, suppose that $f_n$ is $n$-measurable for some $n\in\natz$.
Then either $f_{n+1} = f_{n}$, which implies that $f_{n+1}$ is $n$-measurable and therefore automatically $(n+1)$-measurable.
Otherwise, we have that $n+1 = \sup h_m$ for some $m \in \natz$, implying that $f_{n+1} = g_{h_m}$ is $(\sup h_m)$-measurable and hence $(n+1)$-measurable.
This completes the induction step.
As a result, we have found a sequence of $n$-measurable gambles $\{f_n\}_{n\in\natz}$ that converges point-wise to $f$, for which $\inf f = 0 \leq f_n \leq \sup f$ for all $n \in \natz{}$ and $\limsup_{n \to +\infty} \upprevvovkk{}(f_n \vert s) \leq \upprevvovkk{}(f \vert s)$.
\end{proof}

\begin{proposition}\label{Prop: limits to general}
Consider any $s\in\situations$ and any $f\in\setofextvariables$. 
Then
\begin{align}
	\upprevvovkk (f \vert s) &=\inf{\Big\{ \upprevvovkk(g \vert s) \colon g\in\setoflimitsoffinmeasb{} \text{ and }  g \geq_s f  \Big\}}.
\end{align}
\end{proposition}
\begin{proof}
% Again, we prove this for $s=\Box$. 
% The proof is trivially extended towards any situation $s\in\situations$.
Because $\upprevvovkk$ is monotone [\ref{vovk coherence 4}], we have, for any $g\in\setoflimitsoffinmeasb{}$ such that $f \leq_s g$, that
\begin{align*}
\upprevvovkk(f \vert s)\leq\upprevvovkk(g \vert s).
\end{align*}
It therefore follows immediately that
\begin{align*}
\upprevvovkk (f \vert s) \leq\inf{\Big\{ \upprevvovkk(g \vert s) \colon g\in\setoflimitsoffinmeasb{} \text{ and }  g \geq_s f  \Big\}}.
\end{align*}
It remains to prove the converse inequality.

Consider any $\martingale\in\setofextsupmartb{}$ such that $\lim \martingale(\omega)$ exists for all $\omega\in\Gamma(s)$ and such that $\lim \martingale \geq_s f$.
Then we can also guarantee that $\lim \martingale$ exists everywhere by simply redefining $\martingale(t) \coloneqq\martingale(s)$ for all situations $t \not\sqsupseteq s$. 
Clearly, $\martingale$ then remains to be a supermartingale such that $\lim \martingale \geq_s f$.
Let $\{g_n\}_{n \in \natz{}}$ be the sequence defined by $g_n(\omega) \coloneqq \martingale{}(\omega^n)$ for all $n \in \natz{}$ and all $\omega \in \samplespace{}$.
Then it is clear that $\{g_n\}_{n \in \natz{}}$ is a sequence of $n$-measurable, and therefore finitary, extended real variables that is uniformly bounded below.
Moreover, we have that $\lim_{n \to +\infty} g_n(\omega) = \lim_{n \to +\infty} \martingale{}(\omega^n) \eqqcolon g(\omega)$ for all $\omega \in \samplespace{}$, implying that $g\in\setoflimitsoffinmeasb{}$ and, because $\lim \martingale \geq_s f$, that $g \geq_s f$.
It moreover follows from Definition \ref{def:upperexpectation2} that $\upprevvovkk(g \vert s)\leq\martingale(s)$ because $\lim \martingale \geq_s g$.
This implies that
\begin{align*}
\inf{\Big\{ \upprevvovkk(g \vert s) \colon g\in\setoflimitsoffinmeasb{} \text{ and }  g \geq_s f  \Big\}}\leq\martingale(s).
\end{align*}
Since this holds for any $\martingale\in\setofextsupmartb{}$ such that $\lim \martingale$ exists within $\Gamma(s)$ and $\lim \martingale \geq_s f$, it follows from Proposition~\ref{prop: liminf can be replaced by lim in definition} that
\begin{align*}
\inf{\Big\{ \upprevvovkk(g \vert s) \colon g\in\setoflimitsoffinmeasb{} \text{ and }  g \geq_s f  \Big\}}\leq\upprevvovkk (f \vert s).
\end{align*}
\end{proof}

\bibliographystyle{plain}
\bibliography{references}
\end{document}